%% file: Paper1_combined_.tex
\documentclass[11pt]{amsart}
 \usepackage[margin=1in]{geometry}
\usepackage{mathpazo} 
\usepackage{amsfonts}
\usepackage{stmaryrd}
\usepackage{amsmath}
\usepackage{mathrsfs}
\usepackage{amsthm}
\usepackage{hyperref}
\usepackage{yhmath}
\usepackage{tikz-cd}

\usepackage{color} %\textcolor{red}{Test}

\definecolor{darkred}{rgb}{1,0,0} %can change the intensity in [0,1]
\definecolor{darkgreen}{rgb}{0,0.8,0}
\definecolor{darkblue}{rgb}{0,0,1}
\hypersetup{colorlinks,
linkcolor=darkblue,
filecolor=darkgreen,
urlcolor=darkred,
citecolor=darkred}

\usepackage{mathtools}
\usepackage[all]{xy}

\usepackage{color}
\newcommand{\sbt}{\,\begin{picture}(-1,1)(-1,-1)\circle*{2}\end{picture}\ }
\renewcommand{\dot}[1]{\overset{\sbt}{#1}}

\newtheorem{lemma}{Lemma}[section]

\newtheorem{theorem}[lemma]{Theorem}
\newtheorem{corollary}[lemma]{Corollary}
\newtheorem{proposition}[lemma]{Proposition}

\theoremstyle{definition}
\newtheorem{definition}[lemma]{Definition}
\newtheorem{example}[lemma]{Example}
\newtheorem{remark}[lemma]{Remark}

\theoremstyle{plain}
\newtheorem{mainthm}{Theorem}

\numberwithin{equation}{section}

\newcommand{\deriv}{\left.\frac{\mathrm d}{\mathrm{dt}}\right\vert_{t=0}}
\newcommand{\J}{\mathcal J}

\newcommand{\Id}{\mathrm{Id}}
\newcommand{\cA}{\mathcal{A}}
\newcommand{\cC}{\mathcal{C}}

\newcommand{\cJ}{\mathcal{J}}
\newcommand{\cD}{\mathcal{D}}
\newcommand{\cG}{\mathcal{G}}
\newcommand{\cE}{\mathcal{E}}
\newcommand{\cK}{\mathcal{K}}

\newcommand{\cQ}{\mathcal{Q}}
\newcommand{\bfM}{\mathbf{M}}

\newcommand{\bfY}{\mathbf{Y}}
\newcommand{\bfZ}{\mathbf{Z}}
\newcommand{\bfT}{\mathbf{T}}
\newcommand{\cL}{\mathcal{L}}
\newcommand{\cH}{\mathcal{H}}
\newcommand{\cM}{\mathcal{M}}
\newcommand{\cS}{\mathcal{S}}
\newcommand{\bfW}{\mathbf{W}}
\newcommand{\bfX}{\mathbf{X}}
\newcommand{\bfB}{\mathbf{B}}
\newcommand{\bfC}{\mathbf{C}}

\newcommand{\T}{\mathrm{T}}
\newcommand{\V}{\mathrm{V}}
\newcommand{\rH}{\mathrm{H}}
\newcommand{\rE}{\mathrm{E}}
\newcommand{\rS}{\mathrm{S}}
\newcommand{\sK}{\mathsf{K}}
\newcommand{\sG}{\mathsf{G}}
\newcommand{\sH}{\mathsf{H}}
\newcommand{\sT}{\mathsf{T}}
\newcommand{\sC}{\mathsf{C}}
\newcommand{\sU}{\mathsf{U}}
\newcommand{\sMod}{\mathsf{Mod}}
\newcommand{\sGL}{\mathsf{GL}}
\newcommand{\sSL}{\mathsf{SL}}

\newcommand{\A}{\mathrm{A}}
\newcommand{\B}{\mathrm{B}}
\newcommand{\sSO}{\mathsf{SO}}
\newcommand{\cU}{\mathcal{U}}
\newcommand{\rmC}{\mathrm{C}}

\newcommand{\N}{\mathbb{N}}

\newcommand{\R}{\mathbb{R}}
\newcommand{\C}{\mathbb{C}}
\newcommand{\Z}{\mathbb{Z}}

\newcommand{\rIm}{\mathrm{Im}}

\newcommand{\Sig}{\Sigma}
\newcommand{\I}{\mathrm{I}}

\newcommand{\rd}{\mathrm{d}}

\input{macros}

\title{Higgs bundles, isomonodromic leaves and  minimal surfaces}
\author[Collier]{Brian Collier}
\address{Department of Mathematics,
University of California, Riverside,
   Riverside, CA 92521, USA}
    \email{brianc@ucr.edu}
\author[Toulisse]{Jeremy Toulisse}
\address{Universit\'e C\^{o}te d'Azur, CNRS, LJAD, France}
\email{jeremy.toulisse@univ-cotedazur.fr}

\author[Wentworth]{Richard Wentworth}
\address{Department of Mathematics,
   University of Maryland,
   College Park, MD 20742, USA}
\email{raw@umd.edu}

\begin{document}
\begin{abstract} 
In this paper we give a gauge theoretic construction of the joint moduli space of stable $\sG$-Higgs bundles on closed Riemann surfaces, where the Riemann surface structure is allowed to vary in the Teichm\"uller space of the underlying smooth surface. 
This joint moduli space has many interesting structures that
are preserved by the mapping class group of the surface. 
We describe a surprising relationship between four key objects: the isomonodromic foliation, a canonical hermitian form arising from the Atiyah-Bott-Goldman symplectic structure on the character variety, 
  a canonical holomorphic form which vertically lifts vector fields on Teichm\"uller space,  
     and the energy function for equivariant harmonic maps.  One
     consequence of this work is the construction of pseudo-K\"ahler
     metrics on many examples of  components of character varieties which include rank two higher Teichm\"uller spaces. This recovers some of the recent work on the subject by various authors.
\end{abstract}
\maketitle

\tableofcontents

\section{Introduction}

Let $\Sig$ be a connected closed oriented surface of genus $g\geq 2$ and $\sG$ 
a connected  complex semisimple Lie group. 
Associated to this data is a  holomorphic symplectic orbifold $\bfX(\sG),$
the moduli space of irreducible\footnote{By \emph{irreducible} we mean the
representation does not factor through a proper parabolic subgroup.} representations of the fundamental group of $\Sig$ into $\sG$. 
This moduli space is the orbifold locus of the  \emph{$\sG$-character variety of $\Sig$},
and the holomorphic symplectic form is known as the
\emph{Atiyah-Bott-Goldman form}. It is biholomorphic and symplectomorphic
to the moduli space of irreducible flat $\sG$-connections on $\Sig.$ 
One of the many interesting aspects of character varieties is that they carry a natural action of the mapping class group of the surface $\Sig$ which preserves the holomorphic symplectic structure.

For each choice of Riemann surface structure $X=(\Sig,j)$ on $\Sig$, the
moduli space $\bfM_X(\sG)$ of stable $\sG$-Higgs bundles on $X$ is another
interesting holomorphic symplectic orbifold.   For example, the Higgs
bundle moduli space contains the cotangent bundle of the moduli space of
holomorphic stable $\sG$-bundles as an open dense set, and its completion has the
structure of an algebraically complete integrable system
\cite{IntSystemFibration}.
Since its definition depends on the choice of a Riemann surface structure on $\Sig$, $\bfM_X(\sG)$ does not carry a natural action of the mapping class group.

The \emph{nonabelian Hodge correspondence} (proved by Hitchin \cite{selfduality}, Simpson \cite{SimpsonVHS}, Donaldson \cite{harmoicmetric} and Corlette \cite{canonicalmetrics}) provides an identification between $\bfM_X(\sG)$ and $\bfX(\sG)$. More precisely, the \emph{nonabelian Hodge map} 
\[\rH_X:\xymatrix{\bfM_X(\sG) \ar[r]^{\  \cong}& \bfX(\sG)} \]
is a real analytic isomorphism. Moreover, $\rH_X$ extends to a real analytic homeomorphism between the moduli space of polystable $\sG$-Higgs bundles and the $\sG$-character variety. In this paper, we will focus only on the stable loci of these spaces. 
Importantly, $\rH_X$ is neither holomorphic nor symplectic, but the nonisomorphic structures combine to define a hyperK\"ahler structure on $\bfM_X(\sG)$. In particular, pulling back the real part of Atiyah-Bott-Goldman form by $\rH_X$ defines a K\"ahler structure on $\bfM_X(\G)$.

In addition to this rich geometry, the theory of Higgs bundles has been an effective tool for studying various questions concerning character varieties. 
We briefly discuss two such applications relevant to this paper: parameterizing
\emph{higher Teichm\"uller spaces} and establishing (non) uniqueness of
equivariant minimal surfaces in noncompact Riemannian symmetric spaces. 

There are certain 
real forms $\sG^\R$ of $\sG$ for which the $\sG^\R$-character variety has
distinguished components. These components are usually called ``higher
Teichm\"uller spaces,'' as they generalize various features of the Teichm\"uller space of $\Sig$ when it is identified with the connected component of the $\mathsf{PSL}_2\R$-character variety consisting of holonomy representations of hyperbolic structures on $\Sig$. 
The first family of such components is the \emph{Hitchin components}, which exist for the split real form $\sG^\R$ of any semisimple complex Lie group $\sG$, e.g. $\sSL_n\R<\sSL_n\C$. They were first 
introduced  by Hitchin \cite{liegroupsteichmuller} who  parametrized them 
using Higgs bundles. The dynamical significance of these components was
discovered 
by Labourie \cite{AnosovFlowsLabourie} and, from a different perspective,
by Fock-Goncharov \cite{fock_goncharov_2006}.
A generalization of Hitchin's Higgs bundle parameterization was used in \cite{BCGGOgeneralCayley} to enumerate and parametrize all expected higher rank Teichm\"uller spaces predicted by Guichard-Wienhard's theory of $\theta$-positivity \cite{GWPos,GLWPositive,BGLPWcollar}. 
This built on and generalized previous work on higher rank Teichm\"uller spaces \cite{UpqHiggs,MaxRepsAnosov,BIWmaximalToledoAnnals,BGRmaximalToledo,SOpqComponents,CollierSOnn+1components}.

For a fixed representation $\rho\in\bfX(\sG)$, the corresponding Higgs
bundle on a Riemann surface $X$ is constructed from the existence of
an essentially unique $\rho$-equivariant harmonic map from the universal
cover $\widetilde X$ of $X$ to the Riemannian symmetric space $\G/\K$  of $\sG$. 
From Labourie's work on Anosov representations \cite{CrossRatioAnosovProperEnergy}, for representations in higher rank Teichm\"uller spaces
there exists a Riemann surface structure $X$ on $\Sig$ such that the
associated harmonic map is a conformal immersion. Equivalently, its image
is a $\rho(\pi_1\Sigma)$-invariant minimal surface in $\G/\K$.
%the symmetric space of $\sG$. 
The uniqueness of such a $\rho$-invariant minimal surface was proven for all representations in higher rank Teichm\"uller spaces when the real form has real rank $2$. There are two families: the case of Hitchin representations was proven by Labourie \cite{cyclicSurfacesRank2}, see also \cite{flatmetriccubicdiff,CubicDifferentialsRP2} for $\sSL_3\R$ and \cite{schoen} for $\sSO(2,2),$ and the case of so called  maximal representations into rank 2 hermitian Lie groups was proven by the first two authors with Tholozan \cite{CTTmaxreps}. 
For Hitchin representations in rank at least 3 (and hence many other higher rank
Teichm\"uller spaces of rank at least 3), Markovi\'c, Sagman and Smillie
\cite{SagmanSmillieLabConjarXiv,Markovic,MSS} recently showed that
uniqueness of the minimal surface fails. All of these results used a
detailed, and different, understanding of the associated Higgs bundles.

Both of the above applications involve understanding some aspect of how the Higgs bundle associated to a fixed representation depends on the choice of Riemann surface structure. It is natural to ask how various features of the nonabelian Hodge correspondence depend on the choice of Riemann surface. For example:
other than components for real forms of $\sG,$ are there interesting subvarieties of $\bfM_X(\sG)$ whose image under the nonabelian Hodge map $\rH_X$ is the same for all Riemann surfaces $X?$ 
For a given representation $\rho\in\bfX(\sG)$, is the set of
$\rho$-invariant (branched) minimal surfaces in  $\G/\K$
%the symmetric space of $\sG$ 
discrete or not?

A natural setting for questions of this nature would be to have a
\emph{joint moduli space of $\sG$-Higgs bundles}, where the underlying
Riemann surface is allowed to vary in the Teichm\"uller space $\bfT(\Sig)$
of $\Sig$. In fact, Simpson gave an algebro-geometric construction of a related
space, where the Higgs bundle moduli space is constructed in the relative
case for a family of smooth projective curves over a scheme
of finite type over $\C$ \cite{SimpsonModuli1}. In this context, he proves
that the nonabelian Hodge map is a homeomorphism with the relative
character variety \cite[Theorem 7.18]{Simpson:94b}.  With a small amount of
work, Simpson's construction  produces a complex analytic space that may be
regarded as  a moduli space over $\bfT(\Sig)$ (see \cite[\S 7]{SO23LabourieConj}). 
However, many questions of interest are inaccessible by
algebraic methods.
For example, metric structures, and even the higher regularity of the nonabelian Hodge map in the
joint setting is not transparent.

The starting point for this paper is therefore to give a gauge theoretic construction of the 
joint moduli space of stable $\sG$-Higgs bundles and to
study its mapping class group invariant structure.  
The first result is the following.

\begin{mainthm}\label{MainThmA}
	Let $\Sig$ be a closed oriented surface
    of genus $g\geq 2$ and $\sG$ be a complex semisimple Lie group. Then there is a complex orbifold $\bfM(\sG)$ such that: 
	\begin{enumerate}
		\item $\bfM(\sG)$ fibers holomorphically over the Teichm\"uller space $\bfT(\Sig)$  of complex structures on $\Sig$
	\[\pi:\bfM(\sG)\lra\bfT(\Sig),\] 
	where each fiber $\pi^{-1}(X)$ is biholomorphic to the moduli space of stable $\sG$-Higgs bundles on the Riemann surface $X.$
	\item  The mapping class group of $\Sig$ acts holomorphically on
        $\bfM(\sG)$ covering the standard action on $\bfT(\Sig).$
	\item $\bfM(\sG)$ is equipped with a mapping class group invariant
        closed 2-form $\omega_0$ which is compatible with the complex
            structure and which  restricts to the standard K\"ahler form on the fibers $\pi^{-1}(X)$.
    \item On $\bfM(\sG)$ there is a nonzero holomorphic section $\Theta$ of 
        $\Hom(\pi^\ast(\T\Tbold(\Sigma)), \V\Mbold(\G))$
        that is invariant under the action of the mapping class group.
            Here, $\V\Mbold(\G)$ denotes the vertical holomorphic tangent
            bundle of $\pi$.
	\end{enumerate}   
\end{mainthm}

\begin{remark}
	Quotienting by the action of the mapping class group $\sMod(\Sig)$ of
    $\Sig$, gives a holomorphic fibration (in the orbifold sense) over the
    moduli space of genus $g$ curves $\cM_g$:
    \[\bfM(\sG)/ \sMod(\Sig)\lra \cM_g.\]
\end{remark}

We will discuss the technical aspects of the construction of $\bfM(\sG)$
and its important properties in \S  \ref{sec intro techncal}. 
For now, let us mention the following key points: 
the complex orbifold $\bfM(\sG)$ is isomorphic to
the analytification of the coarse moduli space constructed by 
Simpson using algebraic methods, and, as expected,
the nonabelian Hodge map for a fixed Riemann surface extends to a real
analytic  map (see Theorem \ref{thm:realanal}) 
\[\rH:\bfM(\sG)\lra \bfX(\sG)~.\]
The closed 2-form $\omega_0$ is defined by pulling back the  real part of the Atiyah-Bott-Goldman form on $\bfX(\sG)$ by  the  map $\rH$
and taking the part compatible with the complex structure on $\bfM(\sG)$.
Note that, in contrast to the Weil-Petersson symplectic form on Teichm\"uller space, the definition of $\omega_0$ does not require choosing a conformal metric on $\Sig.$

\subsection{Isomonodromic foliation and energy} 
We now describe a surprising interplay between four different objects
defined on the joint moduli space $\bfM(\sG)$ constructed in Theorem
\ref{MainThmA}: the isomonodromic foliation, the hermitian form associated
to $\omega_0$, the holomorphic section $\Theta$, and the energy function. 

Given $\rho \in \bfX(\sG)$, the level set $\cL_\rho:= \rH^{-1}(\rho)$ is called an \emph{isomonodromic leaf}. The isomonodromic leaves fit together to define the \emph{isomonodromic foliation}. Its tangent bundle is called the \emph{isomonodromic distribution} and will be denoted by $\cD$. 

Denote the complex structure on $\bfM(\sG)$ by $\I$ and the hermitian form associated to $\omega_0$ by
\[h_0(\cdot,\cdot)=2(\omega_0(\I\cdot,\cdot)+i\omega_0(\cdot,\cdot)).\]
As above, let $\V\bfM(\sG)$ be the vertical tangent bundle of the fibration
$\pi:\bfM(\sG)\to\bfT(\Sig)$. The $h_0$-perpendicular space will be denoted
by $\cH$ and referred to as the \emph{horizontal distribution}. 
Since  $h_0$ is positive definite on $\V\bfM$, we have an $h_0$-orthogonal splitting
\[\T\bfM(\sG)=\V\bfM(\sG)\oplus \cH.\] 
It turns out that the hermitian form $h_0$ (equivalently, the closed
$2$-form $\omega_0$) is
\emph{not}
nondegenerate everywhere. In fact, the kernel of $h_0$ is intimately
related to the complex tangencies of the 
isomonodromic distribution $\cD$ and the holomorphic section $\Theta$, as
the next result illustrates.
\begin{mainthm}\label{MainThmB}
Let $(h_0, \I)$ denote the hermitian form and complex structure on
    $\bfM(\sG)$. Then $h_0$ is nonpositive on the horizontal distribution $\cH.$ 
In particular, the kernel $\cK_x$ of $h_0$ at $x\in\bfM(\sG)$     is contained in $\cH_x$ and, when $h_0$ is nondegenerate, it has 
    signature $(\dim_\CBbb\bfX(\sG),3g-3)$. 
Moreover,
\[\cK_x= \cD_x\cap \I(\cD_x)=\ker(\Theta_x),\]
where $\Theta$ is the holomorphic section from Theorem \ref{MainThmA}, 
    and in the last equality we have identified the horizontal distribution $\cH$ with $\pi^*\T\bfT(\Sig)$.  
\end{mainthm}

% We now define the energy function. 
As mentioned above, for a fixed 
representation $\rho\in\bfX(\sG)$, the associated Higgs bundle on a
Riemann surface $X$ is constructed from the unique $\rho$-equivariant
harmonic map $u_\rho:\widetilde X\to\G/\K$.
%from the universal cover of $X$ to the Riemannian symmetric space of $\sG.$ 
Taking the Dirichlet energy of $u_\rho$ on a fundamental domain defines a
smooth function $\rE:\bfM(\sG)\to \R$. Restricting $\rE$ to the
isomonodromic leaf of $\rho$ defines a function on Teichm\"uller space
called the \emph{energy function of $\rho$},
\[\cE_\rho:\bfT(\Sig)\lra \R\ .\]
The following theorem relates the complex Hessian of $\cE_\rho$,  the hermitian form $h_0$ on the horizontal distribution $\cH$ and the holomorphic section $\Theta$. 

\begin{mainthm}\label{MainThmC}
For $x\in\bfM(\sG)$, let $X=\pi(x)$ be the associated Riemann surface, and let $\rho=\rH(x)$ be the associated representation. 
    For each tangent vector $[\mu]\in \T_X\bfT(\Sig)$, let $w_\mu\in \cH_x$ be the unique horizontal lift.
    Then the complex Hessian of $\cE_\rho$  at $X$ in the direction $[\mu]$ is given by 
    \begin{equation}\label{eqn:hessian}
        \Delta_\mu \cE_\rho=-8\Vert
        w_\mu\Vert^2_{h_0}=8\Vert\Theta_x([\mu])\Vert_{h_0}^2\ . 
    \end{equation}
%    where $\Theta_x(\mu)\in\V\bfM(\sG)$ is defined by Item (4) of Theorem \ref{MainThmA}. 
\end{mainthm}

\begin{remark}
    Since $\Vert\Theta([\mu])\Vert_{h_0}^2\geq 0$, Theorem \ref{MainThmC} implies that the energy function $\cE_\rho$ is plurisubharmonic. This  recovers special cases of results of Toledo \cite{Toledo}.
\end{remark}
The kernel of the complex Hessian of $\cE_\rho$ corresponds to the
directions in which $\cE_\rho$ is not strictly plurisubharmonic, and it is identified with a subspace $\cQ_x\subset\cH_x$ of the horizontal distribution by Theorem \ref{MainThmC}. The following corollary is immediate from Theorems B and C.
\begin{corollary} \label{cor:theta}
	Let $x\in\bfM(\sG)$. Using the notation from Theorem \ref{MainThmB}, we have 
	\[\cQ_x=\cK_x = \cD_x \cap\I(\cD_x) =  \ker(\Theta_x).\]
\end{corollary}

\begin{remark}
	In \cite{Tosicpluri}, To\v si\'c also used Higgs bundles to study the
    complex Hessian of the energy function $\cE_\rho$ for the group
    $\sSL_n\C$ along $1$-parameter families of deformations of $X$.
   % Since he does not have a joint moduli space to work in To\v si\'c works with 1-parameter families.
    In particular, he describes the directions in which the complex Hessian
    of $\cE_\rho$ vanishes, 
    the space $\cQ_x$ above, by certain equations involving the Beltrami differential $\mu$ and the Higgs field $\Phi$.
	These equations are given in \eqref{eq :Tosic} and also play a fundamental role in our analysis since their solutions are equivalent to being in the kernel of the holomorphic section $\Theta$.
	While we do not use To\v si\'c's work to establish the above,
    Theorem \ref{MainThmC} and other results in \S \ref{sec:iso hor and energy} 
     provide a new, 
     more geometric perspective on many of the statements in \cite{Tosicpluri}.  For completeness, in Appendix \ref{sec:tosic-comparison} we indicate why
     Equation
    \eqref{eqn:hessian}  is the same as the formula in  \cite[Theorem 1.10]{Tosicpluri}. 
\end{remark}

One consequence of the approach taken here
is that the rank of the kernel of $h_0$ defines a mapping class group invariant stratification
\begin{equation}
	\label{eq intro strat} 
	\bfM(\sG)=\coprod_{0\leq d\leq 3g-3}\bfM_{d} ~,
\end{equation}
where $\bfM_{d}=\{x\in \bfM(\sG)~|~\dim(\cK_x)=d\}$. 
From the relation with $\Theta$ in Corollary \ref{cor:theta}, it follows
that  
each stratum $\bfM_d$ is a holomorphic subvariety of $\Mbold(\G)$,  and
in addition they are preserved by the 
natural holomorphic $\C^*$-action defined by multiplying the Higgs field by
an element of $\C^*$.
Moreover, $\bfM_0$ is open and dense, while $\bfM_{3g-3}$ is closed. 

Theorems B and C give the following characterization of the open stratum $\bfM_0$.

\begin{corollary}\label{Cor open stratum}
For $x\in\bfM(\sG)$, let $X=\pi(x)$ and $\rho=\rH(x)$ be the associated
    Riemann surface and representation, respectively. Then the following
    are equivalent:
\begin{enumerate}
	\item $x$ is contained in the open stratum $\bfM_0$,
	\item the hermitian form $h_0$ is nondegenerate at $x$ with signature $(\dim\bfX(\sG),3g-3)$, 
    \item the isomonodromic distribution $\cD_x$ is totally real, 
        \item the restriction of $\omega_0$ to the isomonodromic
            distribution is nondegenerate,
	\item the energy function $\cE_\rho$ of $\rho$ is strictly plurisubharmonic in a neighborhood of $X$.
\end{enumerate}
\end{corollary}
At the other extreme, we  have the following characterization of the closed
stratum $\bfM_{3g-3}$.

\begin{corollary}\label{cor closed stratum}
    With the same notation as in Corollary \ref{Cor open stratum}, the following are equivalent:
 	\begin{enumerate}
 		\item $x$ is contained in the closed stratum $\bfM_{3g-3}$,
		\item the isomonodromic distribution $\cD_x$ and the kernel $\cK_x$ of $h_0$ are equal,
		\item $\cD_x$ is a complex subspace of $\T_x\bfM(\sG)$,
		 \item the restriction of $\omega_0$ to the isomonodromic
             distribution vanishes,
		\item the complex Hessian of the energy function $\cE_\rho$ vanishes in all directions at $X$.
 	\end{enumerate}
 \end{corollary}
 
Apart from compact representations, 
it would seem very hard to characterize when a general isomonodromic leaf is
entirely contained in a fix stratum.
However, we are able to prove that the isomonodromic leaves in the
so-called Cayley components of \cite{BCGGOgeneralCayley} are contained in
the open stratum $\bfM_0$. Let us first briefly describe the relevant objects.

Guichard-Wienhard \cite{GWPos} classified four families of real forms of
complex simple Lie groups which admit a  $\theta$-positive structure: split real forms, hermitian real forms of tube type, groups locally isomorphic to $\sSO_{p,q}$ with $1<p<q$, and the 
quaternionic real forms of type $\mathsf{E_6}, \mathsf{E_7},\ \mathsf{E_8},$ and $\mathsf{F}_4$. 
For each such real form $\sG^\R$, there is an special class of representations $\rho:\pi_1\Sig\to\sG^\R$ called \emph{$\theta$-positive representations}. 
These representations have a number of interesting geometric and dynamical properties.
Most notably, $\theta$-positive representations define higher rank Teichm\"uller spaces  \cite{BGLPWcollar,GLWPositive}, which, by definition, means they are a union of connected components of the $\sG^\R$-character variety $\bfX(\sG^\R)$ which consist entirely of discrete and faithful representations. 
For split and hermitian families, $\theta$-positive representations coincide with the more well known classes of Hitchin and maximal representations, respectively. 

Reference \cite{BCGGOgeneralCayley} gave a Higgs bundle parameterization of special components, called \emph{Cayley components}, of the $\sG^\R$-character variety for each of
the four families of $\theta$-positive structures. 
It is expected that the set of $\theta$-positive 
representations coincides with the Cayley components.
Indeed, the Cayley components are included in the set of $\theta$-positive representations by \cite{BCGGOgeneralCayley} and \cite{BGLPWcollar}. For Hitchin and maximal representations the equality holds by construction, and for the third family equality follows from the component classification of \cite{SOpqComponents}.

\begin{mainthm}\label{MainThmD}
Let $\rho\in\bfX(\sG)$ be such that it defines a $\theta$-positive representation into some real form $\sG^\R$ of $\sG$ which is in a Cayley component, and let $\cL_\rho$ be its isomonodromic leaf. Then,
\begin{enumerate}
	 	\item $\cL_\rho$ is contained in the open stratum $\bfM_0,$
	 	\item $\cL_\rho$ is a symplectic submanifold of $\bfM_0$ which is totally real and $\omega_0(\I\cdot,\cdot)$-isotropic, and 
	 	\item $h_0$ has signature $(3g-3,3g-3)$ on $\T\cL_\rho\oplus\I(\T\cL_\rho)$.
	 \end{enumerate} 
\end{mainthm}

The following corollary recovers results of Slegers \cite{Slegers} for Hitchin representations in $\sSL_n\R$; it is immediate from Theorem \ref{MainThmD} and Corollary \ref{Cor open stratum}.
\begin{corollary}
	Suppose $\rho$ is a $\theta$-positive representation in a Cayley component. Then the  energy function $\cE_\rho$ is strictly plurisubharmonic. 
\end{corollary}

Our next theorem gives equivalent characterizations of isomonodromic leaves which are contained in the closed stratum. 

\begin{mainthm}\label{MainThmE}
 	For a representation $\rho\in\bfX(\sG)$ with isomonodromic leaf
    $\cL_\rho$, the following are equivalent: 
\begin{enumerate}
	\item $\cL_\rho$ is contained in the closed stratum $\bfM_{3g-3}$,
	\item $\cL_\rho$ is a holomorphic submanifold of $\bfM(\sG)$,
    \item the energy function $\cE_\rho$ of $\rho$ is constant on
        $\Tbold(\Sigma)$. 
\end{enumerate}
 \end{mainthm}
 A representation $\rho:\pi_1\Sigma\to\G$ is called \emph{totally elliptic}
 if the action of $\rho(\gamma)$ on $\G/\K$
% the symmetric space of $\sG$ 
 has zero translation length for each essential simple closed curve $\gamma\subset\Sigma$ (see
    \cite{Maret:25}). 
 By a classical argument of Schoen-Yau \cite[Lemma 3.1]{SchoenYau:79} and
 Sacks-Uhlenbeck \cite[Theorem 4.3]{SacksUhlenbeck:82},  we have the
 following immediate consequence of Theorem \ref{MainThmE}.
\begin{corollary} \label{cor:elliptic}
	If $\rho\in\bfX(\sG)$ is a representation such that the isomonodromic
    leaf $\cL_\rho\subset \bfM(\sG)$ is holomorphic, then
    $\rho$ is totally elliptic.
\end{corollary}
\begin{remark}
	Observe that in the examples above, the dynamical behaviors of
    the representations with isomonodromic leaves in $\bfM_0$ or $\bfM_{3g-3}$ are
    opposite: the first are Anosov (in particular, they  are
    quasi-isometric embeddings), while the second behave similarly to compact representations.
    On the other hand, for $\G=\SL_2(\CBbb)$, it follows from
    \cite[Theorem 1]{WentworthWolf:16} that for every
    nonelementary representation $\rho$, the leaf $\Lcal_\rho$   lies in $\Mbold_0$
    over a nonempty open set in $\Tbold(\Sigma)$.
  It would be interesting
to find a dynamical characterization of which strata $\Mbold_d$ an
    isomonodromic leaf $\Lcal_\rho$ can intersect.
\end{remark}

Let us describe four classes of representations with holomorphic isomonodromic leaves:

\begin{itemize}
	\item For representations valued in a compact subgroup of $\G$, the harmonic map is constant, and so has zero energy. In this case, our theorem recovers results of Biswas \cite{Biswas}.
	% Such representations have a fixed point in the symmetric space, the corresponding harmonic map is constant, and so have zero energy. Our theorem implies the isomonodromic leaf is holomorphic.
	\item The second class was explained to us by Pierre Godfard. Given a
        modular functor, one obtains a finite set of projective representations $\{r_i\}_{i\in I}$ of the mapping
        class group $\sMod(\Sig_{g,1})$ of the punctured surface $\Sigma_{g,1}$,
        and so representations $\{\rho_i\}_{i\in I}$ of $\pi_1(\Sig_g)$ by restriction of
        $r_i$ in the Birman exact sequence. Godfard proves in
        \cite{GodfardConformalBlocks} that each $r_i$ yields a complex variation
        of Hodge structure on the moduli space $\cM_{g,1}$. Pulling back to
        to the fibers of the universal curve over $\Tbold(\Sigma)$, 
        one sees that $\cL_{\rho_i}$ is fixed pointwise by the $\C^*$-action,
        and hence is holomorphic.
	Note that applying this construction to the  Witten-Reshetihin-Turaev modular functors implies the representations considered by Koberda-Santharoubane in \cite{KoberdaSantharoubaneTQFT} have holomorphic isomonodromic leaves (see also \cite{BiswasHellerHeller}).
	\item There is a class of examples arises for punctured spheres, and so  does not fit in 
        the construction presented in this paper.
        In \cite{SupraMax}, Deroin-Tholozan discovered compact components
        of the relative character variety in $\mathsf{PSL}_2\R$, and the
        underlying representations are monodromies of $\C \mathrm{VHS}$s
        for every choice of Riemann surface structure on the punctured sphere. Analogous construction exist in higher rank by Tholozan and the second author \cite{TTcompactcomp}, Feng-Zhang \cite{ZhangCompactSO2n}, and Wu \cite{wuminenergylocalsystems}.
    \item As pointed out to us by Daniel Litt, there is a set of examples coming from
        Kodaira-Par\v{s}in families, see Example 3.1.10   in
        \cite{Litt:24}.
\end{itemize}

\subsection{Minimal surfaces} We now discuss applications of our construction to minimal surfaces in the symmetric space of $\sG$ which are preserved by an action of the fundamental group of $\Sig.$ 

%As mentioned above, given a Riemann surface $X$ and an element $\rho$ in
%$\bfX(\sG)$, the nonabelian Hodge map is obtained from an equivariant
%harmonic map $u_\rho$ from the universal cover of $X$ to the symmetric
%space of $\sG$. 
Associated to the $\rho$-equivariant harmonic map $u_\rho:\widetilde X\to
\G/\K$ there is a holomorphic quadratic
differential on $X$ called the \emph{Hopf differential} of $u_\rho$. 
The Hopf differential of the corresponding harmonic map defines a
holomorphic map from $\bfM_X(\sG)$ to vector space of holomorphic quadratic
differentials on $X$. The preimage of zero is a complex subvariety which we denote
by $\bfW_X(\sG)$. The corresponding harmonic maps are called
\emph{conformal harmonic}. 
 By \cite[Theorem 1.8]{SacksUhlenbeck:81},  branched equivariant  minimal
 immersions (in the sense of \cite{GOR:73})  $\widetilde X\to \G/\K$  are exactly 
the conformal equivariant harmonic maps. Hence, $\bfW_X(\sG)$ is the space of 
equivariant branched minimal surfaces with induced conformal structure $X$. 

The above picture easily generalizes to the joint moduli space. There is a
holomorphic map
\[\varpi^{(2)}:\bfM(\sG)\lra \T^*\bfT(\Sigma),\]
such that preimage of the zero section is a complex subvariety $\bfW(\sG)$
corresponding to equivariant branched minimal immersions.
%from the universal cover of $\Sig$ to the symmetric space of $\sG$.
We call $\bfW(\sG)$ the \emph{moduli space of equivariant minimal surfaces}. 
The space $\bfW(\sG)$ has the same dimension as $\bfX(\sG)$, but is not
smooth in general. However, the space 
\begin{equation}
	\label{eq smooth minimal surf}\bfW_0=\bfW(\sG)\cap \bfM_0,
\end{equation}
is a smooth complex submanifold of $\bfM(\sG)$.  This follows from the following proposition characterizing the strata $\bfM_d$ from \eqref{eq intro strat}.
\begin{proposition}[see Proposition \ref{prop:dimCoker} below]
	A point $x\in\bfM(\sG)$ lies in the stratum $\bfM_d$ if and only if 
	\[d=\dim\coker(\rd_x\varpi^{(2)}).\]	
\end{proposition}

Recall that the closed 2-form $\omega_0$ is defined by pulling back the Atiyah-Bott-Goldman form on $\bfX(\sG)$ by the nonabelian Hodge map $\rH$ and taking the $\I$-invariant part. 
It turns out that on the space of minimal surfaces $\bfW(\sG)$, the
pullback of the Atiyah-Bott-Goldman form is already $\I$-invariant.  From
this we obtain the following. 

\begin{mainthm}\label{MainThmF}
	On the space of minimal surfaces $\bfW_0$ defined in \eqref{eq smooth minimal surf}, the 2-form $\omega_0$ is the pullback of the Atiyah-Bott-Goldman symplectic form by the nonabelian Hodge map 
	\begin{equation} \label{eqn:mini-NAH}
	\rH:\bfW_0\lra\bfX(\sG).
    \end{equation} 
    In particular, $\rH$ is a symplectic immersion at $x\in\bfW_0$ if and only if $h_0|_{\bfW_0}$ is nondegenerate at $x$.
\end{mainthm}

Note that
since the hermitian form $h_0$ is indefinite, nondegeneracy does not
automatically pass to complex submanifolds.
However, we show that nondegeneracy of $h_0$ on $\bfM(\sG)$ does imply
nondegeneracy on $\bfW(\sG)$ for points which are fixed by the action of
the subgroup of $k^{th}$-roots of unity in $\CBbb^\ast$ for $k\geq 3$. 
Such fixed points are usually called \emph{$k$-cyclic Higgs bundles}.

\begin{mainthm}\label{MainThmG} Let $k\geq3$ and $x\in\bfM_0$ be a $k$-cyclic Higgs bundle in the open stratum. Then $x\in\bfW_0$  and $h_0$ is nondegenerate on $\T_x\bfW_0$. In particular,
the   nonabelian Hodge map
   \eqref{eqn:mini-NAH}
is a symplectic immersion at $x$. 
\end{mainthm}

Since Baraglia's influential thesis \cite{g2geometry}, cyclic Higgs bundles have appeared in many different setting that are listed in \S \ref{ss:CyclicHB}. Theorem \ref{MainThmG} applies to all of these cases. Since the space of $k$-cyclic Higgs bundles is a complex submanifold of $\bfM(\G)$, we conclude the following.

\begin{corollary}[Theorem \ref{thm:FamilyOfCyclicHB}]\label{cor:intro cyclic}
Let $k\geq 3$ and $\bfZ$ be the space of $k$-cyclic Higgs bundles in the
    Hitchin component for the split real form of $\G$. Then the nonabelian
    Hodge map restricts to a symplectic immersion on $\bfZ$. In particular,
    $h_0$ restricts to a pseudo-K\"ahler metric on $\bfZ$.
\end{corollary}
 
Corollary \ref{cor:intro cyclic}
strengthens a result of  Labourie \cite{cyclicSurfacesRank2} in the case
when $k-1$ is the length of the longest root in $\gfrak$, see also Remark
\ref{rem:ES cyclic}. Namely, for such a $k$, Labourie proved that the
nonabelian Hodge map is an immersion. For this case, the signature
of the metric is $((2k-1)(g-1) ,3g-3)$.

For any $\rho \in \bfX(\G)$ corresponding either to a maximal
representation in $\SO_{2,n}$ or to a Hitchin representation for $\rank(\G)=2$, a Higgs bundle $x\in \rH^{-1}(\rho)$ is in $\bfW_0$ if and only if it is $4$-cyclic.  Applying Theorem \ref{MainThmG} to these cases equips these spaces with a pseudo-K\"ahler metric.

\begin{mainthm}\label{MainThmH}
    Let $\bfY\subset \bfX(\sG)$ be \emph{either} the submanifold of
    Hitchin representations in case $\mathrm{rk}(\sG)=2$, \emph{or}
    the submanifold of maximal $\sSO_{2,n}$-representations in case
    $\G=\SO_{n+2}(\CBbb)$. 
    Let $\bfW(\sG)\subset\bfM(\sG)$ be the space of equivariant minimal surfaces, $\rH$ be the nonabelian Hodge map, and let $\bfZ=\rH^{-1}(\bfY)\cap \bfW(\sG)$. Then, 
\begin{enumerate}
	\item $\omega_0$ is nondegenerate on $\bfZ$, and
	\item $\rH:\bfZ\to\bfY$ is a symplectomorphism.
\end{enumerate} 
 In particular, $h_0$ defines a mapping class group invariant
    pseudo-K\"ahler metric on $\bfY$, compatible with the
    Atiyah-Bott-Goldman symplectic form, and of signature $\left((\dim\sG-3)(g-1),3g-3\right)$.
\end{mainthm}

\begin{remark}\label{rem:ES cyclic}
Theorem \ref{MainThmH}, for Hitchin representations, and Corollary \ref{cor:intro cyclic}, in the case $k-1$ is the length of the longest root in $\gfrak$, were recently obtained with different methods by El Emam-Sagman in
 \cite[Thm.\ A and A']{EmamSagman:25}.
\end{remark}

\begin{remark}
Theorem \ref{MainThmH} provides a uniform proof of all cases where Labourie's conjecture \cite{CrossRatioAnosovProperEnergy} holds. As mentioned above, there exist Hitchin and maximal representations in groups of rank at least 3 admitting various equivariant minimal surfaces \cite{Markovic,MSS,SagmanSmillieLabConjarXiv}.
\end{remark}

The study of pseudo-K\"ahler structures on moduli spaces of representations was initiated by Mazzoli-Seppi-Tamburelli in \cite{MST}: adapting Donaldson's construction of the Teichm\"uller space,  they construct a para-hyperK\"ahler structure on the Hitchin component for $\SO_{2,2}$. In a similar spirit, Tamburelli-Rungi constructed in \cite{TamburelliRungi} a pseudo-K\"ahler structure on a neighborhood of the Fuchsian locus in the $\SL_3\R$-Hitchin component. Both constructions rely on choosing an adapted metric on the underlying Riemann surface; it is not clear how their structures relate with ours.

\subsection{Technical aspects of $\bfM(\sG)$ and its structures}\label{sec intro techncal}
We now describe more precisely the construction of $\bfM(\G)$. 
The method used in this paper is a combination of the work of  Ebin and  Tromba  on the
differential geometric description of 
Teichm\"uller space (see \cite{Ebin:68,Tromba}),
and the classic construction of Kuranishi slices in gauge theory
(see \cite{AHS:78,FreedUhlenbeck} for  early treatments of this method, 
as well as Hitchin's original construction in
\cite{selfduality} for Higgs bundles).
We note that alternative formulations 
can be found in Earle-Eells \cite{EarleEells:69} in the case of
Teichm\"uller space, and more  recently in the
work of Diez-Rudolph \cite{DiezRudolph:24} for Teichm\"uller space as well
as Yang-Mills connections.

For a fixed principal $\sG$-bundle $P$, 
we consider a \emph{configuration space} $\cC(P)$ that
is an infinite dimensional complex Fr\'echet manifold. Stable $\G$-Higgs
bundles correspond to points $x\in \cC(P)$ solution to an equation
$F(x)=0$, where $F$ is a holomorphic map on $\cC(P)$ valued in some complex
vector space. This realizes $\bfM(\G)$ as a topological
quotient $F^{-1}(0)/\Aut_0(P)$
where $\Aut_0(P)$ is an infinite dimensional Lie group (see
\eqref{eqn:gauge-group}).
A manifold structure is then obtained by proving
a \emph{slice theorem};
namely,  one can locally find complex
finite dimensional submanifolds $\cS$ of $F^{-1}(0)$ such that
the action map $\Aut_0(P)\times \cS \to F^{-1}(0)$ is open
and a diffeomorphism onto its image (moduli the center of $\G$).
The local slices then patch together to 
form a holomorphic atlas on the quotient $\bfM(\G)$.

The local structure of the map $F$ and of the action of $\Aut_0(P)$ is encoded in a \emph{deformation complex} $(\B^\bullet,\delta_\B^\bullet)$ which turns out to be elliptic. Given a Riemannian metric on $\cC(P)$, one can then define harmonic representative in the cohomology $H^1(\B^\bullet)$ of the complex and construct the slice using the implicit function theorem. If the metric is $\Aut_0(P)$-invariant and compatible with the complex structure, one can hope to obtain a K\"ahler structure on $\bfM(\G)$.

Surprisingly, finding such a natural Riemannian metric on the configuration space is tricky. Our approach consists in defining a map
\[\widehat\rH: \cC(P) \longrightarrow \J(\Sig)\times \cA(P)~,\]
where $\J(\Sig)$ is the space of complex structures on $\Sig$ and $\cA(P)$
the space of connection on the underlying smooth principal $\G$-bundle. For any $s\in\R$, we then have a closed two-form
$$\widehat \omega_s = \widehat\rH^* (s\cdotp \omega_{WP}\oplus \omega_{ABG})~,$$
where $\omega_{WP}$ is the Weil-Petersson form on $\J(\Sigma)$ and $\omega_{ABG}$ is the real part of the and Atiyah-Bott-Goldman form on $\cA(P)$.

Let $\omega_s$ denote the $(1,1)$-part of $\widehat\omega_s$. For any $s>0$,
the hermitian form $h_s$ associated to $\omega_s$ fails to be positive
definite on the entire configuration space. Nevertheless, for any $x\in
\cC(P)$, one can find  $s$ large enough such that $h_s$ is positive
definite on $\T_x\cC(P)$. Remarkably, the harmonic representatives of
$H^1(\B^\bullet)$ using $h_s$ are independent of $s$ (see Proposition
\ref{prop:harmonics}) and each $\omega_s$ is closed in the direction of the
harmonics (see Corollary \ref{c:Closed}). As a result, all the structure
descends to the quotient $\bfM(\G)$ where we get a family
$(\omega_s)_{s\in\mathbb R_+}$ of closed 2-forms compatible with the complex structure and such that
\[U_s= \{x\in \bfM(\G)~\vert~(h_s)_x>0\}~,\]
defines an increasing exhaustion of $\bfM(\G)$ by open K\"ahler manifolds which
contain the nilpotent cone for sufficiently large $s$ (see Theorem
\ref{thm:nilpotent-kahler}).

For $s=0$, we no longer have  a notion of harmonic representatives
associated to $h_0$. Rather,
classes in $H^1(\B^\bullet)$ are represented by an infinite dimensional
space of representatives we term \emph{semiharmonic}. It turns out,
however,  that the
$h_0$ norm is independent of the choice of semiharmonic representative
of a given class
(Proposition \ref{prop:semiharmonic}). Thus,  $h_0$ gives a well-defined
hermitian form on  $H^1(\B^\bullet)$,  which, as discussed above, 
is  nondegenerate at points in $\Mbold_0$. 

In addition to the aforementioned work of Simpson, 
 let us note here several variants
 used by various authors to describe joint moduli spaces. 
 For instance, the moment map point of view developed by Donaldson in
 \cite{DonaldsonMoment1,DonaldsonMoment2} was used by Trautwein
 \cite{Trautwein} in his thesis to obtain partial results (see also
 \cite{GarciaFernandezThesis,ACGFGP:13,MST,TamburelliRungi} for related constructions).
 The curvature of connections appears in the realization of the deformation
theory of pairs  $(X,\Ecal)$, $\Ecal\to X$ a holomorphic bundle,  as a differential graded Lie algebra 
(see Huang \cite{huang}). In a similar way, 
Ono \cite{ono,Ono:24} produces local Kuranishi models for joint
deformations of Higgs bundles (in higher dimensions as well), and proves
local triviality of the fibration.
The gauge theoretic approach taken here allows for a precise description of the
isomonodromic distribution.

\subsection*{Structure of the paper} Section \ref{s:Preliminaries} is devoted to 
preliminaries in gauge theory. In Section
\ref{sec:differential_geometry_of_configuration_spaces} we introduce the
configuration space and study its main properties. The construction of the
joint moduli space is done in Section \ref{s:ConstructingModuliSpace}. In
particular, we prove Theorem \ref{MainThmA} and discuss the family of
hermitian forms $(h_s)_{s\in\mathbb R_+}$. 
In Section \ref{sec:iso hor and energy}, we define the horizontal and isomonodromic distributions and study their relation with the energy function, proving Theorems \ref{MainThmB} and \ref{MainThmC}. The stratification of $\bfM(\G)$ is treated in Section \ref{sec:Stratification}, where we prove Theorems \ref{MainThmD} and \ref{MainThmE}. Finally, the perspective on minimal surfaces and Higgs bundles is treated in Section \ref{sec:MinimalSurfaces} where Theorems \ref{MainThmF}, \ref{MainThmG} and \ref{MainThmH} are proved.
The four sections of the  Appendix serve to fix notation,  relate 
the Dolbeault and \v{C}ech deformation complexes, give an
alternative expression for $\Theta$, and  
show how \eqref{eqn:hessian} recovers  To\v{s}i\'c's formula.

\subsection*{Acknowledgements:} The authors are grateful to Nigel Hitchin
for his interest in this work and for many   essential   comments.  
Thanks also to
Luis \' Alvarez C\' onsul, Steve Bradlow, Gerard Freixas i Montplet,
Mario Garcia-Fernandez, Oscar Garc\'ia-Prada, Pierre Godfard,  Fran\c{c}ois
Labourie,\break Qiongling Li, Daniel Litt, Filippo Mazzoli,
Swarnava Mukhopadhyay, 
Nathaniel Sagman, Andy Sanders, Carlos Simpson,
Peter Smillie, Nicolas Tholozan, and Charlie Wu for their input and suggestions. 
Portions  of this work were  completed  at the Institut Henri Poincar\'e and 
the Institut des Hautes Etudes Scientifiques, Paris, and the Instituto
de Ciencias Mathem\'aticas, Madrid; their generous hospitality is
warmly acknowledged. 
B.C.’s research is supported by NSF grant DMS-2337451.
J.T. was partially funded by the European Research Council under ERC-Advanced grant 101095722, and by the Institut Universitaire de France.
R.W.’s research is supported by NSF grants DMS-2204346 and DMS-2506596.

\medskip

After work on this paper was completed, we became aware of
several other related articles. 
The preprint
    \cite{HSSZ:25}, among other results, has a description of the
    isomonodromic distribution similar to the one in Lemma \ref{lem semiharmonic
    iso gauge hermitian}. The characterization of holomorphic leaves in
    Theorem E (1) also appears there. Theorem E has also been obtained in
    \cite{LamLitt}, 
    and Corollary \ref{cor:elliptic} in \cite{LittWu}. 
    An alternative derivation of the result in Theorem C appears in Hitchin's
    note on the universal Higgs bundle moduli space \cite{Nigel}. 
    Finally, the preprint \cite{ACGFGPT:25} gives a construction of a 
    universal space which builds on the authors previous partial results in Trautwien's thesis \cite[Chapter 7]{Trautwein}.

\section{Gauge theory preliminaries}\label{s:Preliminaries}
Throughout the paper, we will fix 
\begin{itemize}
	\item $\Sig$ a  connected closed
        oriented topological surface of genus $g\geq 2$;
    \item $\sG$ a connected  complex semisimple Lie group with center $Z(\G)$, Lie
        algebra $\gfrak$,
        and Killing form $\kappa_\gfrak$; and 
	\item $\pi_P:P\to \Sig$ a smooth right principal $\sG$-bundle.
\end{itemize} 

\subsection{Principal bundles}

\subsubsection{Basic definitions}
The \emph{vertical bundle} of $P$ is the distribution $\V P$ defined by $\V
P = \ker(d \pi_P)$. It fits into the exact sequence
\begin{equation} \label{eqn:tangent-sequence}
    0 \lra \V P \lra \T P \lra \pi_P^*(\T\Sig) \lra 0~.
\end{equation}
Since the right $\sG$-action preserves the vertical bundle, every element $X\in\gfrak$ in the Lie algebra spans a vertical vector field that we denote by $X^\sharp$.

Given a linear representation $\rho:\sG\to\sGL(W)$, the \emph{associated bundle} is defined by
\[P(W):= \sG\backslash(P\times W)~,\]
where $g\cdotp(p,v)=(pg^{-1},\rho(g)v)$. The projection on the first factor turns $P(W)$ into a vector bundle over $\Sig$.
The main example is $P(\gfrak)$, where $\rho$ is the adjoint representation. 
We will denote the space of $P(W)$-valued $k$-forms on $\Sig$ by $\Omega^k(\Sig,P(W))$. 

Denote the space of $W$-valued $k$-forms on $P$ by $\Omega^k(P,W)$, and let $\Omega^k(P,W)^\sG$ denote the subset of forms which are equivariant with respect to the natural $\sG$-action.
A vector field $V$ on $P$ defines a contraction map
$\iota_V:\Omega^k(P,W)\to \Omega^{k-1}(P,W)$.
A $W$-valued $k$-form $\psi\in\Omega^k(P,W)$ is called \emph{basic}
if it is equivariant and $\iota_V(\psi)=0$ for all vertical vector fields $V\in \V P$. Denote the spaces of basic $W$-valued k-forms by $\Omega_b^k(P,W)$. 

Given a basic form $\hat \psi\in\Omega^k_b(P,W)$, define $\psi\in\Omega^k(\Sig,P(W))$ by
\[\psi_x(u_1,...,u_k) := \left[\left(p, \hat\psi_p(\hat u_1,...,\hat u_k)\right)\right]~,\]
where $x\in \Sig$, $u_1,\dots,u_k\in \T_x\Sig,$ and $p, \hat u_1, \dots, \hat u_k$ are lifts of $x, u_1,\dots, u_k$, respectively. One checks that this is independent of the choices of lifts, and defines a isomorphism $\Omega^k_b(P,W)\cong\Omega^k(\Sig,P(W))$ between basic W-valued $k$-forms on $P$ and k-forms valued in $P(W)$.

There is a natural bracket 
\[[\cdot,\cdot]: \Omega^k(P,\gfrak)\times \Omega^\ell(P,\gfrak)\to \Omega^{k+\ell}(P,\gfrak)\]
defined by 
\[[\zeta,\eta](V_1,\cdots, V_{k+\ell})=[\zeta(V_1,\cdots,V_k), \eta(V_{k+1},\cdots,V_{k+\ell})].\]
Note that $[\zeta,\eta]=(-1)^{k\ell+1}[\eta,\zeta]$, and that $[\zeta,\eta]$ is basic whenever $\zeta$ and $\eta$ are basic.  
This defines a bracket on $P(\gfrak)$-valued forms on $\Sig$ which we also denote by $[\cdot, \cdot].$

\subsubsection{Connections on principal bundles and associated bundles} 
There are many ways to think about connections on principal bundles, we will use the following equivalent definitions. 

\begin{definition} 
The following are equivalent definitions of a connection $P:$
\begin{enumerate}
 	\item  A $\sG$-equivariant $1$-form $B \in \Omega^1(P,\gfrak)^\sG$ such that $B(x^\sharp)=x$ for any $x\in\gfrak$.
 	\item A $\sG$-equivariant projection $B: \T P\to \V P$. 
    \item A $\sG$-invariant splitting of \eqref{eqn:tangent-sequence}.
\end{enumerate}
 Denote the space of connections on $P$ by $\cA(P)$. 
\end{definition}
By (3), 
a connection $B$ defines a splitting $\T P=\V P\oplus \rH_B$, where the distribution $\rH_B$ is defined by the kernel of $B$ and is called the \emph{horizontal distribution} of $B$.
Since the difference between two connections on $P$ is basic,  $\cA(P)$ is an affine space modeled on $\Omega^1_b(P,\gfrak)$. 
Explicitly, given a basic 1-form $\psi\in\Omega^1_b(P,\gfrak)$, in the splitting $\T P=\V P\oplus\rH_B$ the connection $B+\psi$ is given by 
\[B+\psi=\begin{pmatrix}
	\Id&\psi\\0&0
\end{pmatrix}:\V P\oplus\rH_B \to\V P\oplus\rH_B ,\]
where we have identified $\Omega^1_b(P,\gfrak)$ with $\sG$-equivariant bundle maps $\pi_P^*(\T\Sig) \cong\rH_B \to \V P$.

The \emph{curvature} of a connection $B$ is the $\gfrak$-valued $2$-form 
\[F_B = dB + \tfrac{1}{2}[B, B]~.\]
One checks that $F_B$ is basic and so can be considered as an element in $\Omega^2(\Sig,P(\gfrak))$. Equivalently, the curvature can be defined using the Lie bracket of horizontal vector fields: given $V,W\in\rH_B$, we have 
$$F_B(V,W)=-B([V,W]).$$

Given a representation $\rho: \sG\to\sGL(W)$, let $d\rho:\gfrak\to\End(W)$ denote the associated Lie algebra representation. For a connection $B\in\cA(P)$, define
\begin{equation}
	\label{eq cov derv of conn}d_B:\Omega^k(P,W)\to \Omega^{k+1}(P, W), 
\end{equation} 
by $d_B(\eta)=d\eta+ (d\rho\circ  B)\wedge \eta.$ One checks that $d_B$ preserves basic forms and hence defines a covariant derivative $d_B:\Omega^k(\Sig,P(W))\to \Omega^{k+1}(\Sig,P(W))$. 

\subsubsection{Structure group reductions and metrics}\label{sec metrics}
% \RW{I have change the Killing form to $\kappa_\gfrak$}
Let $\sK<\sG$ be a maximal compact subgroup of $\sG$ and
$\kfrak\subset\gfrak$  its Lie algebra. As $\kfrak$ is a real form of
$\gfrak$, it is the fixed point subalgebra of a conjugate linear involution
$\tau_0:\gfrak\to\gfrak$
called the \emph{Cartan involution} associated to $\kfrak.$ 
The Killing form $\kappa_\gfrak$ of $\gfrak$ is negative definite on
$\kfrak,$ and hence defines a $\sK$-invariant hermitian inner
product on $\gfrak$ defined by 
\[\langle x,y\rangle_{\tau_0} := \kappa_\gfrak(x,-\tau_0(y))~.\]
For $y\in \gfrak,$ the hermitian adjoint of $\ad_y$ is $\ad_{-\tau_0(y)}$. Hence we use the notation $-\tau_0(y)=y^*$.

A \emph{reduction of structure group} of $P$ to $\sK$ is the choice of
$\sK$-invariant subbundle $P_\sK\subset P$. Using the canonical
identification between the associated bundles $P_\sK(\gfrak)$ and
$P(\gfrak)$ and $\sK$-invariance of $\langle\cdot,\cdot\rangle_{\tau_0},$ the extensions of
$\tau_0$ and $\langle\cdot,\cdot\rangle_{\tau_0}$ to $P(\gfrak)$ define a hermitian metric on $P(\gfrak)$.

Equivalently, such a reduction is given by a \emph{Cartan involution on $P$}, that is a $\G$-equivariant map $\tau: P \to \End(\gfrak)$ (where $\G$ acts on $\End(\gfrak)$ via conjugation by $\Ad$) such that $\tau(p)$ is a Cartan involution for any $p\in P$. The corresponding hermitian metric is then given by 
$$
\langle s_1,s_2\rangle_\tau (x) =-\kappa_\gfrak(s_1(p),\tau(p)s_2(p))
$$
for local sections $s_a$ of $P(\gfrak)$, regarded as $\G$-equivariant
maps $\sigma_a: P\to \gfrak$ for the adjoint action. 

The Cartan involution and hermitian metric extend to $P(\gfrak)$-valued $1$-forms as follows.  

\begin{definition}
Consider $\sigma_1,\sigma_2 \in \Omega^1(\Sigma,P(\gfrak))$ of the form $\sigma_a= \alpha_a \otimes s_a$
 with $\alpha_a\in \Omega^1(\Sigma,\C)$ and $s_a \in \Omega^0(\Sigma,P(\gfrak))$.
 \begin{enumerate}
 	\item The extension of the Cartan involution to $\Omega^1(\Sigma,P(\gfrak))$ is defined by
 	\[\tau(\sigma) := \overline \alpha\otimes \tau(s)~.\]
 	We use the notation: $\sigma^* = -\tau(\sigma) = \overline \alpha\otimes s^*$.
 	\item The Killing form defines 
        \begin{align*}    
            \kappa_\gfrak(\sigma_1\otimes\sigma_2) &:= \kappa_\gfrak(s_1,s_2)
         \alpha_1\otimes\alpha_2~\in \Omega^0(T^\ast\Sigma\otimes T^\ast\Sigma,\C)\ , \\
            \kappa_\gfrak(\sigma_1\wedge \sigma_2) &:= \kappa_\gfrak(s_1,s_2)
         \alpha_1\wedge \alpha_2~\in \Omega^2(\Sigma,\C)\ .
        \end{align*}
 	\item The combination of $\kappa_\gfrak$ and $\tau$ defines: 
 	\[\kappa_\gfrak(\sigma_1\wedge \sigma_2^*) := -\kappa_\gfrak(s_1,\tau(s_2))\alpha_1\wedge \overline\alpha_2 ~.\]
 \end{enumerate}
\end{definition}
For each Riemann surface structure $X=(\Sig,j)$ on $\Sig$, the hermitian metric  defines an $L^2$-inner product on $P(\gfrak)$-valued $(p,q)$-forms. 
The following signs are important (see Appendix  \ref{sec:app-kahler}). 
\begin{lemma}\label{lem signs on L2 innerproducts}
	Suppose $\sigma_1,\sigma_2$ are $(p,q)$-forms valued in $ P(\gfrak),$ then 
	\[\langle \sigma_1,\sigma_2\rangle:=\int_X\kappa_\gfrak(s_1,s_2^*)\alpha_1\wedge \bar\star \alpha_2 =\begin{dcases}
		i\int_X\kappa_\gfrak(\sigma_1\wedge\sigma_2^*)&\text{ if }(p,q)=(1,0)\\ 
		-i\int_X \kappa_\gfrak(\sigma_1\wedge\sigma_2^*)&\text{ if }(p,q)=(0,1)
	\end{dcases}
	\]
\end{lemma}

\subsection{Complex structures on principal bundles}
We now describe the space of complex structures on a principal $\G$-bundle. 

\begin{definition}
A \emph{(principal) complex structure} on $P$ is an almost complex
    structure $J\in \End(\T P)$ satisfying
\begin{enumerate}
	\item $J$ commutes with the right $\G$-action: for any $g$ in $\G$ we
        have $J\circ (R_g)_\ast = (R_g)_\ast \circ J$.
	\item For any $X \in \gfrak$, we have $J(X^\sharp)=(iX)^\sharp$.
\end{enumerate}
We denote by $\J(P)$ the space of principal complex structures on $P$.
\end{definition}

One checks that a complex structure on $P$ is always integrable since the obstruction to integrability of $J$ lies in $\Omega^{0,2}_b(P,\gfrak)$.
Denote the space of (almost) complex structures on $\Sig$ by $\J(\Sig)$. There is a projection map \[\pi_\J:\J(P)\to\J(\Sig).\]
Here $j=\pi_\J(J)$ is defined by $j_x(u) = J_p(\hat u)$ for any lift $(p,\hat u)$ of $(x,u)$ to $\T P$. 
 Since $J$ preserves the vertical bundle and commutes with the $\G$-action, $\pi_\J(J)$ is independent of the choice of lift.

\subsubsection{The Chern-Singer connection}\label{sss:CompatibleConnection}

A connection $B\in\cA(P)$ is \emph{compatible} with a complex structure $J$ if the horizontal distribution of $B$ is $J$-invariant. Equivalently, $B$ is compatible with $J$ if and only if $B$ is of type $(1,0)$ with respect to $J$, meaning that $B \circ J = iB$.

Given a connection $B\in\cA(P)$ and a complex structure $j\in\J(\Sig)$, there is a unique complex structure $J\in\pi_\J^{-1}(j)$ which is compatible with $B$. In the splitting $\T P = \V P \oplus \rH_B $, it is given by
\begin{equation}
	\label{eq J wrt compatible connection}J = \left(\begin{array}{ll} i & 0 \\ 0 & \pi^*j \end{array}\right)~.
\end{equation}

\begin{proposition}\label{prop:SingerConnection}Let $P_\sK\subset P$ be a structure group reduction to a maximal compact subgroup $\sK<\sG$.
\begin{enumerate}
	\item 
	  For each $J\in\J(P)$, there is a unique connection $A_J\in\cA(P_\sK)$ which is compatible with $J$.
	 \item The set of connections on $P$ compatible with a complex structure $J\in\J(P)$ is an affine space modeled on $\Omega^{1,0}(\Sig,P(\gfrak))$, where the type is computed using $\pi_\J(J)$.
	\item The projection $\pi_\J:\J(P)\to \J(\Sig)$ turns $\J(P)$ into an affine  bundle where the fiber over $j$ is modeled on $\Omega^{0,1}(\Sig,P(\gfrak))$.
\end{enumerate}
\end{proposition}

Item (1) is due to Singer \cite{GeometricConnectionSinger} and is the principal bundle analogue of the Chern connection on a holomorphic hermitian vector bundle. We will refer to the isomorphism
\begin{equation}
	\label{eq chern-singer map}\rS:\J(P)\lra \cA(P_\sK),
\end{equation}
as the Chern-Singer map, and the
unitary connection $\rS(J)=A_J$ as the \emph{Chern-Singer connection}. 
We will use the same notation for the extension of $A_J$ to a connection on $P$.

\begin{proof}
For item (1), let $J\in \J(P)$. For each $p\in P_\sK$, consider the $J$-invariant space $H_p(J):=\T_p P_\sK \cap J(\T_p P_\sK)$. Since $\sK$ is a real form of $\G$, $H_p(J)$ intersects trivially the vertical space $\V_p P_\sK$. Furthermore, since $\T_p P_\sK$ has co-dimension $\tfrac{1}{2}\dim(\G)$, then $H_p(J)$ has dimension at least $2$. In particular, $\T_p P_K = \V_p P_K \oplus H_p(J)$ and so $\{H_p(J)\}_{p\in P_\K}$ defines a horizontal distribution. Such a distribution is $\sK$-invariant because $J$ commutes with the $\G$-action.  

For item $(2)$, consider a connection $B$ compatible with $J$. In the splitting $\T P = \V P \oplus \rH_B$, any other connection $B'$ is given by
\[B' = \left(\begin{array}{ll} \Id & \beta \\ 0 & 0 \end{array}\right)~,\]
where $\beta \in \Omega^1_b(P,\gfrak)$. The condition $B'\circ J = i B'$ is equivalent to $\beta \circ j =i\beta$ for $j=\pi_\J(J)$. 

For item $(3)$, fix a background connection $B$ on $P$. In the splitting $\T P = \V P \oplus \rH_B$, a point $J$ in the fiber $\pi_\J^{-1}(j)$ has the form
\[J = \left(\begin{array}{ll} i & \beta \\ 0 & \pi^*j \end{array}\right)~,\]
for some basic $1$-form $\beta\in \Omega^1_b(P,\gfrak)$. The condition $J^2=-\Id$ is equivalent to $i\beta+\beta\circ \pi^*j=0$. Hence, under the identification $\Omega^{1}_b(P,\gfrak)\cong \Omega^1(\Sig,P(\gfrak))$, $\beta$ is a $(0,1)$-form with respect to $j$.
\end{proof}

\subsubsection{Dolbeault operator on associated bundles} We now explain how
a complex structure $J$ on $P$ defines a Dolbeault operator on associated bundles. Given a representation $\rho:\G\to \GL(W)$ and a connection $B$ on $P$ compatible with $J$, define
\[\dbar_j: \Omega^0(X,P(V))\lra \Omega^{0,1}(X,P(V))\ :\ s\longmapsto (d_B s)^{0,1}~.\]

\begin{lemma}\label{lem hol w comp conn}
The operator $\overline \partial_J$ is a Dolbeault operator on $P(V)$ that is independent of the choice of compatible connection $B$.
\end{lemma}

\begin{proof}
The fact that  $\overline \partial_J$ is a Dolbeault operator follows directly from the Leibniz rule for $d_B$. The second statement follows from item (2) of Proposition \ref{prop:SingerConnection}. In fact, any other compatible connection differs from $B$ by a $(1,0)$-form, hence its $(0,1)$-projection is the same.
\end{proof}

\section{Differential geometry of configuration spaces}
\label{sec:differential_geometry_of_configuration_spaces}
\subsection{Tangent spaces and complex structures}
The tangent space at a complex structure $j\in \J(\Sig)$ is given by 
\[\T_j \J(\Sig) = \{ m \in \Omega^1(\Sig,\T\Sig) ~ \vert ~ jm+mj=0 \}~.\]
The condition $jm+mj=0$ implies that the complex linear extension $m^\C$ of $m$ has the form 
\[m^\C=\begin{pmatrix}
	0&\mu\\\overline\mu&0
\end{pmatrix}:\T_j^{1,0}\Sig\oplus \T_j^{0,1}\Sig\lra \T_j^{1,0}\Sig\oplus \T_j^{0,1}\Sig.\]
The tensor  $\mu\in\Omega^{0,1}(\Sig,\T_j^{1,0}\Sig)$ is called a \emph{Beltrami differential}. 

\begin{definition}\label{def compl on J(S)}
	Post-composing with $j$ defines an almost complex structure
    $\I:\T\J(\Sig)\to \T\J(\Sig)$,
    $$\I_j:\T_j\J(S)\lra\T_j\J(S) ~:~ m\longmapsto j m\ .$$
In terms of Beltrami differentials, the almost complex structure is $\I_j(\mu)=i\mu.$
\end{definition}

Given a complex structure $J\in\J(P)$, let $\Omega^0(\End(\T P))^\sG$ denote the space of equivariant sections. The tangent space $\T_J\J(P)$ is given by
\[\T_J\J(P)=\{M\in \Omega^0(\End(\T P))^\sG~|~ \text{$M(X)=0$ for all $X\in \Omega^0(\V P)$ and } JM+MJ=0\}~.\]
Let $B\in\cA(P)$ be a connection which is compatible with $J\in\J(P)$. In the splitting $\T P=\V P\oplus\rH_B$, tangent vectors are given by 
\begin{equation}
	\label{eq decomp of M wrt J-comp B}
	M=\begin{pmatrix}
	0&\beta\\0&m
\end{pmatrix}:\V P\oplus \rH_B\to \V P\oplus \rH_B.\end{equation}
Using \eqref{eq J wrt compatible connection}, we have
\[JM+MJ= \begin{pmatrix}
	0&i\beta+\beta\circ \pi^*j\\0&m\circ \pi^*j+\pi^*j\circ m
\end{pmatrix}.\]
Since $M$ is equivariant,  $\beta$ is identified with a basic $\gfrak$-valued $(0,1)$-form $\beta\in \Omega^{0,1}_b(P,\gfrak)$, and $m$ is the pullback of a tangent vector $\T_j\J(\Sig)$. 

As in Definition \ref{def compl on J(S)}, we have an almost complex structure $\I:\T\J(P)\to \T\J(P)$ defined by
\begin{equation}
	\label{eq complex str on J(P)}
    %\I_J:\xymatrix@R=0em{\T_J\J(P)\ar[r]&\T_J\J(P)\\M\ar@{|->}[r]& JM}.
    \I_J:\T_J\J(P)\lra\T_J\J(P) ~:~ M\longmapsto JM\ .
\end{equation}
If $B\in\cA(P)$ is a connection compatible with $J,$ then writing $M=(\beta, m)$  as in \eqref{eq decomp of M wrt J-comp B}, we have 
\[\I_J(\beta,m)=(i\beta, jm)~.\]  
In particular, the projection $\pi_\J:\J(P)\to\J(\Sig)$ is holomorphic.
The complex structures $\I$ and holomorphic projection $\pi_\J$ extend to the product
\[\J(P)\times\Omega^1_b(P,\gfrak).\] 
\begin{definition}\label{def: higg config}
	The \emph{configuration space} of Higgs bundles $\cC(P)$ is  defined by
 \[\cC(P)= \{(J,\Phi)\in \J(P)\times \Omega^1_b(P,\gfrak)~|~\Phi\circ J=i \Phi\}.\] 
 Denote the restriction of the projection map by
$\pi_\J:\cC(P)\to\J(\Sig)$.
\end{definition}

The tangent space of $\cC(P)$ is given by
\[\T_{(J,\Phi)}\cC(P)=\{(M,\theta)\in \T_J\J(P)\times \Omega^1_b(P,\gfrak)~|~ \Phi\circ M+\theta\circ J=i\theta\}~.\]
\begin{lemma}
	For each tangent vector $(M,\theta)\in \T_{(J,\Phi)}\cC(P)$ we have 
	\[\theta=\psi+\tfrac{1}{2i}\Phi\circ\mu,\]
	where $\psi$ has type $(1,0)$ with respect to $\pi_\J(J,\Phi)$ and $\mu$ is the Beltrami differential associated to $d\pi_\J(M,\theta)$.
\end{lemma}
We will write $\Phi\circ\mu$ simply as $\Phi\mu$.

\begin{proof}
The condition $\Phi\circ M+\theta\circ J=i\theta$ is equivalent to $\theta^{(0,1)_J}=\tfrac{1}{2i}\Phi\circ M~.$  Moreover, since $\Phi$ is basic and of type $(1,0)$, $\Phi\circ M=\Phi\mu$, where $\mu$ is the associated Beltrami differential.
\end{proof}
For $(M,\theta)\in \T_{(J,\Phi)}\cC(P)$ we have 
\[\Phi\circ(JM)+J\theta\circ J=i J\theta.\] Indeed, $J\theta=i\theta$  since $\theta$ is valued in vertical bundle and $\Phi(JM)=i\Phi\mu$ by definition of $(J,\Phi)\in\cC(P)$. In particular, $\cC(P)$ is a holomorphic submanifold of $\J(P)\times\Omega^1_b(P,\gfrak)$ and so inherits a complex structure $\I$. Explicitly, 
$$\I_{(J,\Phi)}:\T_{(J,\Phi)}\cC(P)\lra\T_{(J,\Phi)}\cC(P) ~:~ (M,\theta)\longmapsto
    (JM,J\theta)=(JM,i\theta)\ .$$

Fix a structure group reduction $P_\sK$ to a maximal compact subgroup $\sK$ of $\sG$.
Let $(J,\Phi)\in \cC(P)$ be a Higgs bundle and $A_J$ be the associated Chern-Singer connection of $J$. The following lemma is immediate from the above discussion.

\begin{lemma}\label{lem: tangent vector decomp}
	In the splitting $\T P=\V P\oplus \rH_{A_J}$, a tangent vector $(M,\theta)\in \T_{(J,\Phi)}\cC(P)$ is given by 
\[M=\begin{pmatrix}
	0&\beta\\0&m
\end{pmatrix}\ \ \ \ \ \text{and}\ \ \ \ \ \theta = \begin{pmatrix}
	0&\psi+\tfrac{1}{2i}\Phi\mu\\0&0
\end{pmatrix},\]
where $\mu$ is the Beltrami differential associated to $m\in\T_j\J(\Sig)$. In particular, $(M,\theta)$ is uniquely determined by a tuple $(\mu,\beta,\psi),$ where $\mu$ is a Beltrami differential, $\beta\in\Omega_b^{0,1}(P,\gfrak)$ and $\psi\in\Omega_b^{1,0}(P,\gfrak).$
\end{lemma}

\subsection{The holomorphic section $\Theta$}\label{section hol bundle map}
As above, let $\pi_\cJ:\cC(P)\to \J(\Sig)$ be the holomorphic projection. 
We now define a holomorphic section $\Theta$ of the bundle $\Hom(\pi_\cJ^*\T\J(\Sig),\ker(\rd\pi_\J))$ over $\cC(P).$ 

To start, consider the bundle map $\widehat\Theta:\T\cC(P)\to\T\cC(P)$ defined at $(J,\Phi)\in\cC(P)$
$$\widehat\Theta_{(J,\Phi)} : \T_{(J,\Phi)}\cC(P)\lra\T_{(J,\Phi)}\cC(P) ~:~
(M,\theta)\longmapsto(\tfrac{1}{2i}\Phi\circ M,0)\ .
$$

Since $\Phi$ takes values in the vertical bundle, $\widehat\Theta$ takes values in the subbundle $\ker(\rd\pi_\cJ)$. Since $\Phi$ is basic and has type $(1,0)$ we have $\Phi\circ M=\Phi \mu$, where $\mu$ is the Beltrami differential associated to $\rd\pi_\cJ(M).$ 
Hence, $\widehat\Theta$ defines a bundle map 
\begin{equation}
	\label{eq bundle map eta}
    \Theta: \pi^*_\J\T\J(\Sig)\lra\ker(\rd\pi_\cJ) ~:~
    ((J,\Phi),\mu)\longmapsto((J,\Phi),(\tfrac{1}{2i}\Phi\circ\mu, 0)) \ .
\end{equation}
\begin{remark}\label{rem:eta in beta notation}
	As in Lemma \ref{lem: tangent vector decomp}, fixing a structure group reduction $P_\sK\subset P$ allows us to write tangent vectors $(M,\theta)$ as $(\mu,\beta,\psi)$. In terms of this data, we have 
	\[\Theta(\mu)=(0,\tfrac{1}{2i}\Phi\mu,0).\]
\end{remark}
\begin{lemma}
	The bundle map $\Theta$ is holomorphic.
\end{lemma}
\begin{proof}
To prove the lemma, we will show that $\widehat\Theta$, seen as a smooth map from $\T \cC(P)$ to itself covering the identity, is holomorphic.The derivative of $\widehat\Theta$ at $(J,\Phi,M,\theta)$ 
 	\[\rd_{(J,\Phi,M,\theta)}\widehat\Theta(X,\varphi,N, \vartheta)=(X, \varphi, \tfrac{1}{2i}\varphi\circ M+\tfrac{1}{2i}\Phi\circ N,0).\]
 	The complex structure $\widehat I$ on $\T_{(J,\Phi,M,\theta)}(\T\cC(P))$ is given by
 	$\widehat I(X,\varphi,N, \vartheta)=(JX,i\varphi, JN, i\vartheta).$
 	Hence,
 	\begin{align*}
 		\rd_{(J,\Phi,M,\theta)}\widehat\Theta(\widehat I(X,\varphi,N, \vartheta))&=\rd_{(J,\Phi,M,\theta)}\widehat\Theta(JX,i\varphi,JN, i\vartheta)\\ 
 		&=(J X,i \varphi, \tfrac{1}{2i} i\varphi\circ M+\tfrac{1}{2i}\Phi\circ JN,0)\\
 		&=(J X,i \varphi,i(\tfrac{1}{2i}\varphi\circ M+\tfrac{1}{2i}\Phi\circ N),0),
 	\end{align*}
where for the last equality we used $\Phi\circ J N=\Phi\circ(i\nu)$ where $\nu$ is the Beltrami differential associated to $d\pi (N).$
Since $\varphi\circ M+\Phi\circ N$ is a vertical vector, we have 
\[(J X,i \varphi,i(\tfrac{1}{2i}\varphi\circ M+\tfrac{1}{2i}\Phi\circ N),0)=\widehat I(X, \varphi,\tfrac{1}{2i}\varphi\circ M+\tfrac{1}{2i}\Phi\circ N,0).\]
Thus, $\rd\widehat\Theta$ is complex linear as desired.
\end{proof}

\subsection{The closed 2-form $\omega_0$ and the hermitian form $h_0$}
In this subsection we fix a structure group reduction $P_\sK\subset P$ to
a maximal compact subgroup $\sK<\sG.$ Recall that $-\tau(X)=X^*$ is the
hermitian adjoint of $X$ with respect to the hermitian metric
$\kappa_\gfrak(\cdot,-\tau(\cdot))$, where $\kappa_\gfrak$ is the Killing form.
\subsubsection{The Atiyah-Bott-Goldman form} 
For a principal $\sG$-bundle $P,$ the vector space $\Omega^1(\Sig,P(\gfrak))$ has a complex symplectic form called the \emph{ Atiyah-Bott-Goldman form}. It is defined by 
\[\omega^\C_{ABG}(\eta_1,\eta_2)=\int_\Sig \kappa_\gfrak(\eta_1\wedge \eta_2).\]
Since the  space of connections $\cA(P)$ is affine over $\Omega^1(\Sig,P(\gfrak))$, $\omega_{ABG}^\C$  defines a nondegenerate 2-form on $\cA(P)$. Moreover, $\omega^\C_{ABG}$ is closed since it is independent of the base point.

The structure group reduction $P_\sK\subset P$  defines a real symplectic form on $\cA(P)$ which we denote by $\omega_{ABG}$. 
Writing $\eta=X+iY$ for $X,Y\in \Omega_b^1(P_\sK,\kfrak)$, we have
\[\omega^\C_{ABG}(\eta_1,\eta_2)=\omega^\C_{ABG}(X_1+iY_1,X_2+iY_2)=\int_\Sig \kappa_\gfrak( X_1\wedge X_2-Y_1\wedge Y_2)+ i \int_\Sig \kappa_\gfrak(X_1\wedge Y_2+Y_1\wedge X_2).\]
The Killing form $\kappa_\gfrak$ is real on $\kfrak$, $\omega_{ABG}$ is defined as the real part of $\omega^\C_{ABG}$:
\begin{equation}
	\label{eq ABG}\omega_{ABG}(X_1+iY_1,X_2+iY_2)=\int_\Sig \kappa_\gfrak(X_1\wedge X_2-Y_1\wedge Y_2).
\end{equation}

\subsubsection{The pullback of the Atiyah-Bott-Goldman form}
Recall from \eqref{eq chern-singer map} that the structure group reduction $P_\sK\subset P$ defines the Chern-Singer isomorphism  $\rS:\J(\Sig)\to\cA(P_\sK)$, where we denote $\rS(J)=A_J$. The defining property of $A_J$ is  $A_J\circ J=i A_J.$ Consider the following smooth map 
\begin{equation}
	\label{eq H map}\rH:\J(P)\times\Omega^1_b(P,\gfrak)\lra  \cA(P)\ :\ (J,\Psi)\longmapsto A_J+\tfrac{1}{2i}(\Psi-\Psi^*)
\end{equation}

We start by computing the derivative of the Singer map.
\begin{lemma}\label{lem:DifferentialOfChernSinger}
	For $J\in\J(\Sig)$, the differential of the Chern-Singer map $\rS$ at $J$ is given by 
	\[\rd_J\rS(M)=\tfrac{1}{2i}(\beta+\beta^*),\]
	where $M=\begin{pmatrix}
		0&\beta\\0&m
	\end{pmatrix}$ in the splitting $\T P=\V P\oplus \rH_{A_J}$ associated to $A_J$.
\end{lemma}
\begin{proof}
	Consider a path $J_t\in\J(P)$ with $J=J_0$ and $\deriv J_t=M$. In the splitting $\T P=\V P\oplus \rH_{A_J} P$, we have
	\[J_t=\begin{pmatrix}
		i&\alpha_t\\0&j_t
	\end{pmatrix}\ \ \ \ \ \text{and} \ \ \ \ \ \rS(J_t)=A_{J_t}=\begin{pmatrix}
		\Id&\eta_t\\0&0
	\end{pmatrix},\]
where $\eta_t,\alpha_t\in\Omega^1_b(P,\gfrak)$ satisfy $\eta_t=-\eta_t^*$ and $\alpha_t\circ J_t=i\alpha_t$ and are both zero when $t=0$. 

Denoting the derivatives of $\eta_t$ at $t=0$ by $\dot \eta$, we have $\rd_J\rS(M)=\dot\eta.$ 
The equation $A_{J_t}\circ J_t=iA_{J_t}$ is given by  
\[\begin{pmatrix}
		\Id&\eta_t\\0&0
	\end{pmatrix}\begin{pmatrix}
		i&\alpha_t\\0&j_t
	\end{pmatrix}=\begin{pmatrix}
		i\Id&i\eta_t\\0&0
	\end{pmatrix}.\]
	Hence, $\alpha_t+\eta_t\circ j_t=i\eta_t$. 
	Differentiating at $0$ and using $\eta_0=0$, we have 
	\[\beta+\dot\eta\circ j=i\dot\eta,\]
	where $\beta=\deriv\alpha_t$. 
	Thus, $\beta=2i (\dot\eta)^{0,1}$, and so
	\[\beta+\beta^* = 2i (\dot\eta)^{0,1} - 2i \left((\dot\eta)^{0,1}\right)^* = 2i(\dot\eta)^{0,1}+ 2i(\dot\eta)^{1,0}~,\]
	where we used $\dot\eta=-\dot\eta^*$.
\end{proof}
The following lemma is now immediate.
\begin{lemma}\label{lem deriv of H config}
	The derivative of the map $\rH:\J(P)\times \Omega^1_b(P,\gfrak)$ at $(J,\Psi)$ in the direction $(M,\theta)$ is 
		\[\rd_{(J,\Psi)}\rH(M,\theta)=\tfrac{1}{2i}(\beta+\beta^*)+\tfrac{1}{2i}(\theta-\theta^*),\]
		where $M=\begin{pmatrix}
		0&\beta\\0&m
	\end{pmatrix}$ in the splitting $\T P=\V P\oplus \rH_{A_J}$ associated to $A_J$.
\end{lemma}
Note that the decomposition $\rd_{(J,\Psi)}\rH(M,\theta)=X+iY$ for $X,Y\in\Omega^1(P_\sK,\kfrak)$ is given by
\[\xymatrix{X=\tfrac{1}{2i}(\beta+\beta^*)&\text{and}& Y=-\tfrac{1}{2}(\theta-\theta^*)}.\]
Using \eqref{eq ABG}, the pullback of $\omega_{ABG}$ by $\rH$ is thus given by
\begin{equation}
	\label{eq pullbackABG}
	\vcenter{\xymatrix@=0em{H^*\omega_{ABG}((M_1,\theta_1),(M_2,\theta_2))&=& -\dfrac{1}{4}\displaystyle\int_\Sig \kappa_\gfrak(\beta_1\wedge\beta_2^*+\beta_1^*\wedge\beta_2)&+ \dfrac{1}{4}\displaystyle\int_\Sig \kappa_\gfrak(\theta_1\wedge\theta_2^*+\theta_1^*\wedge\theta_2)\\ 
		&&-\dfrac{1}{4}\displaystyle\int_\Sig \kappa_\gfrak(\theta_1\wedge\theta_2+\theta_1^*\wedge\theta_2^*)&}}
\end{equation}
\begin{lemma}\label{lem I-inv closed}
The first two integrals in \eqref{eq pullbackABG} are $\I$-invariant and the third integral is $\I$-anti-invariant. Moreover, all three integrals define  closed 
    2-forms on $\J(P)\times \Omega^1_b(P,\gfrak).$ 
\end{lemma}
\begin{proof}
	 For $(J,\Psi)\in\J(P)\times \Omega^1_b(P,\gfrak)$, recall the complex structure $\I$ acts on a tangent vector $(M,\theta)=(\mu,\beta,\theta)$ by $\I(\mu,\beta,\theta)=(i\mu, i\beta,i\theta)$. Hence, the first two integrals are $\I$-invariant, and the third is $\I$-anti-invariant.
	The last two integrals are closed because they are independent of the
    base point $(J,\Psi)$ in $\J(P)\times \Omega^1_b(P,\gfrak)$. The
    first
     integral is now closed because $\omega_{ABG}$ is closed.
\end{proof}

\subsubsection{The closed 2-form $\omega_0$ and hermitian form $h_0$}
We denote by $\rH:\cC(P)\to \cA(P)$, the restriction of the map $\rH$ from \eqref{eq H map} to the configuration space of Higgs bundles $\cC(P)$. 
By Lemma \ref{lem I-inv closed}, the restriction of the $\I$-invariant part of $\rH^*\omega_{ABG}$ to $\cC(P)$ is closed.
\begin{definition}\label{def:omega0}
	The closed 2-form $\omega_0$ on $\cC(P)$ is defined to be the $\I$-invariant part of the pullback of the Atiyah-Bott-Goldman form by the map $\rH$. 
	Denote the associated hermitian form  by $h_0$:
	\[h_0(v_1,v_2)=2(\omega_0(\I(v_1),v_2)+i\omega_0(v_1,v_2)).\]
\end{definition}

\begin{lemma}\label{lem omega_0 formula}
	Fix $(J,\Phi)\in \cC(P)$ and tangent vectors $v_1,v_2$ with $v_a=(\mu_a,\beta_a,\theta_a)$ and $\theta_a=\psi_a+\tfrac{1}{2i}\Phi\mu_a$. Then 
	\begin{enumerate}
	\item the form $\omega_0$ is given by
		\begin{align*}
			\omega_0(v_1,v_2)&= -\frac{1}{4}\int_\Sig \kappa_\gfrak(\beta_1\wedge\beta_2^*+\beta_1^*\wedge\beta_2)+\frac{1}{4}\int_\Sig \kappa_{\gfrak}(\theta_1\wedge\theta_2^*+\theta_1^*\wedge\theta_2)\\ &=-\frac{1}{4}\int_\Sig \kappa_\gfrak(\beta_1\wedge\beta_2^*+\beta_1^*\wedge\beta_2)+\frac{1}{4}\int_\Sig \kappa_{\gfrak}(\psi_1\wedge\psi_2^*+\psi_1^*\wedge\psi_2)\\&\ \ \ \ \ +\frac{1}{16}\int_\Sig \kappa_\gfrak(\Phi\mu_1\wedge\Phi^*\overline\mu_2+\Phi^*\overline\mu_1\wedge\Phi\mu_2) ~.
		\end{align*}
	\item The hermitian form $h_0$ is equal to
		\[h_0(v_1,v_2)=\langle \beta_1,\beta_2\rangle +\langle \psi_1,\psi_2\rangle-\tfrac{1}{4}\langle \Phi\mu_1,\Phi\mu_2\rangle,\]
where $\langle \cdot,\cdot\rangle$ denotes the $L^2$-inner products from Lemma \ref{lem signs on L2 innerproducts}.
\end{enumerate}
\end{lemma}
\begin{proof}
For $\omega_0$, the first equation follows from \eqref{eq pullbackABG}. The second follows from  $\theta_a=\psi_a+\tfrac{1}{2i}\Phi\mu_a$ and $\theta_a^*=\psi_a^*-\tfrac{1}{2i}\Phi^*\overline\mu_a$, and the fact that $\psi_a$ and $\Phi\mu_a$ has type $(1,0)$ and $(0,1)$, respectively.

For the second item, we compute $h_0(v_1,v_2)=2\omega_0(\I(v_1),v_2)+i2\omega_0(v_1,v_2)$ to be 
	\[h_0(v_1,v_2)=-i\int_\Sig \kappa_\gfrak(\beta_1\wedge\beta_2^*)+i\int_\Sig \kappa_\gfrak(\psi_1\wedge\psi_2^*)+\frac{i}{4}\int_\Sig \kappa_\gfrak(\Phi\mu_1\wedge\Phi^*\overline\mu_2)\]
	The result follows from $\beta$ and $\Phi\mu$ being $(0,1)$-forms, $\psi$ being a $(1,0)$-form and Lemma \ref{lem signs on L2 innerproducts}.
\end{proof}

\subsection{The family of forms $\omega_s$ and $h_s$}
Unlike the Atiyah-Bott-Goldman form $\omega_{ABG}$, the closed 2-form
$\omega_0$ (equivalently the hermitian form $h_0$) is not everywhere
nondegenerate. Even when $h_0$ is nondegenerate, it might not be positive
definite. In this section, we modify these forms by adding a multiple of
the Weil-Petersson form on $\J(\Sig)$ to obtain a positive definite form
$h_{s_0}$ on the tangent space $\T_{(J,\Phi)}\cC(P)$. The scale $s_0\in\R_+$,
will depend on the point $(J,\Phi)$. 

\subsubsection{Weil-Petersson form}
For each complex structure $j\in\J(\Sig)$ on $\Sig,$ let $\rho$ denote the associated conformal Riemannian metric with constant curvature $-1$. This choice defines nondegenerate 2-form $\omega_{WP}$ on $\T_j\J(\Sig)$ called the Weil-Petersson form. For tangent vectors $m_1,m_2\in T_j\J(\Sig)$, it is defined by 
 	 \[\omega_{WP}(m_1,m_2) = \frac{1}{2}\int_\Sig \tr(m_1 j m_2)
     \nu_{\rho} =\frac{1}{2i}\int_\Sig (\mu_1\overline\mu_2 -
     \mu_2\overline\mu_1)\, \nu_{\rho}~,\]
     where $\nu_\rho$ is the area form associated to $\rho$ (see Appendix
     \ref{sec:app-kahler} for our conventions).
 	 The associated hermitian form  is given by
\[h_{WP}(\mu_1,\mu_2)=2\omega_{WP}(i\mu_1,\mu_2)+2i\omega_{WP}(\mu_1,\mu_2)=\int_\Sig\mu_1\overline\mu_2
\, \nu_\rho ~.\]
The Weil-Petersson form $\omega_{WP}$ is \emph{not} closed on $\J(\Sig)$,
but it does descend to a closed form on the Teichm\"uller space of $\Sig$,
see the proof of Corollary \ref{c:Closed} below. 

 \subsubsection{The forms $\omega_s$ and $h_s$}\label{sec: omega_s and h_s}
\begin{definition}
	For $s\in \R_{\geq 0}$, let $\omega_s$ be the 2-form on $\cC(P)$ defined by 
	\[\omega_s=s\cdot \omega_{WP}+\omega_0.\]
	The associated hermitian form will denoted by $h_s=s \cdot h_{WP}+h_0$.
\end{definition}
Explicitly for tangent vectors $v_1,v_2$ with
$v_a=(\mu_a,\beta_a,\psi_a)$, the hermitian form $h_s$ is given by 
\begin{equation}
	\label{eq hs metric}h_s(v_1,v_2)=s\cdot
    h_{WP}(\mu_1,\mu_2)-\tfrac{1}{4}\langle \Phi\mu_1,\Phi\mu_2\rangle
    +\langle \beta_1,\beta_2\rangle +\langle \psi_1,\psi_2\rangle\ .
\end{equation}
Equivalently, we have 
\[h_s(v_1,v_2)=\int_\Sig
(s-\tfrac{1}{4}|\Phi|^2)\mu_1\overline\mu_2\, \nu_\rho-\frac{i}{2}\int_\Sig \kappa_\gfrak(\beta_1\wedge\beta_2^*)+\frac{i}{2}\int_\Sig \kappa_\gfrak(\psi_1\wedge\psi_2^*),\]
where $\vert \Phi\vert$ is the pointwise norm of $\Phi$. 
The following is immediate.

\begin{lemma}\label{lem sufficient large s}
For each $(J,\Phi)\in\cC(P)$, there exists  $s_0>0$ such that $h_s$ is
    positive definite on $\T_{(J,\Phi)}\cC(P)$ for any $s> s_0$.
\end{lemma}

\subsection{Exponential maps}
In our moduli space construction we will need a holomorphic ``exponential map'' from the tangent space of $\T_{(J,\Phi)}\cC(P)$ to $\cC(P)$. This is more involved than the exponential map for the affine space of connections. 
\subsubsection{Almost complex structures}
Following  \cite[\S 4]{Tromba}, we describe explicit local $\I$-holomorphic charts on $\J(\Sig)$ centered at $j$. 
\begin{definition}\label{def: exp_j}
	For a complex structure $j\in\J(\Sig)$, the exponential map at $j$ is defined by 
$$\exp_j:U_j\subset
    T_j\J(\Sig)\lra\J(\Sig) ~:~ m\longmapsto\left(\Id+\tfrac{1}{2}jm\right)
    \circ j\circ \left(\Id+\tfrac{1}{2}jm\right)^{-1},$$
	where $U_j$ is the subset of where $\left(\Id+\tfrac{1}{2}jm\right)$ is invertible.
\end{definition}
In terms of Beltrami differentials, one computes that $U_j\subset T_j\J(S)$ is the complex subset 
\[U_j=\left\{\mu\in \Omega^{0,1}(T^{1,0}_j\Sig)~,~|\mu|^2<4\right\} .\]
The following lemma justifies calling the map in Definition \ref{def: exp_j}
an ``exponential map.'' 
\begin{lemma}\label{lem expo props j}
	Let $j\in\J(\Sig)$, the exponential map $\exp_j$ satisfies the following:
	\begin{enumerate}
		\item $\exp_j(0)=j$,
		\item $\rd_0\exp_j=\Id$, and 
		\item $\exp_j$ is holomorphic.
	\end{enumerate}
\end{lemma}
\begin{proof}
Item (1) is  obvious. For item (2) and (3), fix $j\in \J(\Sig)$ and set $k=\exp_j(m).$ 
The differential of $\exp_j$ at $m\in \T_j\J(\Sig)$ is given by 
\begin{align*}
	\rd_m\exp_j (\dot m)&=\tfrac{1}{2}j \dot m j \left(\Id+\tfrac{1}{2}jm\right)^{-1}-\left(\Id+\tfrac{1}{2}jm\right)j\left(\Id+\tfrac{1}{2}jm\right)^{-1}\frac{j\dot m}{2}\left(\Id+\tfrac{1}{2}jm\right)^{-1}\\&=\tfrac{1}{2}(\Id-kj)\dot m\left(\Id+\tfrac{1}{2}jm\right)^{-1},
\end{align*}
where we used that tangent vectors $\dot m\in\T_m(\T_j\J(\Sig))$ satisfies $j\dot m+\dot mj=0$.
Evaluating at $m=0$, we have $\rd_0\exp_j=\Id$, proving item (2). Item (3) is equivalent to $D_m\exp_j \circ j=k \circ D_m\exp_j$. The result follows from direct computation
\[\rd_m\exp_j\circ j= \tfrac{1}{2}(\Id-kj) j \left(\Id+\tfrac{1}{2}jm\right)^{-1} =\tfrac{1}{2}(j+k) \left(\Id+\tfrac{1}{2}jm\right)^{-1} = k\circ \rd_m\exp_j.\]
\end{proof}
Analogous to Definition \ref{def: exp_j}, for $J\in\J(P)$ define the exponential map
$$\exp_J: U_J\subset T_J\J(P) \lra  \J(P) ~:~ 
M\longmapsto\left(\Id+\tfrac{1}{2}JM\right) J
\left(\Id+\tfrac{1}{2}JM\right)^{-1}~,$$
where $U_J=\{M\in T_J\J(P)~|~ \left(\Id+\tfrac{1}{2}JM\right) \text{ is invertible}\}$.
 Parallel to Lemma \ref{lem expo props j}, we have that  $\exp_J$ is holomorphic and satisfies $\exp_J(0)=J$ and $\rd_0\exp_J=\Id.$ 

\begin{lemma} \label{lem explicit form of exp_J}
	Let $B\in\cA(P)$ be a connection compatible with $J\in\J(P)$, let $j=\pi_\J(J)$  and write tangent vectors $M=(\beta,m)\in \T_J\J(P)$ as in \eqref{eq decomp of M wrt J-comp B}. Then, in the splitting $\T P=\V P\oplus \rH_B $,  we have
\begin{enumerate}
	\item $U_J\cong \Omega^{0,1}_b(P,\gfrak)\times U_{j}$, 
	\item $\exp_J(\beta,m)=\begin{pmatrix}
		i&\beta \left(\Id+\tfrac{1}{2}jm\right)^{-1}\\0&\exp_j(m)
	\end{pmatrix}$ 
\end{enumerate}
 \end{lemma}

\begin{proof}
	For Item (1), in the splitting $\T P=\V P\oplus \rH_B$ we have 
	\[\left(\Id+\tfrac{1}{2}JM\right)=\begin{pmatrix}
		\Id&\tfrac{1}{2}i\beta\\0&\Id+\tfrac{1}{2}jm
	\end{pmatrix}.\]
	Hence $\left(\Id+\tfrac{1}{2}JM\right)$ is invertible if and only if $\left(\Id+\tfrac{1}{2}jm\right)$ is invertible. Specifically,
	\[\left(\Id+\tfrac{1}{2}JM\right)^{-1}=\begin{pmatrix}
		\Id&\tfrac{1}{2i}\beta\left(\Id+\tfrac{1}{2}jm\right)^{-1}\\0&\left(\Id+\tfrac{1}{2}jm\right)^{-1}
	\end{pmatrix}.\]
	Item (2) now follows by a direct computation using that $\beta$ has type $(0,1)$ with respect to $j$.
\end{proof}

\subsubsection{Higgs bundles}
Before defining the exponential map for the configuration space $\cC(P)$ of Higgs bundles,  we prove the following lemma. 
\begin{lemma}
	Fix $(J,\Phi)\in \cC(P)$ and a tangent vector $M\in\T_J\J(P)$. Then, for $\eta\in \Omega^{(1,0)_J}(P,\T P),$ we have 
	\[(\Phi+\eta)\left(\Id-\tfrac{1}{2}JM\right)\circ \exp_J(M)= i\cdot (\Phi+\eta)\left(\Id-\tfrac{1}{2}JM\right).\]
\end{lemma}
\begin{proof}
	Using the definition of  $\exp_J(M)$, we compute
	\begin{align*}
				(\Phi+\eta)\left(\Id-\tfrac{1}{2}JM\right)\circ \exp_J(M)&=(\Phi+\eta)\left(\Id-\tfrac{1}{2}JM\right)\left(\Id+\tfrac{1}{2}JM\right)J\left(\Id+\tfrac{1}{2}JM\right)^{-1}\\ &=(\Phi+\eta)\left(\Id-\tfrac{1}{4}JMJM\right)J\left(\Id+\tfrac{1}{2}JM\right)^{-1}\\ 
				&=(\Phi+\eta)J\left(\Id-\tfrac{1}{4}M^2\right)\left(\Id+\tfrac{1}{2}JM\right)^{-1}\\
				&=i\cdot(\Phi+\eta)\left(\Id-\tfrac{1}{4}M^2\right)\left(\Id+\tfrac{1}{2}JM\right)^{-1}\\ 
				&= i\cdot(\Phi+\eta)\left(\Id-\tfrac{1}{2}JM\right)
			\end{align*}
			where we used $J^2=-\Id$, $JM=-MJ$ and $(\Phi+\eta)J=i(\Phi+\eta)$. 
\end{proof}

With the above lemma, we define the exponential map for $\cC(P)$ as follows.
\begin{definition}\label{def Higgs exp}
Let $(J,\Phi)\in\cC(P)$, and for  $(M,\theta)\in \T_{(J,\Phi)}\cC(P)$, write $\theta=\psi+\frac{1}{2i}\Phi\mu$. The exponential map at $(J,\Phi)$ is defined by 
\[\exp_{(J,\Phi)}(M,\theta)=\left(\exp_J(M), (\Phi+\psi)\left(\Id-\tfrac{1}{2}JM\right)\right),\]
	where the domain is the set $V_J=\left\{(M,\theta)\in \T_{(J,\Phi)}\cC(P)~,~\left(\Id+\tfrac{1}{2}JM\right) \text{ is invertible}\right\}$.
\end{definition}
As with $\J(P)$, the exponential map satisfies the following.
\begin{lemma}
 	For each $(J,\Phi)\in\cC(P)$, the exponential map $\exp_{(J,\Phi)}$ is holomorphic and satisfies
 	\begin{enumerate}
 		\item $\exp_{(J,\Phi)}(0)=(J,\Phi)$, and
 		\item $\rd_0\exp_{(J,\Phi)}=\Id$, 
 	\end{enumerate}	
 \end{lemma} 
 \begin{proof}
 	The first item is obvious. For $M\in \T_J\J(P)$, set $K=\exp_{J}(M)$. The differential of $\exp_{(J,\Phi)}$ at $(M,\theta)\in \T_{(J,\Phi)}\cC(P)$ is 
 	\[\rd_{(M,\theta)}\exp_{(J,\Phi)}(\dot M,\dot \theta)=\left(\tfrac{1}{2}(\Id-KJ)\dot M\left(\Id+\tfrac{1}{2}JM\right)^{-1}, \dot\psi(\Id-\tfrac{1}{2}JM)+(\Phi+\psi)(-\tfrac{1}{2}J\dot M)\right ),\]
 	where $\dot\theta=\dot\psi+\tfrac{1}{2i}\Phi\dot M$, and the first term is computed as in the proof of Lemma \ref{lem expo props j}. 
 	The second item now follows by evaluating at $(M,\psi)=0$. Namely, 
 	\[\rd_{0}\exp_{(J,\Phi)}(\dot M,\dot \theta)=(\dot M, \dot \psi+\tfrac{1}{2i}\Phi\dot M)= (\dot M,\dot \theta),\]
 	where we used that $\Phi \circ J=i \Phi$.
The exponential map is holomorphic since  the complex structure at $K$ is given by postcomposing the first factor by $K$ and multiplying the second factor by $i$. The first term is holomorphic since $\exp_J$ is holomorphic. 
 \end{proof}

Fix a reduction $P_\sK\subset P$ of the  structure group to a maximal
compact $\K<\G$.
Let $(J,\Phi)\in \cC(P)$ be a Higgs bundle and $A_J$ 
the Chern-Singer connection associated to $J$.
The following lemma is immediate from the above discussion.

\begin{lemma} 
In terms of the tangent data $(\mu,\beta,\psi),$ the exponential map $\exp_{(J,\Phi)}:\T_{(J,\Phi)}\cC(P)\to\cC(P)$ is  
 \[\exp_{(J,\Phi)}(\mu,\beta,\psi)=\left(\exp_J(\mu,\beta),~ \Phi+\psi+\tfrac{1}{2i}(\Phi+\psi) \mu\right).\]
\end{lemma}

\section{Construction of the joint moduli space}\label{s:ConstructingModuliSpace}
This section is devoted to the construction of the joint moduli space
$\bfM(\G)$ of stable $\G$-Higgs bundles. 
This moduli space is a complex orbifold equipped with a holomorphic submersion to 
Teichm\"uller space
and a holomorphic action of the mapping class group of the  surface $\Sig$.
Furthermore, $\bfM(\sG)$ admits an exhaustion by open sets $\cU_s$ equipped with mapping class group invariant K\"ahler forms $\omega_s$, see Theorem \ref{theo:ModuliSpace}. 
We have included a fair number of details in this section which are likely well known to experts. We hope the exposition benefits those who are not familiar with the various subtleties.

\subsection{The automorphism group}\label{sec auto group}
An \emph{automorphism} of $P$ is a $\sG$-equivariant diffeomorphism $\varphi:P\to P$. Such an automorphism preserves the fibers of $P$, and so covers a diffeomorphism of $\Sig$. We will be interested in the subgroup $\Aut_0(P)<\Aut(P)$ which covers the identity component $\Diff_0(\Sig)$ of $\Diff(\Sig)$. In particular, the group $\Aut_0(P)$ fits into the short exact sequence
\begin{equation} \label{eqn:gauge-group}
    0 \lra \cG (P) \lra \Aut_0(P) \lra \Diff_0(\Sig) \lra 0~.
\end{equation}
The subgroup $\cG(P)$ is called the \emph{gauge group} of $P$. We note that $\Aut_0(P)$ is not necessarily connected since it has the same components as the gauge group $\cG(P).$ 
The quotient $\Aut(P)/\Aut_0(P)$ is isomorphic to the mapping class group of $\Sigma$.

On the Lie algebra level, the above exact sequence gives
\[0 \lra \Omega^0(\Sig,P(\gfrak)) \lra \aut(P) \lra \Omega^0(\Sig,\T \Sig)
\lra 0~.\]
The Lie algebra $\aut(P)$ of $\Aut(P)$ consists of $\sG$-invariant vector fields on $P$. 
The Lie algebra of $\cG(P)$ consists of $\G$-invariant vertical vector fields on $P$, which is identified with sections of the adjoint bundle $P(\gfrak)$.

For a structure group reduction $P_\sK\subset P$ the groups $\sG(P_\sK)$ and $\Aut_0(P_\sK)$ are defined analogously.
\begin{remark} \label{rem:vertical}
The invariance has the following two consequences which will be used later on. 
\begin{enumerate}
    \item If $V\in \aut(P)$ vanishes on a $\sK$-subbundle $P_{\sK}\subset P$, it is identically zero on $P$.
	\item
If $X\in \gfrak$, then $[X^\sharp,V]=0$ for all $V\in \aut(P)$. 
\end{enumerate}
\end{remark}

The automorphism group $\Aut(P)$ acts on the configurations space $\cC(P)$ from Definition \ref{def: higg config} and the projection map $\pi:\cC(P)\to \J(\Sig)$ is equivariant with respect to the actions of $\Aut_0(P)$ and $\Diff_0(\Sig)$. 

\subsection{The holomorphicity condition and the Hitchin equations}
Recall from Definition \ref{def: higg config} that the configuration space of Higgs bundles $\cC(P)$ is defined by
\[ \cC(P) = \{(J,\Phi)\in \J(P)\times \Omega^1_b(P,\gfrak)~\vert~\Phi\circ J=i\Phi\}~.\] 
A pair $(J,\Phi)\in \cC(P)$ is called a \emph{Higgs bundle} if $\Phi$ is holomorphic with respect to $J.$ The holomorphicity condition will be encoded as the zero level set of the map
\begin{equation}
	\label{eq hol F map}
    F: \cC(P) \lra \Omega^2_b(P,\gfrak) ~:~ (J,\Phi)\longmapsto\dbar_J \Phi \ .
\end{equation}
Recall that tangent vectors in $\T_{(J,\Phi)}\cC(P)$ are written $(\mu,\beta,\theta)$ where  $\theta=\psi+\tfrac{1}{2i}\Phi\mu$, and $\psi$ has type $(1,0)$. Recall also that the complex structure $\I$ on $\cC(P)$ acts on $(\mu,\beta,\theta)$ by multiplication by $i.$

\begin{lemma}\label{lem:DifferentialOfF}
The map $F$ is holomorphic.
\end{lemma}

\begin{proof}
Fix $(J,\Phi)\in\cC(P)$ and a reduction $P_\sK\subset P$. By Lemma \ref{lem hol w comp conn}, $F$ is given by
\[F(J,\Phi) = \rd\Phi + [A_J,\Phi] = \left(\rd_{A_J}\Phi \right)^{1,1}~,\]
where $A_J$ is the Chern-Singer connection of $J$.
 Using Lemma \ref{lem:DifferentialOfChernSinger} we get
\begin{equation}\label{eq hol eq of F}
	\rd_{(J,\Phi)} F(\mu,\beta,\theta) = \rd \theta +\tfrac{1}{2i} [\beta+\beta^*, \Phi] + [A_J,\theta] = \overline \partial_J \psi + \tfrac{1}{2i}\left( \rd_{A_J} \Phi\mu \right)^{1,1} + \tfrac{1}{2i}[\beta,\Phi]~,
\end{equation}
where for the second equality we used the fact that this 2-form is basic, so its $(2,0)$ and $(0,2)$ parts both vanish. Hence, $\rd_{(J,\Phi)}F(i\mu,i\beta,i\theta)=i\rd_{(J,\Phi)}F(\mu,\beta,\theta)$ and we conclude that $F$ is holomorphic.
\end{proof}

To construct a moduli space of Higgs bundles on a fixed Riemann surface, it is necessary to restrict to the set of so-called \emph{polystable} $\sG$-Higgs bundles. This notion, as well as the notions of (semi-)stability come from Geometric Invariant theory and are relevant for constructing the moduli space as an algebraic variety. 
We will not need any of the technical definitions of stability for this
paper. Instead, we will use the equivalence between stability and
irreducible solutions
to the so called Hitchin equations. Before explaining this, we mention a
few important features.

Fixing, a Riemann surface $X=(\Sig,j)$ and recall there is a holomorphic projection $\pi:\cC(P)\to\J(\Sig)$. The set 
\[\bfM_X(\sG)=\{(J,\Phi)\in F^{-1}(0)~|~\pi(J,\Phi)=j\text{ and $(J,\Phi)$
is stable}\}/\cG(P)\]
has the structure of a normal quasi-projective variety (with only orbifold
singularities) called the moduli space stable $\sG$-Higgs bundles on $X$.
The smooth locus of $\Mbold_X(\G)$ 
consists of isomorphism classes of stable $\sG$-Higgs bundles whose
automorphism group is the center of $\G$. We will refer to such Higgs
bundles as \emph{regularly stable}. 

We will make use of the following property of stable $\G$-Higgs bundles, see \cite{BiswasRamananInfitesimal}. which we will use are the following. 
\begin{proposition}\label{prop stable auts}
	The automorphism group of a stable $\sG$-Higgs bundle is finite. 
 In particular, at stable points, 
    $F^{-1}(0)\subset \cC(P)$ is locally  a holomorphic submanifold.
\end{proposition}

We now describe the relation between stability and solutions to the so called Hitchin equation. This requires fixing a structure group reduction $P_\sK\subset P$ to the maximal compact subgroup $\sK<\sG.$ Consider the map $\cC(P)\to \Omega^2_b(P,\gfrak)$ defined by 
\[(J,\Phi)\mapsto F_{A_J}+\tfrac{1}{4}[\Phi,\Phi^*]~,\]
where $F_{A_J}$ is the curvature of the Chern-Singer connection and $\Phi^*$ is the adjoint of $\Phi.$ 
The equation 
\begin{equation}
 	\label{eq Hitchin eq} F_{A_J}+\tfrac{1}{4}[\Phi,\Phi^*]=0
 \end{equation} 
 is called the \emph{Hitchin equation}, and the set of $(J,\Phi)\in F^{-1}(0)$ solving equation \eqref{eq Hitchin eq} is called the space of \emph{solutions to the Hitchin equation}. 
﻿
 Recall from \eqref{eq H map} that the map $\rH:\cC(P)\to\cA(P)$ sends
 $(J,\Phi)$  to the $\sG$-connection $D_{(J,\Phi)}:=A_J+\frac{1}{2i}(\Phi-\Phi^*)$. For $(J,\Phi)$ a solution to the Hitchin, the connection $D_{(J,\Phi)}$ is flat.  Indeed, 
 \[ F_{D_{(J,\Phi)}}=F_{A_J}+\tfrac{1}{4}[\Phi,\Phi^*]+\tfrac{1}{2i}\dbar_J\Phi-\tfrac{1}{2i}\partial_{J}\Phi^*.\]

\begin{remark} 
The factor of $\tfrac{1}{4}$ appears naturally since the configuration space is defined in terms of almost complex complex structures instead of $\dbar$-operators, see Lemma \ref{lem:DifferentialOfChernSinger}.
\end{remark}

Since the Hitchin equation requires fixing a structure group reduction, the space of
solutions is preserved by the $\sK$-gauge group $\cG(P_\sK)$ but not the
$\sG$-gauge group. The following theorem of Hitchin for $\sSL_2\C$, and
Simpson in general, relates solutions to the Hitchin equation and stability. 
\begin{theorem}[\cite{selfduality,SimpsonVHS}] \label{thm Hitchin-Simpson}
	A Higgs bundle $(J,\Phi)\in F^{-1}(0)$ is polystable if and only if there exists a gauge transformation $g\in\cG(P)$ such that $g\cdot(J,\Phi)$ solves the Hitchin equations \eqref{eq Hitchin eq}. Moreover, $g$ is unique up to the actions of the $\cG(P)$-stabilizer of $(J,\Phi)$ and the compact gauge group $\cG(P_\sK)$. In particular, for $(J,\Phi)$ stable, $g$ is unique up to the action of $\cG(P_\sK)$.
\end{theorem}

As a consequence, for each structure group reduction $P_\sK\subset P$ to the maximal compact $\sK<\sG$, the moduli space of stable Higgs bundles on a fixed Riemann surface can be equivalently described as
\[\bfM_X(\sG)=\left\{(J,\Phi)\in F^{-1}(0)^{st} \mid (J,\Phi) \text{
    satisfies \eqref{eq Hitchin eq} with } \pi(J,\Phi)=j\right\}/\cG(P_\sK).\]

\subsection{The joint moduli space} 
Fix a structure group reduction $P_\sK\subset P$. Consider the map  
$F$ from \eqref{eq hol F map}, and denote the set of stable (resp.\
regularly stable) Higgs bundles by $F^{-1}(0)^{st}\subset \cC(P)$
(resp.\ $F^{-1}(0)^{rs}\subset \cC(P)$). 
The joint moduli space of stable $\sG$-Higgs bundles is the space
\[\bfM(\sG)=F^{-1}(0)^{st}/\Aut_0(P)=\{(J,\Phi)\in F^{-1}(0)^{st} \mid
(J,\Phi)  \text{ satisfies \eqref{eq Hitchin eq}}\}/\Aut_0(P_\sK).\]
There is a natural projection to $\pi:\bfM(\sG)\to \bfT(\Sig)$ to the Teichm\"uller space of $\Sigma.$ Moreover, the mapping class group $\sMod(\Sig)$, identified with $\Aut(P)/\Aut_0(P)$, acts naturally on $\bfM(\sG)$ covering its standard action on $\bfT(\Sig)$.
Recall from \S \ref{sec: omega_s and h_s} that $\cC(P)$ has a 1-parameter family of closed 2-forms $\{\omega_s\}_{s\in \mathbb R}$ and an associated family of hermitian forms $\{h_s\}_{s\in \mathbb R}$. Finally, recall the holomorphic bundle map $\Theta$ from \S \ref{section hol bundle map}.

\begin{theorem}\label{theo:ModuliSpace}
The joint moduli space $\bfM(\G)$ of stable $\G$-Higgs bundles on $\Sig$ is an orbifold equipped with an integrable complex structure $\I$, a holomorphic submersion $\pi:\bfM(\sG)\to \bfT(\Sig)$ to the Teichm\"uller space of $\Sig$ and a holomorphic bundle map 
\[\Theta: \pi^*\T\bfT(\Sig)\lra \ker(\rd\pi).\]
In addition, $\bfM(\sG)$ is equipped with a  $1$-parameter family
 $\{\omega_s\}_{s \in \R}$ of closed $2$-forms which are compatible with $\I$, and the associated hermitian forms $h_s$ satisfy the following:
\begin{enumerate}
	\item  The open sets 
	\[\cU_s :=\{ p\in \bfM(\sG) ~ \vert ~ (h_s)_p>0\}\]
	define an increasing exhaustion of $\bfM(\sG)$ as $s$ goes to $+\infty$.
	\item The restriction of $h_s$ to $\pi^{-1}(j)$ is independent of $s$,
        positive and $\pi^{-1}(j)$ is isomorphic as a K\"ahler manifold to
        the moduli space of stable Higgs bundles on the fixed Riemann
        surface $X=(\Sigma,j)$.
\end{enumerate} 
Moreover, all of these structures are invariant under the action of $\sMod(\Sig)$.
\end{theorem}
\begin{remark}
	We note also that $(2,0)$-part of the pullback of the
    Atiyah-Bott-Goldman form defines a closed holomorphic 2-form on $\bfM(\sG)$.
\end{remark}
The rest of this section is devoted to setting up and proving the
Theorem \ref{theo:ModuliSpace}.
The standard approach to prove such a theorem is to construct a local slice for the action.
First, following Atiyah-Bott \cite[p. 577]{AtiyahBott:82}, 
let us denote $\overline \Aut_0(P):=\Aut_0(P)/Z(\G)$, where the inclusion
$Z(\G)\hookrightarrow  \Aut_0(P)$ is by constant gauge transformations in
the center $Z(\G)$.  Then $\overline \Aut_0(P)$ still acts on $\Ccal(P)$.
Now,  given a point $p$ in $\bfM(\sG)$, and a lift $x\in F^{-1}(0)^{st}$
which solves the Hitchin equations \eqref{eq Hitchin eq}, we seek a holomorphic submanifold $S_x$ of $F^{-1}(0)^{st}$ through $x$ such that the orbit map
\[\overline\Aut_0(P) \times S_x \longrightarrow F^{-1}(0)^{st}\]
is a diffeomorphism onto an open neighborhood of $x$ in $F^{-1}(0)^{st}$
(or more generally a finite ramified cover). 
If the slice is natural, in the sense that $S_y=g\cdot S_x$ for $g\in
\overline\Aut_0(P_\sK)$ and $y=g\cdot x$, then the slices define a holomorphic (orbifold) atlas on $\bfM(\sG)$. 

By Lemma \ref{lem sufficient large s}, for sufficiently large $s$ the hermitian form $h_s$ is positive at $x$. 
We use this metric to complete all relevant Fr\'echet spaces to Hilbert manifolds using the Sobolev topology. 
The slice is constructed as follows: the deformation complex arising from the infinitesimal action of $\Aut_0(P)$ on $\cC(P)$ and the variation of the holomorphicity equation $F$ turns out to be elliptic, and so has a finite dimensional space of harmonics. 
Applying the exponential map from Definition \ref{def Higgs exp} to the
space of harmonics,  and an implicit function theorem to project back to
$F^{-1}(0)^{st}$, one obtains a complex submanifold $\cS_x$ of
$F^{-1}(0)^{st}$. When $x$ is regularly stable, this is shown to be a local slice using properness and
freeness of the action of $\overline\Aut_0(P)$. For $x$ stable but not
regularly stable, the stabilizer of $x$ in $\overline\Aut_0(P)$ is finite and preserves $\cS_x$, yielding an orbifold structure on $\bfM(\G)$.

Surprisingly, the harmonics at $x$, defined with respect to the metric $h_s$ for $s$ sufficiently large, are independent of the parameter $s$. 
Using the explicit form of the harmonics, we show that $\rd\omega_s$
vanishes on the tangent spaces to the slices for all $s$. For sufficiently
large $s$, this implies that the real part of $h_s$ defines
a K\"ahler metric in a neighborhood of $p\in\bfM(\sG)$.

\subsection{Sobolev completion}

In order to work with Hilbert manifolds, it is classical to complete the different spaces. Fixing a background metric on $\Sig$ and structure group reduction $P_\sK\subset\ P$ define the Sobolev $W^{k,2}$-norm on spaces of tensors valued in associated bundles (where the $W^{k,2}$-norm is defined using the $L^2$-norm of the first $k$ derivatives). This norm can be used to complete $\cC(P)$ into a Hilbert manifold $\cC(P)^k$. Compactness of $\Sig$ implies that the $W^{k,2}$-topology is independent of the choices. 

Similarly, $\Aut_0(P)$ has a Sobolev completion $\Aut_0(P)^{k+1}$ and the
different maps (exponential map and holomorphicity) extend smoothly to those completions and their extension will be denoted with an exponent $k$. We will refer to the subspaces $\cC(P)$ and $\Aut_0(P)$ of the respective completions as the smooth tensors (or symmetries). We will also assume $k$ large enough to ensure our tensors to be $C^2$ and that $W^{k,2}$ is closed under taking products.

Let us highlight two classical issues coming with this completion:
\begin{enumerate}
	\item The Hilbert topology and the Fr\'echet topology on $\cC(P)$ are different.
	\item Even if the action of $\Aut_0(P)$ on $\cC(P)$ is smooth, that
        of $\Aut_0(P)^{k+1}$ on $\cC(P)^k$ is not.
\end{enumerate}

The first issue will be fixed using elliptic regularity: the equation $F(x)=0$ is elliptic and thus the Sobolev and smooth topology will coincide on $F^{-1}(0)$.
For the second issue, even is the action of $\Aut_0(P)^{k+1}$ on $\cC(P)^k$ is not smooth, given  $x\in\cC(P)^k$ which is a smooth tensor, the orbit map 
$$\Aut_0(P)^{k+1}  \longrightarrow  \cC(P)^k  ~:~ g  \longmapsto  gx $$
is smooth. We now gather some properties of the action of $\Aut_0(P)^{k+1}$ on $\cC(P)^k$.

\begin{proposition}\label{prop:ProperFreeAction}
The action of $\Aut_0(P)^{k+1}$ on $\cC(P)^k$ is proper, and acts with
    finite stabilizer on the space of stable Higgs bundles. The stabilizer
    is $Z(\G)$ for regularly stable Higgs bundles.
    Moreover, given a smooth $x\in \cC(P)$ and an element $g\in \Aut_0(P)^{k+1}$, if $gx\in \cC(P)$, then $g\in Aut_0(P)$.
\end{proposition}

\begin{proof}
The proof will follow from the same statement for the action of $\Diff^{k+1}_0(\Sig)$ on $\J^k(\Sig)$ proved by Tromba \cite{Tromba} as well as for the action of $\cG^{k+1}(P)$ on $1$-forms proved by Freed-Uhlenbeck \cite{FreedUhlenbeck}.

Indeed, for properness, let $(x_n)_{n\in\mathbb N}$ and $(g_n)_{n\in\mathbb N}$ be sequences in $\cC(P)^k$ and $\Aut_0(P)^{k+1}$ respectively such that $(x_n)_{n\in\mathbb N}$ and $(g_nx_n)_{n\in\mathbb N}$ converge. Since the projection $\pi: \cC(P)^k \to \cJ^k(\Sig)$ is equivariant under the map $p: \Aut^{k+1}(P) \to \Diff^{k+1}(\Sig)$, we get that $(\pi(x_n))_{n\in \mathbb N}$ and $(p(g_n)\pi(x_n))_{n\in \mathbb N}$ both converge in $\cJ^{k+1}(\Sig)$. By the properness of the action of $\Diff_0^{k+1}(\Sig)$ on $\cJ^k(\Sig)$ (see \cite[Theorem 2.3.1]{Tromba}), up to extracting, we have that $(p(g_n))_{n\in \mathbb N}$ converges in $\Diff_0^{k+1}(\Sig)$. 

For any $n$, let $h_n$ be a lift of $p(g_n)^{-1}$ to $\Aut_0^{k+1}(P)$ such that $(h_n)_{n\in \mathbb N}$ converges. The sequence $(h_nx_n)_{n\in \mathbb N}$ and $(h_ng_nx_n)_{n\in \mathbb N}$ are then converging sequences in a fixed fiber of $\pi$, and one can apply Freed-Uhlenbeck properness \cite[Proposition A.5]{FreedUhlenbeck}: up to extracting again, the sequence $(h_ng_n)_{n\in \mathbb N}$ converges in $\cG^{k+1}(P)$. Since $(h_n)_{n\in \mathbb N}$ already converges, $(g_n)_{n\in \mathbb N}$ converges in $\Aut_0^{k+1}(P)$. This proves properness.

\medskip

To show that if $x$ and $gx$ are both in $\cC(P)$ then $g$ must be in
    $\Aut_0(P)$ we follow the same trick: projecting to $\cJ(\Sig)$ and
    applying \cite[Remark 2.4.3]{Tromba} we get that $p(g)$ is in
    $\Diff_0(\Sig)$. Taking a lift $h$ of $p(g)$ in $\Aut_0(P)$, we get
    that $hg$ is a gauge transformation 
    in $\cG^{k+1}(P)$ mapping a smooth element $x\in \cC(P)$ to a smooth
    element $hgx\in \cC(P)$. Hence, $hg \in \cG(P)$ by \cite[Proposition
    A.5]{FreedUhlenbeck}, and so $g\in \Aut_0(P)$.

\medskip

Finally, freeness of the action of $\overline\Aut_0(P)^{k+1}$ on regularly stable
    Higgs bundles follow from freeness of the action of
    $\Diff^{k+1}_0(\Sig)$ on $\cJ^k(\Sig)$ (see \cite[Theorem
    2.2.1]{Tromba}) and freeness of the action of 
    the gauge group $\Gcal(P)/Z(\G)$ on regularly stable Higgs bundles. The finite stabilizer property is proved analogously. 
\end{proof}

\subsection{Deformation complex} \label{sec:deformation}
Fix a stable Higgs bundle $(J,\Phi)\in F^{-1}(0)^{st}$. By stability, up to acting by an element in $\Aut_0(P)$, we can furthermore assume that $(J,\Phi)$ satisfies the Hitchin equations \eqref{eq Hitchin eq}. The deformation complex in our situation is given by 
\begin{equation}
	\label{eq joint deformation
    complex}\xymatrix{(\B^\bullet,\delta_\B^\bullet)
    :&0\ar[r]&\aut(P)^k\ar[r]^{\delta_\B^0\quad }\ar[r]&
    \T_{(J,\Phi)}\cC(P)^k\ar[r]^{\delta_\B^1}\ar[r]&
    \Omega^2_b(P,\gfrak)^{k-1}\ar[r]&  0},
\end{equation}
where the map $\delta_\B^0$ is the derivative at the identity of the action
of $\Aut_0(P)^{k+1}$ on $\cC(P)^k$, i.e., taking the Lie derivative of $(J,\Phi)$, and $\delta_\B^1$ is the derivative of $F^k$ at $0.$ 

Consider the $\Aut_0(P)^{k+1}$-equivariant projection $\pi^k:\cC(P)^k\to\J(\Sig)^k$. For $j=\pi(J,\Phi)$,  we get a 
subcomplex $(\A^\bullet,\delta_\A^\bullet)$ isomorphic to the deformation complex for Higgs bundles on the fixed Riemann surface $(\Sig,j)$ and a quotient deformation complex isomorphic to the deformation complex for $j\in\J(\Sig).$ We thus obtain an exact sequence of complexes of Hilbert spaces
\[0^\bullet \lra (\A^\bullet,\delta_\A^\bullet) \lra
(\B^\bullet,\delta_\B^\bullet) \lra (\rmC^\bullet,\delta_\rmC^\bullet) \lra 0^\bullet~.\]
Denote $(\pi^k)^{-1}(j)$ by $\cC(P)^k_j$, the complexes $(A^\bullet,\delta_\A^\bullet)$ and $(\rmC^\bullet,\delta_C^\bullet)$ are given by
\[\xymatrix@R=0em{(\A^\bullet,\delta_\A^\bullet):&0\ar[r]&\Omega^0_b(P,\gfrak)^k\ar[r]^{\delta_\A^0\quad }&\T_{(J,\Phi)}\cC(P)^k_j\ar[r]^{\delta_\A^1}&\Omega^2_b(P,\gfrak)^{k-1}\ar[r]&0\\
(\rmC^\bullet,\delta_\rmC^\bullet): &0\ar[r]&\Omega^0(\T\Sig)^k\ar[r]^{\delta_{\rmC}^0}&\T_j\J(\Sig)^k\ar[r]&0&}.\]
The maps $\delta_\A^0$ and $\delta_\rmC^0$ are the derivative at the identity of the actions of the gauge group $\cG(P)$ and the diffeomorphism group, respectively, and the map $\delta_A^1$ is the differential of the holomorphicity condition. 

We now give a more explicit description of the above maps, which will serve
to simplify  the computations of the harmonics in the next section. Recall that we are assuming $(J,\Phi)\in F^{-1}(0)^{st}$ solves the Hitchin equations. 
Following Simpson, we introduce the following operators acting on
$P(\gfrak)$-valued forms on $\Sig$, 
\begin{equation} \label{eqn:simpson-ops}
    \xymatrix{D''=\dbar_{A_J}+\tfrac{1}{2i}\Phi&\text{and}&  D'=\partial_{A_J}+(\tfrac{1}{2i}\Phi)^\ast}.
\end{equation}
The holomorphicity condition implies $(D'')^2=0$ and $(D')^2=0$, and the assumption that $(J,\Phi)$ solves the Hitchin equations \eqref{eq Hitchin eq} for the reduction $P_\sK$ implies $D''D'=-D'D''$.

The maps $\delta_\A^0$ and $\delta_\A^1$ are defined by applying the operator $D''$ to the appropriate spaces.  
In terms of the operators $D'$ and $D''$, the map $\delta_\B^1$ is described as follows.
\begin{lemma}\label{lem delta 1 B}
	For $(\mu,\beta,\theta)\in \T_{(J,\Phi)}\cC(P)^k$ with $\theta=\psi+\frac{1}{2i}\Phi\mu$, we have
	\[\delta_\B^1(\mu,\beta,\theta)=D''(\beta,\psi)+D'\left(\tfrac{1}{2i}\Phi\mu\right).\]
\end{lemma}
\begin{proof}
	The map $\delta_\B^1$ is the differential of the map $F$ at $0$. By \eqref{eq hol eq of F}, we have 
	\[\delta_\B^1(\mu,\beta,\theta)=d \theta +\tfrac{1}{2i} [\beta+\beta^*, \Phi] + [A_J,\theta]= \dbar_J \psi +\left[\tfrac{1}{2i}\Phi,\beta\right]+\partial_{A_J}\left(\tfrac{1}{2i}\Phi\mu\right), \]
	where we used that the $(0,2)$ and $(2,0)$ parts vanish. The first two terms are $D''(\beta,\psi)$ and third term is $D'(\frac{1}{2i}\Phi\mu)$.
\end{proof}
We now give an explicit description of $\delta_\B^0.$ For $V\in \aut(P)$,
we will denote the contraction of $V$ with the Chern-Singer connection by $\eta_V=A_J(V)$.

\begin{lemma}  
\label{lem:d0-simplified}
    Suppose $(J,\Phi)$ is a stable Higgs bundle satisfying \eqref{eq Hitchin eq} such that $\pi(J,\Phi)=j$. Let $V\in\aut(P)^k$ and denote the projected vector field by $v\in \Omega^0(\T^{(1,0)}_j\Sig)^k$.
    Writing $\delta_B^0(V)=(\mu,\beta,\theta)$, we have
    $$
    \xymatrix{\mu= 2i \dbar_j v&\text{and}&(\beta,\theta)=2i D''(\eta_V)+D'(\Phi(V))+D''(\Phi(V))}.
    $$
\end{lemma}
\begin{proof}

	We first compute the action of $\aut(P)^k$ on the tangent space $\T_J\J(P)^k$.  
	
    Denote this by 
   \[\aut(P)^k\lra \T\J(P)^k ~:~ V\mapsto V^\sharp_J = L_V\circ J-J\circ L_V\in \Omega^0_b(\End(\T P))\ ,
    \]
    where $L_V$ denotes the Lie derivative.

For general vector field $W\in\Omega^0(\T P)$, denote the vertical and
    horizontal parts with respect to the Chern-Singer connection by $W_v$ and $W_h$.  
By Remark \ref{rem:vertical}, $V^\sharp_J(W)= V^\sharp_J(W_h)$.
Hence,
\begin{align*}
	V^\sharp_J(W)&=[V,jW_h]-J([V,W_h])\\ 
	&=[\eta_V,jW_h]-J([\eta_V,W_h])+[V_h,jW_h]-J([V_h,W_h])
	% \\
	% &=[\eta_V,jW_h]_v-i[\eta_V,W_h]_v+[V_h,jW_h]_v-i[V_h,W_h]_v+[\eta_V,jW_h]_h-j([\eta_V,W_h])_h+[V_h,jW_h]_h-j([V_h,W_h]_h)~.
\end{align*}
The decomposition of the image of $V^\sharp_J$ into its vertical and horizontal parts corresponds to the decomposition of the tangent vector $M=(m,\beta)$ from \eqref{eq decomp of M wrt J-comp B}, where $\mu$ is the $(0,1)$-part of $m$. 
Hence, we have 
\begin{align*}
	\beta(W)&=[\eta_V,jW_h]_v-i[\eta_V,W_h]_v+[V_h,jW_h]_v-i[V_h,W_h]_v\\
	&=2i ([\eta_V,W_h^{(0,1)_J}]_v+[V_h,W_h^{(0,1)_J}]_v)
\end{align*}
Using $F_{A_J}(V,W)=-[V_h,W_h]_v$ has type $(1,1)$ and the Hitchin equations, we have 
$$
	\beta= 2i\left(\dbar_J(\eta_V)+\iota_{V^{(1,0)_J}}F_{A_J}\right)
	 = 2i\left(\dbar_J(\eta_V)-\tfrac{1}{4}\left[\Phi(V),\Phi^*\right]\right)
	= 2i\dbar_J(\eta_V)+\left[\left(\tfrac{1}{2i}\Phi\right)^*,\Phi(V)\right].
$$
The horizontal piece depends only on the projected vector fields $v',w'$ of $V,W$. Thus, 
\[m(w')=[v',jw']-j([v',w']).\]
Denoting the $(1,0)$-part of $v'$ by $v$, the $(0,1)$-part of $m$ is given by 
$\mu=2i\dbar_j v$.
Finally we compute the Higgs field term. For a vector field $W\in \Omega^0(\T P)$ we have
\[L_V\Phi(W)= d\Phi(V,W)+W \Phi(V).\]
 Since $\Phi$ is basic, the (2,0) part of $d\Phi+[A_J,\Phi]$ vanishes. By the holomorphicity condition, we have

\begin{align*}
	L_V\Phi(W)&=-[A_J,\Phi](V,W)+W\Phi(V)\\
	&= [\Phi(W), \eta_V]+[A_J(W),\Phi(V)]+W\Phi(V)\\ 
	&= \left([\Phi,\eta_V]+d_{A_J}\Phi(V)\right)(W) 
\end{align*}
Adding the expressions for $\beta$ and $L_V\Phi$ and using $[\Phi,\Phi(V)]=0$ gives the desired result.
\end{proof}

\begin{remark}\label{rem D''(Phi(V))}
	The $(0,1)$-part of $L_V\Phi$ is $D''(\Phi(V))$. This is indeed $\frac{1}{2i}\Phi\mu$ since, using the notation above, 
	\[D''(\Phi(V))=\dbar_J(\Phi(V))= \dbar_J(\Phi(v)) =\Phi(\dbar_j v)=\tfrac{1}{2i}\Phi(\mu),\]
	where we used the expression for $v$ from Lemma \ref{lem:d0-simplified}, and that $\Phi$ is basic and holomorphic.
\end{remark}
\begin{lemma}
	The sequence $(\B^\bullet,\delta_\B^\bullet)$ is a complex, i.e., $\delta_\B^1\circ\delta_\B^0=0.$
\end{lemma}
\begin{proof}
	Let $V\in \aut(P)$. Then, by the above remark,   
	\[\delta_\B^1(\delta_\B^0(V))=D''(2i D''(\eta_V)+D'(\Phi(V))+D''(\Phi(V))) +D'(D''\Phi(V)).\]
	The result follows from the fact that $(D'')^2=0$ and $D'D''(\Phi(V))=-D''D'(\Phi(V)).$
\end{proof}

\begin{lemma}
For $(J,\Phi)$ stable, the sequence 
\[0\lra H^1(\A^\bullet)\lra H^1(\B^\bullet)\lra H^1(\rmC^\bullet)\lra 0\]
is exact and all other cohomology groups vanish. 
\end{lemma}

\begin{proof}
Since $\Diff_0(\Sig)$ acts freely on $\J(\Sig)$, the cohomology group
    $H^0(\rmC^\bullet)$ is vanishes. The assumption that $(J,\Phi)$ is
    stable and $\sG$ is semisimple implies that $H^0(\A^\bullet)=0$ and, 
    by Serre duality,
    $H^2(\A^\bullet)=0$, see Proposition \ref{prop stable auts}.
    The statement now follows from the associated long exact sequence in cohomology.
\end{proof}
\subsection{Harmonic and semiharmonic representatives}
Recall that $(J,\Phi)$ is a fixed stable Higgs bundle which solves the Hitchin equation. 
Let $s \in \R$ be large enough such that the hermitian form $h_s$ is positive on $\T_{(J,\Phi)}\cC(P)$, and use it to define the different adjoints in the above diagram. Recall that, at $(J,\Phi)$, $h_s$ involves the hyperbolic metric on $\Sig$ which uniformizes the induced Riemann surface. 
Denote the adjoints of $\delta_\B^0$ and $\delta_\B^1$ with respect to
$h_s$ by  $(\delta_\B^0)^*$ and $(\delta_\B^1)^*$, respectively. The
adjoints of the operators $D'$ and $D''$ from \eqref{eqn:simpson-ops}
satisfy the K\"ahler identities (see Appendix \ref{sec:app-kahler}):
\begin{equation} \label{eqn:kahler-simpson}
    \xymatrix{(D'')^\ast=-i[\Lambda,D']\quad&\text{and}&\quad (D')^\ast=i[\Lambda,D'']},
\end{equation}
where $\Lambda$ 
denotes the contraction with the area form of the hyperbolic metric on $X$.

Standard Hodge theory yields the following orthogonal decomposition
\[ \T_x \cC(P)^k = \rIm(\delta^0_\B) \oplus \cH^1(\B^\bullet) \oplus \rIm((\delta_\B^1)^*)~,\]
where $\cH^1(\B^\bullet)= \ker(\delta_\B^1)\cap \ker((\delta^0_\B)^*)$ is
the \emph{space of harmonics} which is 
naturally identified with $H^1(\B^\bullet)$. 
The harmonics are given by the following proposition. 

\begin{proposition} \label{prop:harmonics}
  Let $(J,\Phi)$ be a stable Higgs bundle satisfying \eqref{eq Hitchin
    eq}, and $s\in\R$ be large enough so that the hermitian form
    $h_s$ is positive definite at $(J,\Phi)$. A tangent vector
    $(\mu,\beta,\theta)\in \T_{(J,\Phi)} \cC(P)$  with $\theta=\psi+\frac{1}{2i}\Phi\mu$ is harmonic 
    with respect to $h_s$ (that is, it is in $\cH^1(\B^\bullet)$)
    if and only if it satisfies 
\begin{enumerate}
    \item $D''(\beta,\psi)+D'(\frac{1}{2i}\Phi\mu)=0$;
    \item $D'(\beta,\psi)=0$;
	\item $\dbar^\ast\mu=0$.
\end{enumerate}
In particular, the space of harmonics is independent of the parameter $s$ and consists of smooth tensors, i.e., vectors tangent to $\cC(P)$ in the Sobolev completion. 
\end{proposition}
\begin{proof}
Item (1) is the property of being in the kernel of $\delta_\B^1$. 
For (2), we compute the terms in the expression \eqref{eq hs metric}. 
Let $(\mu,\beta,\psi)\in\ker(\delta_\B^1)$. As above, let $\eta_V=A_J(V)$ for any $V\in \aut(P)$ and $(\mu_1,\beta_1,\theta_1)=\delta_B^0(V)$.
   From Lemma \ref{lem:d0-simplified}, Remark \ref{rem D''(Phi(V))} and the K\"ahler identities
    \eqref{eqn:kahler-simpson} we have 
    \begin{align*}
        \langle (\beta_1,\psi_1),(\beta,\psi)\rangle &=
        \langle 2i\, D''\eta_V,(\beta,\psi)\rangle +
        \langle D'(\Phi(V)),(\beta,\psi)\rangle  \\
        &=-2i\langle\eta_V, i\Lambda D'(\beta,\psi)\rangle+\langle\Phi(V),
        i\Lambda D''(\beta,\psi)\rangle \\
        &=-\langle\eta_V, 2\Lambda D'(\beta,\psi)\rangle-\langle\Phi(V),
        i\Lambda D'\left(\tfrac{1}{2i}\Phi\mu\right)\rangle~,
    \end{align*}
    where we have  the definition of  $\ker(\delta_\B^1)$. 
    On the other hand
    \[ -\tfrac{1}{4}\langle \Phi\mu_1,\Phi\mu\rangle ~=~ -\tfrac{i}{2}\langle
        D''(\Phi(V)), \Phi\mu\rangle~=~-\langle
        D''(\Phi(V)),\tfrac{1}{2i}\Phi\mu\rangle~= ~+\langle \Phi(V), i\Lambda D'\left(\tfrac{1}{2i}\Phi\mu\right)\rangle.\]
    It follows that
    $$
    \langle \delta_B^0(V),(\mu,\beta,\theta)\rangle_{h_s}=
-\langle\eta_V, 2\Lambda D'(\beta,\psi)\rangle+2is\langle
    v,\dbar^\ast\mu\rangle_{WP} \ .
    $$
    Since $V\in \aut(P)$ was arbitrary, the result follows.
\end{proof}
We have the following corollary.
\begin{corollary}\label{c:Closed}
For any $s$ the form $\rd\omega_s$ vanishes on $\cH^1(\B^\bullet)$.
\end{corollary}
\begin{proof} 
	Recall from \S \ref{sec: omega_s and h_s} that $\omega_s=s\cdot
    \omega_{WP}+\omega_0$, where $\omega_0$ is the closed 2-form from
    Definition \ref{def:omega0} and $\omega_{WP}$ is the Weil-Peterson form
    on $\J(\Sig)$. The result follows from the fact that the Weil-Petersson
    symplectic form is closed on the space of harmonic Beltrami
    differentials, see \cite{Ahlfors:61,Wolpert:86}.
\end{proof}

The hermitian form $h_0$ from Definition \ref{def:omega0} does not make use of a conformal metric on $\Sig$, and so we have no notion of harmonic representatives for Beltrami
differentials. This motivates the following definition.
\begin{definition} \label{def:semiharmonic}
    Let $(J,\Phi)$ be a Higgs bundle satisfying \eqref{eq Hitchin
    eq}.  We say a vector
    $(\mu,\beta,\psi)\in \T_{(J,\Phi)} \cC(P)$ is \emph{semiharmonic}
    if it satisfies 
\begin{itemize}
    \item $D''(\beta,\psi)+D'\left(\tfrac{1}{2i}\Phi\mu\right)=0$,\ \  (holomorphicity condition);
    \item $D'(\beta,\psi)=0$,\ \  (gauge condition).
\end{itemize}
\noindent In other words, we impose no condition on the Beltrami differential (item
(3) in Proposition \ref{prop:harmonics}).
\end{definition}

There is an infinite dimensional
space of
semiharmonic representatives for a given cohomology
class in $H^1(\B^\bullet)$. The following result characterizes these.

\begin{lemma} \label{lem:semiharmonic}
Let $(J,\Phi)$ be a Higgs bundle which solves the Hitchin equations,  $V\in
    \aut(P)$ and set $\eta_V=A_J(V)$.  If $(\mu,\beta,\theta)$ and 
    $(\mu_1,\beta_1,\theta_1)=(\mu,\beta,\theta)+\delta_B^0(V)$
    are both semiharmonic, then $V$ is horizontal with respect to the
    Chern-Singer connection $A_J$. 
\end{lemma}

\begin{proof}
    From Lemma \ref{lem:d0-simplified}, we have $\mu_1=\mu+2i\dbar v$, and
    $$
    (\beta_1,\psi_1)=(\beta,\psi)+2i D''(\eta_V)+D'(\Phi(V))\ .
    $$
    Applying  the gauge condition to both sides, we have $D'D''(\eta_V)=0$,
    and therefore $\eta_V\in H^0(\B^\bullet)$.
    But since stable implies simple,
    $\eta_V$ must take values in the center. Since $\G$ is semisimple,
    $\eta_V=0$, and  therefore $V$ is horizontal. 
\end{proof}

An important consequence of the above is the following.
\begin{proposition} \label{prop:semiharmonic}
The difference of any two semiharmonic representatives of a given cohomology class
    $[(\mu,\beta,\theta)]$
    lies in the kernel of the hermitian form  $h_0$.
    In particular, $h_0$ is a well-defined  on
    $H^1(\B^\bullet)$.
\end{proposition}

\begin{proof}
   Suppose $(\mu,\beta,\psi)$ and $(\mu_1,\beta_1,\psi_1)$ are cohomologous, and  $(\mu_2,\beta_2,\theta_2)\in \ker \delta_\B^1$. Then, by the previous lemma  we have
   \[\langle (\beta_1,\psi_1)-(\beta,\psi), (\beta_2,\psi_2)\rangle 
        ~=~ \langle D'(\Phi(V)),(\beta_2,\psi_2)\rangle. \]
        The K\"ahler identities and the holomorphicity condition twice gives
        \[\langle D'(\Phi(V)),(\beta_2,\psi_2)\rangle~= ~ \langle \Phi(V), i\Lambda D''(\beta_2,\psi_2)\rangle ~=~ \langle D''( \Phi(V)), \tfrac{1}{2i}\Phi\mu_2\rangle.\]
        Using Remark \ref{rem D''(Phi(V))}, we have 
        \[\langle (\beta_1,\psi_1)-(\beta,\psi), (\beta_2,\psi_2)\rangle~=~ \langle \tfrac{1}{2i}\Phi(\mu_1-\mu), \tfrac{1}{2i}\Phi\mu_2\rangle~.\]
    Hence,
    $
    \langle (\mu_1,\beta_1,\theta_1)-(\mu,\beta,\theta),
    (\mu_2,\beta_2,\theta_2)\rangle_{h_0} =0
    $.
\end{proof}

\subsection{Constructing a local slice}

\begin{proposition}\label{prop:slice}
Fix $p\in\bfM(\sG)$ and let $x\in \cC(P)$ be a lift of $p$ that satisfies the Hitchin equation \eqref{eq Hitchin eq}. Let $\B^\bullet$ be the associated deformation complex from \eqref{eq joint deformation complex}. Then there exists a complex submanifold $\mathcal S_x$ of  $F^{-1}(0)$ which passes through $x$ and satisfies the following:
\begin{enumerate}
	\item $\T_x\mathcal S_x=\cH^1(\B^\bullet)$ ;
	\item $\mathcal S_x$ is biholomorphic to a neighborhood of $0$ in $\cH^1(\B^\bullet)$;
	\item if $g\in \overline\Aut_0(P_\sK)$ and $y=g x$
        is another lift of $p$ solving the Hitchin equations, then
        $\cS_y=g\cS_x$; and
	\item if $x$ is regularly stable, then the action map 
        $\overline\Aut_0(P) \times \mathcal S_x \to F^{-1}(0)$
        is a diffeomorphism onto an open neighborhood of $x$. 
\end{enumerate}
\end{proposition}

\begin{proof} Let $s\in\R$ be large enough such that $h_s$ is positive definite on $\T_x\cC(P)^k$ and consider the Hodge decomposition
\[\T_x\cC(P)^k = \rIm(\delta_B^0)\oplus \cH^1(\B^\bullet)\oplus \rIm((\delta_B^1)^*)~.\]
This decomposition is orthogonal with respect to $h_s$ and $\I$-invariant since both $F^k$ and the orbit $\Aut_0(P)^{k+1} \cdotp x$ are holomorphic. Given neighborhoods $U$ and $V^k$ of the origin in $\cH^1(\B^\bullet)$ and $\rIm((\delta_B^1)^*)$, respectively, we consider the holomorphic map
\[\Psi^k:= F^k\circ \left({\exp^k_x}_{\big\vert U\oplus V^s}\right) : U \oplus V^k \longrightarrow \Omega^2_b(P,\mathfrak{g})^{k-1}~.\]
Using that $\rd_0 \exp^k_x$ is the identity, we obtain that $\rd_{(0,0)}\Psi^k$ is the restriction of $\delta_B^1$ to $\cH^1(\A^\bullet) \oplus \rIm((\delta_B^1)^*)$. In particular, its restriction to the second factor is a linear isomorphism. By the implicit function theorem for complex Hilbert manifolds, up to shrinking $U$ if necessary, there exists a unique holomorphic map $\varphi^k : U \to V^k$ such that for $(u,v)\in U\times V^k$ we have $\Psi^k(u,v)=0$ if and only if $v=\varphi^k(u)$.

\begin{lemma}
The map $\varphi^k$ described above takes value in $\T_x \cC(P)$ (in particular it is independent of $k$ and will be denoted $\varphi$). Moreover, it satisfies $\rd_0\varphi=0$
\end{lemma}

\begin{proof}
For the first item, observe that doing the same construction with the Sobolev completion $W^{k+1,2}$ instead of $W^{k,2}$ yields a map $\varphi^{k+1}: U \to V^{k+1}$ such that for $(u,v)\in U\times V^{k+1}$ we have $\Psi^{k+1}(u,v)=0$ if and only if $v=\varphi^{k+1}(u)$.
Nevertheless, $\cC(P)^{k+1}$ is contained in $\cC(P)^k$, so $V^{k+1} = V^k\cap \T_x \cC(P)^{k+1}$ and $\Psi^{k+1}$ is the restriction of $\Psi^k$ to $U\oplus V^{k+1}$. 
Hence, for all $u\in U$ we have 
\[\Psi^{k+1}(u,\varphi^{k+1}(u))=\Psi^k(u,\varphi^{k+1}(u))~,\]
so the uniqueness part of the implicit function theorem gives $\varphi^{k+1}(u)=\varphi^k(u)$. In particular, $\varphi^k$ takes value in $\bigcap_{\alpha\in \mathbb N} V^{k+\alpha}$ which is equal to $V^k\cap \T_x\cC(P)$.

To compute $\rd_0\varphi$, we differentiate the equation $\Psi^k(u,\varphi(u))=0$. It gives
\[\rd^1_{(0,0)}\Psi^k + \rd^2_{(0,0)}\Psi^k \circ \rd_0\varphi = 0\]
where $\rd^1\Psi^k$ and $\rd^2\Psi^k$ are respectively the differential of $\Psi^k$ with respect to the first and second variable, and so their value at $(0,0)$ coincides with the restriction of $\delta_B^1$ to $\cH^1(\A^\bullet)$ and $\rIm((\delta_B^1)^*)$ respectively. Since $\cH^1(\A^\bullet)$ is contained in $\ker(\delta_B^1)$, the above equation gives $\rd_0\varphi =0$.
\end{proof}

Consider now the immersion
\begin{equation}\label{eq sigma}
	\sigma:  U\subset \cH^1(\B^\bullet)  \longrightarrow  \cC(P)^k  ~:~
 u  \longmapsto  \exp_x(u,\varphi(u))  \ ,
\end{equation}
and let $\cS_x=\sigma (U)$. The above lemma and discussion imply that
    $\cS_x$ is a holomorphic submanifold of $\cC(P)$ contained in
    $F^{-1}(0)$ and passing through $x$ with $\T_x \cS_x =
    \cH^1(\A^\bullet)$. Furthermore, if $y=gx$ is another
    lift of $p$ solving the Hitchin equation, then $g\in \Aut_0(P_\K)$ and
    so preserves the harmonic metric. By the uniqueness part of the implicit
    function theorem, we easily obtain that $\sigma_y=g\sigma_x$, and so $\cS_y=g\cS_x$.

\begin{lemma}
If $x$ is regularly stable, up to shrinking $\cS_x$ if necessary, the
    action of $\overline\Aut_0(P)^{k+1}$ on $\cS_x$ defines a map
$$
\wp^k :  \overline\Aut_0(P)^{k+1} \times \cS_x  \longrightarrow  (F^k)^{-1}(0) ~:~
 (g,y)  \longmapsto  g\cdotp y  \ ,
 $$
which is a diffeomorphism onto an open neighborhood of $x$ in $(F^k)^{-1}(0)$.
\end{lemma}

\begin{proof}
First note that since $\cS_x$ is contained in $\cC(P)$, the map $\wp^k$ is smooth. 
    Moreover, its differential at $(\Id,x)$ is given by 
$$(\delta_B^0,\iota) : \aut(P)^k \times \cH^{1}(\A^\bullet) \to \ker(\delta_B^1)~,$$ 
where $\iota$ is the natural injection. Using the decomposition
    $\ker(\delta_B^1)=\rIm(\delta_B^0)\oplus \cH^1(\B^\bullet)$, we get
    that $D_{(\Id,x)} \wp^k$ is a linear isomorphism.
    By the inverse function theorem, up to shrinking $\cS_x$, 
    there exists a neighborhood $W$ of $\Id$ in $\overline\Aut_0(P)^{k+1}$
    such that $\wp^k$ restricts to a diffeomorphism from $W\times \cS_x$ onto an open neighborhood of $x$ in $(F^k)^{-1}(0)$. 

The only thing that remains to be proved is that, up to shrinking $\cS_x$
    if necessary, then for any $y\in \cS_x$ its $\overline\Aut_0(P)^{k+1}$-orbit
    does not intersect $\cS_x$ at a different point. Assume to the contrary
    that this is false. Then there are sequences
    $(y_n)_{n\in\mathbb N}$ in $\cS_x$ and $(g_n)_{n\in \mathbb N}$ in
    $\overline\Aut_0(P)^{k+1}$ such that $(y_n)_{n\in\mathbb N}$ converges to $x$
    and $(g_ny_n)_{n\in\mathbb N}$ is a sequence of $\cS_x$ converging to
    $x$. By Proposition \ref{prop:ProperFreeAction}, the action of
    $\overline\Aut_0(P)^{k+1}$ on $\cC(P)^k$ is proper, so the sequence
    $(g_n)_{n\in\mathbb N}$ subconverges to some $g\in \overline\Aut_0(P)^{k+1}$. In
    particular, $gx=x$, and so $g=\Id$ since the action is free
    on regularly stable Higgs bundles (Proposition \ref{prop:ProperFreeAction}).
    But this implies that $g_n\in W$ for $n$ large enough, contradicting
    the fact that $\wp^k$ is a diffeomorphism.
\end{proof}

The proof  now follows from Proposition \ref{prop:ProperFreeAction}: if
$y\in \cC(P)$ and $g\in \overline\Aut_0(P)^{k+1}$ is such that $gy\in \cC(P)$, then
$g\in \cG(P)$.
\end{proof}

 \subsection{Proof of Theorem \ref{theo:ModuliSpace}}

\begin{proof}[Proof of Theorem \ref{theo:ModuliSpace}]
Let us first prove the theorem on the open set $\bfM^{rs}(\G)$ corresponding to regularly stable Higgs bundles. In this case, the slices constructed in
    Proposition \ref{prop:slice} form a smooth atlas for which the quotient
    map $\pi: F^{-1}(0)^{rs} \to \bfM^{rs}(\sG)$ is a smooth principal $\Aut_0(P)$-bundle. To prove this atlas is holomorphic, the
    key remark is that the almost complex structure $\I$ on $F^{-1}(0)^{rs}$
    naturally induces an almost complex structure (still denoted by $\I$)
    on $\bfM^{rs}(\sG)$ as follow: given a tangent vector $u \in \T_p
    \bfM^{rs}(\sG)$, define
\[\I_p(u) := \rd_x\pi( \I_x (v))~,\]
where $v\in \T_x F^{-1}(0)$ is such that $\rd_x\pi (v)=u$. Since $\I$ is $\Aut_0(P)^k$-invariant and $\I$ preserves vertical vectors, one easily checks that the above definition is independent of the choice of the lift. Moreover, the restriction of $\pi$ to the slices is holomorphic. Since $\I$ is integrable on $\cC(P)$ (by the exponential maps), it is on the slices, and hence on $\bfM^{rs}(\sG)$. Furthermore, the transition functions for nearby slices are holomorphic, defining a holomorphic atlas on $\bfM^{rs}(\sG)$. 

For points lifting to stable but not regularly stable Higgs
    bundles $x$ in $F^{-1}(0)$, the complex submanifold $\cS_x$ of
    Proposition \ref{prop:slice} is invariant under the stabilizer of $x$
    in $\Aut_0(P)$ (by item (3)). Since this stabilizer is finite, we get an orbifold structure on $\bfM(\G)$.

To see that the holomorphic bundle map $\Theta:\pi^*\T\J(\Sig)\to
    \ker(\rd\pi_\cJ)$ from \S \ref{section hol bundle map}  descends to $\bfM(\sG)$, it suffices to check $\Theta_{(J,\Phi)}(\mu)$ is tangent to $F^{-1}(0)$ for each $(J,\Phi)\in F^{-1}(0)$ and each $\mu\in \pi^*\T\J(\Sig).$ This is equivalent to $\Theta_{(J,\Phi)}(\mu)$ being in the kernel of the boundary map $\delta_{\B}^1$ from Lemma \ref{lem delta 1 B}. By Remark \ref{rem:eta in beta notation}, we have
\[\delta_{\B}^1\left(\Theta_{(J,\Phi)}(\mu)\right)=\delta_{\B}^1\left(0,\tfrac{1}{2i}\Phi\mu,0\right)=D''\left(\tfrac{1}{2i}\Phi\mu,0\right)=\left[\Phi,\tfrac{1}{2i}\Phi\mu\right]=0.\]

For $s\in \R_{>0}$, the hermitian form $h_s$ descends to $\bfM(\sG)$ via
\[h_s(u,v)_p := h_s(\hat u,\hat v)_x\]
where $\hat u$ and $\hat v$ are harmonic vectors lifting $u$ and $v$. Since $h_s$ is $\Aut_0(P_\sK)$-invariant, such a definition is independent of the choice of $x$ lifting $p$ (as soon as $x$ satisfies the Hitchin equations). The corresponding $2$-form is closed by Corollary \ref{c:Closed}.

For $s=0$, it suffices to define $h_0$ on $\bfM(\sG)$ by evaluating along semiharmonic lift: by Proposition \ref{prop:semiharmonic}, it defines a hermitian form on $\bfM(\sG)$.
\end{proof}
 
\begin{remark} It is important to note that if $\mu\in \T\Jcal(\Sigma)$ corresponds to
    a trivial deformation of $X$, then $\Theta(\mu)\in \rIm D''$. However,
    it may be the case that $\Theta(\mu)\in \rIm D''$ even for nontrivial
    classes in $H^1(X,T_X)$. 
\end{remark}
\subsection{Nonabelian Hodge and variation of the metric}
We end this section by proving two important propositions concerning the regularity of the nonabelian Hodge map and the first variation of the metric along the slice.

The \emph{$\sG$-character variety} of the surface $\Sig$ is the affine GIT quotient
$\Hom(\pi_1\Sig,\sG)\sslash \sG$
of the space of group homomorphisms  $\pi_1\Sig\to\sG$ by conjugation (see
\cite{LubotzkyMagid:85}, and also \cite{Sikora:12} and the references therein).
As in the introduction we denote the irreducible locus of the character variety
by $\bfX(\sG)$,
\begin{equation}
	\label{eq smooth char}\bfX(\sG)\subset \Hom(\pi_1\Sig,\sG)\sslash \sG.
\end{equation} 
The holonomy map $\mathsf{hol}$ from the moduli space of irreducible
flat $\sG$-connections to  $\bfX(\sG)$ is a complex analytic isomorphism usually called the \emph{Riemann-Hilbert correspondence}.

Recall that the map $\rH$ from \eqref{eq H map} sends a Higgs bundle $(J,\Phi)$ to the $\sG$-connection $A_J+\frac{1}{2i}(\Phi-\Phi^*)$, and that $\rH(J,\Phi)$ is flat if $(J,\Phi)$ 
solves the Hitchin equations \eqref{eqn:hitchin-eqns}.  
Since $\rH$ is equivariant with respect to the action of $\Aut(P_\sK)$, it
descends to a map from the joint moduli space $\bfM(\sG)$ to the moduli
space of flat $\sG$-connections.  Post-composing with the (Riemann-Hilbert) holonomy map $\mathsf{hol}$ defines a map from the joint moduli space $\bfM(\sG)$ to $\bfX(\sG)$ which we also denote by $\rH$ 
\begin{equation}
	\label{eq NAH map}
	\rH:\bfM(\sG)\lra \bfX(\sG) ~:~
    [J,\Phi]\longmapsto\mathsf{hol}([A_J+\tfrac{1}{2i}(\Phi-\Phi^*)])\ ,
\end{equation}
 and which we call the \emph{nonabelian Hodge map}. 
 By work of Hitchin \cite{selfduality}, Simpson \cite{SimpsonVHS},
 Donaldson \cite{harmoicmetric} and Corlette \cite{canonicalmetrics}, the
 restriction  $\rH_X: \bfM_X(\sG)\to\Xbold(\G)$ is a real analytic
 diffeomorphism (in the orbifold sense) for each $X\in\Tbold(\Sigma)$. 
According to \cite[Theorem 7.18]{Simpson:94b},  
the map 
$$\Mbold(\G)\lra \Tbold(\Sigma)\times \Xbold(\G)
~:~ [J,\Phi]\longmapsto ([j], \rH[J,\Phi])
$$
is  a homeomorphism.
   As expected, the relative nonabelian Hodge map is real
   analytic, at least on the stable locus. For a closely related result,
   see Slegers \cite{Slegers:22}.
   In the following, we assume a fixed real analytic structure on the
underlying principal bundle $P\to \Sigma$.

\begin{theorem}\label{thm:realanal}
	The nonabelian Hodge map $\rH:\bfM(\sG)\to \bfX(\sG)$ is real
    analytic.
\end{theorem}

\begin{proof}
    Fix a point $x\in \Ccal(P)$ corresponding to a point $p\in \Mbold(\G)$.
    We assume $x=(J_0,\Phi_0)$ satisfies \eqref{eq Hitchin eq}. 
 Recall from \S \ref{sec metrics} that a structure group reduction to a
    maximal compact $P_\sK\subset P$ is equivalent to a Cartan involution
    $\tau$. Let $\tau_0$ denote a harmonic metric for $(J_0,\Phi_0)$.  
    Let $\Scal_x$   be the slice obtained in Proposition \ref{prop:slice}.
    By Theorem \ref{thm Hitchin-Simpson} of Hitchin and Simpson, for each
    $(J,\Phi)\in \Scal_x$ there is a unique harmonic metric $\tau$ satisfying 
    \eqref{eq Hitchin eq}.
    We may therefore  view $\tau$ as a map $\tau:\Scal_x\times P \to
    \End(\gfrak)$.  
    We first claim that $\tau$ is real analytic.  To see this,
    recall that we may express $\tau=\Ad_{g^{-1}}\tau_0\Ad_g$ for some $g$.
    Let $k=(\Ad_g)^\ast\Ad_g$, where the adjoint is with respect to the
    hermitian structure on $P(\gfrak)$ defined by $\tau_0$ (see \S \ref{sec
    metrics}). Analyticity of
    $k$ implies that of $\tau$; indeed, $\tau=\tau_0 k$.  Now 
    for each $(J,\Phi)\in \Scal_x$,  Equation \eqref{eq Hitchin eq}
    in the adjoint representation becomes
    \begin{equation} \label{eqn:change-of-metric}
        \dbar_{A_J^0}(k^{-1}\partial_{A_J^0}k)+F_{A_J^0}+[\Phi,k^{-1}\Phi^\ast
        k]=0\ ,
    \end{equation}
    where $A_J^0$ is the Chern-Singer connection with respect to
    $(J,\tau_0)$. 
By the continuity of $\rm H$, we may assume $k$ is uniformly invertible.
    Then in local coordinates and trivializations, \eqref{eqn:change-of-metric} 
    is a second order  nonlinear 
    elliptic system in $k$.  Moreover, since the map $\sigma$ from \eqref{eq sigma} 
     defining the slice is holomorphic, in particular real
    analytic,
and since $A_J^0$ is analytic in $J$,
    the coefficients of \eqref{eqn:change-of-metric} are real analytic
    on $\Scal_x$. By a theorem of Morrey \cite{Morrey:58}, $k$ (and therefore also $\tau$)  is
    real analytic on $\Scal_x\times P$. 
    Next, from the analyticity of $\tau$ it
    follows that the Chern-Singer connection $A_J$ associated to $(J,\tau)$ is real
    analytic  as well. Hence, the family 
    $$D_{(J,\Phi)}=A_J+\tfrac{1}{2i}(\Phi+\tau(\Phi))$$ 
    of flat connections is real analytic on $\Scal_x\times P$. Finally, from the
    analyticity of the Riemann-Hilbert map, the result follows.  
\end{proof}

We now show that the first variation of the metric solving the Hitchin equations is zero along the slice associated to the semiharmonics. 
The following theorem  generalizes
\cite[Theorem 3.5.1]{LabourieWentworthFuchLocus} and \cite[Proposition 3.12]{CollierWentworth:19}
to the setting of the joint moduli space.

\begin{theorem}\label{thm:MetricvariationSlice}
	Fix a Cartan involution $\tau$ on $P$, and let $(J,\Phi)$ be a stable $\sG$-Higgs which solves the Hitchin equation for $\tau$. Suppose $(J_t,\Phi_t)$ is a 1-parameter family of Higgs bundles with $(J,\Phi)=(J_0,\Phi_0)$ which is contained in a 
    slice $\cS_{(J,\Phi)}$ from Proposition \ref{prop:slice}.
As in the proof of  Theorem \ref{thm:realanal},
    let $(\tau(t))$ be the 1-parameter family of Cartan involutions such that $(J_t,\Phi_t)$ solves the $\tau(t)$-Hitchin equation. Then 
	\[\deriv\tau_t=0.\]
\end{theorem}

Before proving the theorem, we prove an auxiliary lemma. Let $(J_t)$ and $(\tau_t)$ be 1-parameter families of complex structures on $P$ and Cartan involutions, respectively, and denote the associated Chern-Singer connection by $A_{J_t}$.  Set $\tau=\tau(0)$ and $\dt\tau=\deriv\tau(t)$.
Note that $\tau\dt\tau$ is a derivation. Since $\gfrak$ is semisimple there
is section $Z$ of $P(\gfrak)$ with 
$\tau\dt\tau=\ad_Z$.

\begin{lemma}With the above  set up, we have
\begin{equation}\label{eq aux lem}
	( \dt A_J)^{1,0}=-(\rd_{A_J}\tau(Z))^{1,0} +\tau((\dt A_J)^{0,1}).
\end{equation}
\end{lemma}
\begin{proof}
	Let $X$ be a smooth section of $P(\gfrak)$, independent of the
    parameter. By definition of the Chern-Singer
    connection, $\rd_{A_{J_t}}\tau_t(X)=\tau_t(\rd_{A_{J_t}}X)$.
   Differentiating this at $t=0$, we have
    \begin{align*}
        [\dt A_J,\tau(X)]+\rd_{A_J}\dt\tau(X)&= \dt\tau(\rd_{A_J}
        X)+\tau([\dt A_J,X]) \\
        [\dt A_J,\tau(X)]+\rd_{A_J}[\tau(Z),\tau(X)]&= [\tau(Z),\tau(\rd_{A_J}
        X)]+[\tau(\dt A_J),\tau(X)] \\
        [\dt A_J+\rd_{A_J}\tau(Z),\tau(X)]&=
        [\tau(\dt A_J),\tau(X)]. 
    \end{align*}
Since $X$ was arbitrary, $\dt A_J=-\rd_{A_J}\tau(Z)+\tau(\dt
    A_J)$. Taking the $(1,0)$-part completes the proof.
\end{proof}

\begin{proof}[Proof of Theorem \ref{thm:MetricvariationSlice}]
Let $(J_t,\Phi_t)$ be a 1-parameter family in a semiharmonic slice $\cS_{(J,\Phi)}$ such that $(J_0,\Phi_0)=(J,\Phi)$ and $\tau_t$ be the
    1-parameter family of Cartan involutions which such that $(J_t,\Phi_t)$
    solve the $\tau_t$-Hitchin equation
\begin{equation}
	\label{eqn:hitchin-eqns}F_{A_{J_t}}-\tfrac{1}{4}[\Phi_t,\tau_t(\Phi_t)]=0.
\end{equation}
Set  $(\frac{1}{2i}\beta,\psi)=((\dt A_J)^{0,1},(\dt\Phi)^{1,0})$. By assumption $(\beta,\psi)$ are semiharmonic, hence
\begin{equation}
	\label{eq semiharmonic var proof}0=\tfrac{1}{2i}D'(\beta,\psi)=\rd_{A_J}\left(\tfrac{1}{2i}\beta\right)+\left[\left(\tfrac{1}{2i}\Phi\right)^*,\tfrac{1}{2i}\psi\right]=\tfrac{1}{2i}\rd_{A_J}\beta-\tfrac{1}{4}[\tau(\Phi),\psi].
\end{equation}
We first compute the first variation of \eqref{eqn:hitchin-eqns} at $t=0$. Set $\tau=\tau_0$. 
\[\overset{\bullet}{\wideparen{\tau(\Phi)}}=\tau(\dt\Phi)+\dt\tau(\Phi)= \tau(\dt\Phi)+\tau(\tau\dt\tau(\Phi))=\tau(\dt\Phi)+\tau([Z,\Phi]). \]
Since $\Phi$ has type $(1,0)$, the $(0,1)$-part of $\dt\Phi$ does not enter in the computation, and we have
\[\overset{\bullet}{\wideparen{[\Phi,\tau(\Phi)]}}=[\psi,\tau(\Phi)]+[\Phi,\tau(\psi)]+[\Phi,[\tau(Z),\tau(\Phi)]].\]
For the curvature term, we have $\dt F_{A_{J}}= \rd_{A_J}( \dt A_J )$. Using \eqref{eq aux lem}, and $\tfrac{1}{2i}\beta=(\dt A_J)^{0,1}$, we have
$$
	\dt F_{A_{J}}= \tfrac{1}{2i}\rd_{A_J}\beta+ \rd_{A_J}\left(\tau\left(\tfrac{1}{2i}\beta\right)-(\rd_{A_J}\tau (Z))^{1,0}\right)=\tfrac{1}{2i}
	\rd_{A_J}\beta+\tau\left(\tfrac{1}{2i}\rd_{A_J}\beta\right)-\rd_{A_J}((\rd_{A_J}\tau(Z))^{1,0}).
$$
Hence, using $[\Phi,\tau(\psi)]=\tau([\tau(\Phi),\psi])$ and \eqref{eq semiharmonic var proof}, the linearization of the Hitchin equations \eqref{eqn:hitchin-eqns} is given by
\begin{align*}
	 0&=\tfrac{1}{2i}\rd_{A_J}\beta-\tfrac{1}{4}[\tau(\Phi),\psi]+\tau\left(\tfrac{1}{2i}\rd_{A_J}\beta-\tfrac{1}{4}[\tau(\Phi),\psi]\right)-\rd_{A_J}((\rd_{A_J}\tau(Z))^{1,0})-\tfrac{1}{4}[\Phi,[\tau(Z),\tau(\Phi)]]\\&
	 =-\rd_{A_J}((\rd_{A_J}\tau(Z))^{1,0})-\tfrac{1}{4}[\Phi,[\tau(Z),\tau(\Phi)]]\\
	 &=-D''\left((\rd_{A_J}\tau(Z))^{1,0}-\tfrac{1}{2i}[\tau(Z),\tau(\Phi)]\right)\\
	 % &=-D''\left(\partial_{A_J}\tau(Z)+\left[-\tau\left(\tfrac{1}{2i}\Phi\right),\tau(Z)\right]\right)\\
	 &=-D''D'(\tau(Z)).
\end{align*}
By the stability assumption, this implies $\tau(Z)=0$, completing the proof.
\end{proof}

\section{Isomonodromic and horizontal distributions and energy}\label{sec:iso hor and energy}
Recall that $\omega_0$ is the closed two form on $\bfM(\sG)$ defined by the
$\I$-invariant part of the pullback of the Atiyah-Bott-Goldman form by the
nonabelian Hodge map \eqref{eq NAH map}.
The associated hermitian form $h_0$ is given by 
\[h_0(\cdot,\cdot)=2(\omega_0(\I\cdot,\cdot)+i\omega_0(\cdot,\cdot)).\]

In this section we prove Theorems \ref{thm Non intro Thm B} and \ref{thm Non intro Thm C} concerning the relation between the kernel of $h_0,$ the isomonodromic distribution and the energy. These are equivalent to Theorems B and C from the introduction. 
We start by defining the isomonodromic  distribution and deducing some immediate consequences. 
\subsection{Horizontal and isomonodromic distribution}
Recall that the hermitian form $h_0$ is positive definite on the fibers of $\pi:\bfM(\sG)\to\bfT(\Sig)$. Denote the vertical tangent bundle by
\[\V\bfM=\ker(\rd\pi)\subset \T\bfM(\sG).\]
\begin{definition}\label{def:horizontal dist}
	The \emph{horizontal distribution} is the subbundle $\cH\subset \T\bfM(\sG)$ which is $h_0$-perpendicular to the vertical tangent bundle. For each $x\in \bfM(\sG)$ and each tangent vector $[\mu]\in \T_{\pi(x)}\bfT(\Sig)$, there is a unique lift $w_{[\mu]}\in\cH_x$ which we call the \emph{horizontal vector} associated to $[\mu]$. 
\end{definition}
Note that the horizontal vector $w_{[\mu]}$ is the lift of $[\mu]$ with minimal $h_0$-norm.

The fibers of the nonabelian Hodge map $\rH$ define a foliation of $\bfM(\sG)$ called the \emph{isomonodromic foliation}. The leaves of the foliation are called \emph{isomonodromic leaves} and will be denoted by $\cL_\rho= \rH^{-1}(\rho)$. The  leaf $\cL_\rho$ can also be interpreted as the image of a section 
\[s_\rho:\bfT(\Sig)\to\bfM(\sG).\]
 \begin{definition}
The \emph{isomonodromic distribution} is the subbundle $\cD\subset \T\bfM(\sG)$ defined by
     $\cD = \ker(\rd \rH)$.
For each $x\in \bfM(\sG)$ and each tangent vector $[\mu]\in \T_{\pi(x)}\bfT(\Sig)$, there is a unique lift $\ell_{[\mu]}\in\cD_x$ which we call the \emph{isomonodromic vector} associated to $[\mu].$
\end{definition}

In the \S \ref{sec:semiharmonic_representatives}, we will prove the following proposition relating the holomorphic section $\Theta$ from \S \ref{section hol bundle map} and Theorem \ref{theo:ModuliSpace} with the horizontal and isomonodromic distributions, see Remark \ref{rem proof of item 1} and Lemma \ref{lem key relation of iso}. 
\begin{proposition}\label{prop hor and iso dist properties}
	For $x\in\bfM(\sG)$ and $[\mu]\in \T_{\pi(x)}\bfT(\Sig)$, let $w_{[\mu]}$ and $\ell_{[\mu]}$ be the horizontal and isomonodromic lifts of $[\mu]$. Then 
	\begin{enumerate}
		\item $\Vert w_{[\mu]}\Vert^2_{h_0}=-\Vert\Theta_x([\mu])\Vert^2_{h_0}$, and 
		\item $w_{[\mu]}=\frac{1}{2}\left(\ell_{[\mu]}-\I\ell_{[i\mu]}\right)$.
	\end{enumerate}
\end{proposition}

The following two lemmas are immediate from the definition of $\omega_0$. 

\begin{lemma}\label{lem I-inv D in K}
Let $x\in\bfM(\sG)$ and $\cD_x\subset\T_x\bfM(\sG)$ be the isomonodromic distribution at $x.$ Then $\cD_x\cap \I(\cD_x)$ is contained in the kernel of $\omega_0$ and hence the kernel of $h_0.$
\end{lemma}
\begin{proof}
	Suppose $\ell\in \cD_x\cap \I(\cD_x)$. For any tangent vector $u\in T_x\bfM(\sG)$ we have
	\[\omega_0(\ell,u)=\tfrac{1}{2}(\rH^*\omega_{ABG}(\ell,u)+\rH^*\omega_{ABG}(\I(\ell),\I(u)))=0\]
	since $\ell$ and $I(\ell)$ are in the kernel of $\rd_x\rH$.
\end{proof}
We will prove the kernel of $\omega_0$ at $x$ is exactly $\cD_x\cap \I(\cD_x)$ in the next subsection.
\begin{lemma}\label{lem D isotr}
	Both isomonodromic distribution $\cD$ and  $\I(\cD)$ are isotropic with respect to the real part of $h_0$. In particular, for any isomonodromic vector $\ell\in \cD$ we have 
	$\Vert \ell\Vert _{h_0}^2=0$.
\end{lemma}
\begin{proof}
	The real part of $h_0$ is $2\omega_0(\I\cdot,\cdot)$. For $\ell,u\in \cD$ we have
	\[ 2 \omega_0(\I (\ell),u)=\rH^*\omega_{ABG}(\I(\ell),u)+\rH^*\omega_{ABG}(-\ell,\I(u))=0\]
	since $\ell$ and $u$ are in the kernel of $\rd\rH.$ The proof is the same for $\I(\cD).$
\end{proof}
\begin{remark}
	Lemma \ref{lem D isotr} implies that $h_0$ is nonpositive on the horizontal distribution $\cH$ since $\cD$ and $\cH$ are both transverse to the vertical bundle $\V\bfM$ and $h_0$ is positive definite on $\V\bfM$.
\end{remark}

\subsection{Semiharmonic representatives}\label{sec:semiharmonic_representatives}
 For this section, we fix a stable Higgs bundle $(J,\Phi)$ solving the Hitchin equation and let $x=[J,\Phi]$ be the associated point in $\bfM(\sG)$.  We now relate the semiharmonic representatives of horizontal vectors and isomonodromic vectors. 

Recall that the tangent space $\T_x\bfM(\sG)$ can be represented by 
the space of semiharmonic tangent vectors at $(J,\Phi)$.  
The nonabelian Hodge map on the configuration space is defined by 
\[\rH(J,\Phi)=A_J+\tfrac{1}{2i}\Phi+\left(\tfrac{1}{2i}\Phi\right)^*.\]
As before, we write this flat connection as $D=\rH(J,\Phi)$ and decompose $D=D''+D'$, where 
\[\xymatrix{D''=\dbar_J+\frac{1}{2i}\Phi&\text{and}&D'=\partial_{A_J}+\left(\tfrac{1}{2i}\Phi\right)^*}.\] 
Now $D'$ and $D''$ satisfy the K\"ahler identities
\eqref{eqn:kahler-simpson}, and since $D$ is flat we have $D'D''=-D''D'$. 
As in Definition \ref{def:semiharmonic}, a tangent vector $(\mu,\beta,\psi)\in \T_{(J,\Phi)}\cC(P)$ is called semiharmonic if 
\begin{equation}
	\label{semiharmonic THm BC sect} \xymatrix{D''(\beta,\psi)+D'\left(\frac{1}{2i}\Phi\mu\right)=0&\text{and}&D'(\beta,\psi)=0}.
\end{equation}. 
Finally, recall that the $h_0$-norm of a class $v=[(\mu,\beta,\psi)]\in \T\bfM(\sG)$ is \[\Vert v\Vert ^2_{h_0}= -\Vert \tfrac{1}{2i}\Phi\mu\Vert^2+\Vert (\beta,\psi)\Vert^2~,\] 
where $(\mu,\beta,\psi)$ is any semiharmonic representative of
$v$ (see Proposition \ref{prop:semiharmonic}), and where henceforth
$\Vert\cdot\Vert$
(unannotated) always denotes the $L^2$-norm.
At $(J,\Phi)$, the section $\Theta$ from \S \ref{section hol bundle map} associates to a Beltrami differential $\mu$ the vertical vector \[\Theta_{(J,\Phi)}(\mu)=(0,\tfrac{1}{2i}\Phi\mu,0).\]
Note that $\Theta_{(J,\Phi)}(\mu)$ is not usually harmonic. Its
harmonic representative is
$\pr_{\ker(D')}\left(\tfrac{1}{2i}\Phi\mu\right)$, where $\pr_{\ker(D')}$ is the
$L^2$-orthogonal projection onto $\ker(D').$ 
By holomorphicity of $\Phi$, this only depends on the class $[\mu]$ of $\mu$. 
The following proposition is  immediate.
\begin{proposition}\label{prop h0 norm of Theta}
Let $x\in\bfM(\sG)$ and $[\mu]\in\T_{\pi(x)}\bfT(\Sig)$ be a tangent vector. Then 
\[\Vert\Theta_x([\mu])\Vert_{h_0}^2= \Vert
    \pr_{\ker(D')}\left(\tfrac{1}{2i}\Phi\mu\right)\Vert^2,\]
where $(J,\Phi)$ is a stable Higgs bundle solving the Hitchin equation with $x=[J,\Phi].$
\end{proposition}

\subsubsection{Semiharmonic horizontal vectors}

We have the following lemmas characterizing the semiharmonic horizontal vectors and their norms.

\begin{lemma}\label{lem D' and D'' eta}
Fix a Beltrami differential $\mu,$ and suppose $(\mu,\beta,\psi)$ satisfies
    $D''(\beta,\psi)+D'(\frac{1}{2i}\Phi\mu)=0$. Then $(\mu,\beta,\psi)$ is
    the semiharmonic representative of the horizontal vector
    $w_{[\mu]}\in\cH_x$ if and only there exists a unique
    $\eta\in\Omega^0(\Sig,P(\gfrak))$ such that $D'(\eta)=(\beta,\psi)$.
    In this case, $D''(\eta)=\pr_{\mathrm{Im}(D'')}\left(\tfrac{1}{2i}\Phi\mu\right),$
	where $\pr_{\mathrm{Im}(D'')}$ denotes the orthogonal projection of onto $\mathrm{Im}(D'')$.
\end{lemma}
\begin{proof}

	If $(\mu,\beta,\psi)$ is a semiharmonic horizontal vector then it
    satisfies $D''(\beta,\psi)+D'(\frac{1}{2i}\Phi\mu)=0$ and is
    $h_0$-perpendicular to all vertical vectors. But  vertical vectors are
    given by triples $(0,\beta_1,\psi_1)$ with $(\beta_1,\psi_1)\in\ker(D'')$. Thus, 
	\[0=\langle (\mu,\beta,\psi),(0,\beta_1,\psi_1)\rangle_{h_0}=\langle(\beta,\psi),
    (\beta_1,\psi_1)\rangle~\]
	implies $(\beta,\psi)\in\ker(D'')^\perp=\mathrm{Im}(D').$ So there exists $\eta\in \Omega^0(\Sig,P(\gfrak))$ such that $D'\eta=(\beta,\psi)$. Uniqueness of $\eta$ follows from the stability assumption. 
	Conversely, if such an $\eta$ exists, then $(\beta,\psi)$ is $h_0$-perpendicular to the vertical space and semiharmonic since $D'^2=0$. Hence $(\mu,\beta,\psi)$ represents a horizontal vector $w_{[\mu]}\in \cH.$

Now suppose $(\mu,\beta,\psi)$ is a semiharmonic horizontal vector, and
    write $(\beta,\psi)=D'\eta.$ It remains to prove $D''\eta=\pr_{\mathrm{Im}(D'')}(\frac{1}{2i}\Phi\mu).$ 
The $D''$ and $D'$ Laplacians on $\Omega^0(\Sig,P(\gfrak))$ are given by
    $(D'')^*D''$ and $(D')^*D'$, respectively. By the stability assumption,
    both of these operators have trivial kernel. Let $G_{D''}$ and $G_{D'}$
    be the associated Green's functions (i.e.\ bounded two-sided inverses)
    of the respective Laplacians. 
    Using $D'D''=-D''D'$, one checks that $G_{D'}=G_{D''}.$

Since $(\mu,\beta,\psi)$ is semiharmonic, we have $D''(\beta,\psi)=-D'(\frac{1}{2i}\Phi\mu)$. By the K\"ahler identities, we can write $\eta$ as 
\[\eta=G_{D'}(D')^*D'\eta = G_{D'}(D')^*(\beta,\psi) = -iG_{D'}\Lambda D'\left(\tfrac{1}{2i}\Phi\mu\right) =G_{D''}(D'')^*\left(\tfrac{1}{2i}\Phi\mu\right). \]
Using the Hodge decomposition, write $\tfrac{1}{2i}\Phi\mu=D''u + v$, where $v\in \ker (D'')^\ast$. Then 
\begin{equation}
	\label{eq Im D''}D''\eta=
    D''G_{D''}(D'')^*\left(\tfrac{1}{2i}\Phi\mu\right) =
    D''G_{D''}(D'')^*(D''u)=D''u=\pr_{\mathrm{Im}(D'')}\left(\tfrac{1}{2i}\Phi\mu\right).
\end{equation}
This completes the proof.
\end{proof}

\begin{lemma}\label{lem norm of min norm vect}
 	Let $(\mu,\beta,\psi)$ is a horizontal semiharmonic vector. Then
\begin{enumerate}
	\item $\Vert (\beta,\psi)\Vert ^2=\Vert \pr_{\mathrm{Im}(D'')}\left(\tfrac{1}{2i}\Phi\mu\right)\Vert ^2$, and 
	\item $\Vert (\mu,\beta,\psi)\Vert _{h_0}^2=-\Vert \pr_{\ker(D')}\left(\tfrac{1}{2i}\Phi\mu\right)\Vert^2,$
\end{enumerate}
	where $\pr_{\mathrm{Im}(D'')}$ and $\pr_{\ker(D')}$ are the orthogonal projections onto $\mathrm{Im}(D'')$ and $\ker(D')$, respectively.
\end{lemma}
\begin{remark}\label{rem proof of item 1}
	Combining Item (2) of Lemma \ref{lem norm of min norm vect} with Proposition \ref{prop h0 norm of Theta} gives the proof of Item (1) of Proposition \ref{prop hor and iso dist properties}. 
\end{remark}
	\begin{proof}
Let $(\mu,\beta_\mu,\psi_\mu)$ is a horizontal semiharmonic vector and $\eta\in\Omega^0(\Sig, P(\gfrak))$ be as in Lemma \ref{lem D' and D'' eta}. Then, using the K\"ahler identities and $D''(\beta_\mu,\psi_\mu)=-D'(\frac{1}{2i}\Phi\mu)$, we have 
\[\Vert (\beta_\mu,\psi_\mu)\Vert^2=\langle D'\eta,D'\eta\rangle=\langle \eta,(D')^*D'\eta\rangle=\langle \eta, -i\Lambda D'\left(\tfrac{1}{2i}\Phi\mu\right)\rangle=\langle D''\eta,\tfrac{1}{2i}\Phi\mu\rangle.\]
Item (1) now follows from the fact that $D''\eta=\pr_{\mathrm{Im}(D'')}(\frac{1}{2i}\Phi\mu).$

Item (2) follows from Item (1) and the fact that $\ker(D')=\ker((D'')^*)$. Namely,
\begin{align*}
	\Vert (\mu,\beta_\mu,\psi_\mu)\Vert_{h_0}^2&=-\Vert
    \tfrac{1}{2i}\Phi\mu\Vert ^2+\Vert (\beta_\mu,\psi_\mu)\Vert ^2
    \\&=-\Vert \tfrac{1}{2i}\Phi\mu\Vert ^2+\Vert
    \pr_{\mathrm{Im}(D'')}\left(\tfrac{1}{2i}\Phi\mu\right)\Vert ^2 \\& =
    -\Vert \pr_{\ker(D')}\left(\tfrac{1}{2i}\Phi\mu\right)\Vert ^2.
\end{align*} 
This completes the proof.
\end{proof}

\subsubsection{Semiharmonic isomonodromic vectors}
We now characterize semiharmonic representatives of the isomonodromic distribution. We refer to these as semiharmonic isomonodromic vectors. 

\begin{lemma}\label{lem isomon vect on config}
	Let $D=\rH(J,\Phi)$ be the flat connection associated to $(J,\Phi).$ Fix a Beltrami differential $\mu,$ and suppose $(\mu,\beta,\psi)$ satisfies the holomorphicity condition $D''(\beta,\psi)+D'(\frac{1}{2i}\Phi\mu)=0$.
	Then, $(\mu,\beta,\psi)$ is a representative of the isomonodromic vector $\ell_{[\mu]}\in\cD$ if and only if there exists a unique $\zeta\in \Omega^0(\Sig,P(\gfrak))$ such that 
	\[\rd_{(J,\Phi)}\rH(\mu,\beta,\psi)= D\zeta.\]
\end{lemma}
\begin{proof}
	A tangent vector $(\mu,\beta,\psi)$ defines an isomonodromic vector if and only if the  $\rd_{(J,\Phi)}\rH(\mu, \beta,\psi)$ is zero when projected to the moduli space of flat connections. This is equivalent to being in the image of the flat connection $D.$ Uniqueness follows from the assumption that $(J,\Phi)$ is stable. 
\end{proof}
 From the computation of the derivative of $\rH$ in Lemma \ref{lem deriv of H config},
and using Theorem \ref{thm:MetricvariationSlice},
 the equation $\rd_{(J,\Phi)}\rH(\mu,\beta,\psi)=D\zeta$ is written explicitly as
	\begin{equation}
		\label{eq iso gauge
        explicit}\dbar_J\zeta+\left[\tfrac{1}{2i}\Phi,\zeta\right]+\partial_{A_J}\zeta+\left[\left(\tfrac{1}{2i}\Phi\right)^*,\zeta\right]=\tfrac{1}{2i}\left(\beta+\psi+\tfrac{1}{2i}\Phi\mu+\beta^\ast
        -\psi^\ast-\left(\tfrac{1}{2i}\Phi\mu\right)^*\right).
	\end{equation}
\begin{lemma}\label{lem semiharmonic iso gauge hermitian}
	Suppose $(\mu,\beta,\psi)$ is a semiharmonic isomonodromic vector, and
    let $\zeta$ be the unique solution to
    $\rd_{(J,\Phi)}\rH(\mu,\beta,\psi)=D\zeta$. Then $\zeta=\zeta^*$, and 
\begin{equation}
	 \label{eq D' isomono vect}\xymatrix{D'\zeta=\tfrac{1}{2i}\left(\beta+\psi-\left(\tfrac{1}{2i}\Phi\mu\right)^*\right)&\text{and}&D''\zeta=\tfrac{1}{2i}\left(\beta^*-\psi^*+\tfrac{1}{2i}\Phi\mu\right).}
	 \end{equation} 
\end{lemma}
\begin{proof}
	Suppose $(\mu,\beta,\psi)$ is a semiharmonic isomonodromic vector, then we have 
	\[D\zeta= D'\zeta+D''\zeta=\tfrac{1}{2i}\left(\beta+\beta^*+\psi-\psi^*+\tfrac{1}{2i}\Phi\mu-\left(\tfrac{1}{2i}\Phi\mu\right)^*\right).\]
    Using $(D')^2=0$ and $D'(\beta,\psi)=0$, we have
		\[D'D''\zeta=D'\left(-\psi^*,\beta^*\right)+D'\left(\tfrac{1}{2i}\Phi\mu\right)-D'\left(\left(\tfrac{1}{2i}\Phi\mu\right)^*\right).\]
	The term $D'\left(\tfrac{1}{2i}\Phi\mu\right)^*$ vanishes since
	\[D'\left(\tfrac{1}{2i}\Phi\mu\right)^*=\left[\left(\tfrac{1}{2i}\Phi\mu\right)^*,\left(\tfrac{1}{2i}\Phi\mu\right)^*\right]=0.\]
	For the term $D'\left(-\psi^*,\beta^*\right),$ we have 
	\[D'\left(-\psi^*,\beta^*\right)=-\partial_{A_J}+[\left(\tfrac{1}{2i}\Phi\right)^*,\beta^*]=-(\dbar_J\psi+\left[\tfrac{1}{2i}\Phi,\beta\right])^* =-(D''(\beta,\psi)^* =D'\left(\tfrac{1}{2i}\Phi\mu\right)^*,\]
	where we used the holomorphicity condition in the last equality.  Hence, we have
	\[D'D''\zeta = \tfrac{1}{2i}\left(D'\left(\tfrac{1}{2i}\Phi\mu\right)^*+D'\left(\tfrac{1}{2i}\Phi\mu\right)\right).\]
	From this expression and the fact that $D''D'=-D'D''$, we conclude 
\[D''D'\zeta= (D'D''\zeta)^*.\]

On the other hand, one checks that $(D'D''\zeta)^*=D'D''\zeta^*$ for any $\zeta\in \Omega^0(\Sig,P(\gfrak))$. Hence,
\[D'' D' (\zeta-\zeta^*)=0.\]
Since $\mathrm{Im}(D')=\mathrm{Im}(D'')^*$ and stability implies $\ker(D')=0$, we conclude $\zeta=\zeta^*$, as desired. 

We now deduce the expressions for $D'\zeta$ and $D''\zeta.$ The $(1,0)$ part of \eqref{eq iso gauge explicit} is
\[\partial_{A_J}\zeta+\left[\tfrac{1}{2i}\Phi,\zeta\right]=\tfrac{1}{2i}\left(\beta^*+\psi-\left(\tfrac{1}{2i}\Phi\mu\right)^*\right)\]
Using $\zeta=\zeta^*$, we have 
\[\partial_{A_J}\zeta-\left[\tfrac{1}{2i}\Phi,\zeta\right]=\partial_{A_J}\zeta^*-[\tfrac{1}{2i}\Phi,\zeta^*]=\left(\dbar_J\zeta+\left[\left(\tfrac{1}{2i}\Phi\right)^*,\zeta\right]\right)^*=-\tfrac{1}{2i}\left(\beta^*-\psi+\left(\tfrac{1}{2i}\Phi\mu\right)^*\right). \]
Combining these two equations gives 
\[\xymatrix{\partial_{A_J}\zeta=\tfrac{1}{2i}(\psi-\left(\tfrac{1}{2i}\Phi\mu\right)^*)&\text{and}&\left[\tfrac{1}{2i}\Phi,\zeta\right]=\tfrac{1}{2i}\beta^*}.\]
It follows that $\left[\left(\tfrac{1}{2i}\Phi\right)^*,\zeta\right]=\tfrac{1}{2i}\beta$, and we conclude 
\[D'\zeta=\tfrac{1}{2i}\left(\beta+\psi-\left(\tfrac{1}{2i}\Phi\mu\right)^*\right),\]
as desired. The formula for $D''\zeta$ follows from a similar computation. 
\end{proof}

\subsubsection{Relating the semiharmonic vectors}
We now use Lemmas \ref{lem D' and D'' eta} and \ref{lem semiharmonic iso gauge hermitian}  to relate the semiharmonic isomonodromic and horizontal vectors.

\begin{lemma}\label{lem key relation of iso}
 Let $(\mu,\beta,\psi)$ be a semiharmonic horizontal vector, and $(\mu,\beta_1,\psi_1)$ and $(i\mu,\beta_2,\psi_2)$ be semiharmonic isomonodromic vectors, then
$(\beta,\psi)=\tfrac{1}{2}\left((\beta_1,\psi_1)-i(\beta_2,\psi_2)\right).$ 
 In particular,
 \[(\mu,\beta,\psi)=\tfrac{1}{2}\left((\mu,\beta_1,\psi_1)-\I(i\mu,\beta_2,\psi_2)\right).\]
\end{lemma}

\begin{proof}
	Let $(\mu,\beta_1,\psi_1)$ and $(i\mu,\beta_2,\psi_2)$ be semiharmonic isomonodromic vectors. Consider the  vector 
	\[w=(\mu,\tfrac{1}{2}(\beta_1-i\beta_2),\tfrac{1}{2}(\psi_1-i\psi_2))=\tfrac{1}{2}\left((\mu,\beta_1,\psi_1)-\I(i\mu,\beta_2,\psi_2)\right).\]
	 By Lemma \ref{lem semiharmonic iso gauge hermitian}, there are hermitian sections $\zeta_1,\zeta_2\in \Omega^0(\Sig,P(\gfrak))$ such that 
	\[\xymatrix{D'\zeta_1=\tfrac{1}{2i}\left(\beta_1+\psi_1-\left(\tfrac{1}{2i}\Phi\mu\right)^*\right)&\text{and}&D'\zeta_2=\tfrac{1}{2i}\left(\beta_2+\psi_2-\left(\tfrac{1}{2i}\Phi i\mu\right)^*\right)}.\]
	Hence 
	\[D'(i \zeta_1+\zeta_2)= \tfrac{1}{2}\left((\beta_1,\psi_1)-i(\beta_2,\psi_2)\right).\]
	Since $w$ is semiharmonic and $\tfrac{1}{2}\left((\beta_1,\psi_1)-i(\beta_2,\psi_2)\right)$ is in the image of $D'$, Lemma \ref{lem D' and D'' eta} implies $w$ is a semiharmonic horizontal vector. 
\end{proof}

\begin{remark}
	Using the same notation as the above lemma, note that  
	\begin{equation}\label{eq D'' horizontal}
		D''\left(i\zeta_1-\zeta_2\right)=-\beta^*+\psi^*.
	\end{equation}
	Indeed, $-\beta^*+\psi^*= \tfrac{1}{2}(\beta_1^*+i\beta_2^*+\psi_1^*+i\psi_2^*)$ and 
	\[\xymatrix{D''\zeta_1=\tfrac{1}{2i}\left(-\beta_1^*-\psi_1^*+\tfrac{1}{2i}\Phi\mu\right)&\text{and}&D''\zeta_2=\tfrac{1}{2i}\left(-\beta_2^*-\psi_2^*+\tfrac{1}{2i}\Phi i\mu\right)}.\] 
\end{remark}

\subsection{Proof of Theorem \ref{MainThmB}}
The following is equivalent to Theorem \ref{MainThmB} from the introduction.

\begin{theorem}\label{thm Non intro Thm B}
	Let $\cH$ and $\cD$ be the horizontal and isomonodromic distributions on the joint moduli space $\bfM(\sG)$, respectively, and $\Theta:\pi^*\T\bfT(\Sig)\to \V\bfM$ be the holomorphic bundle map from Theorem \ref{theo:ModuliSpace}. Then the hermitian form $h_0$ is nonpositive on $\cH$, and for each $x\in\bfM(\sG)$ the kernel $\cK_x$ of $h_0$ is given by
	\[\cK_x=\{w\in\cH~|~\Vert w\Vert ^2_{h_0}=0\}= \cD_x\cap \I(\cD_x)=\ker(\Theta),\]
	where in the last equality $\cH$ has been identified with $\pi^*\T\bfT(\Sig)$ via $\rd\pi$. 
	In particular, if $h_0$ is nondegenerate at $x$ then it has signature $(\dim\bfX(\sG),3g-3).$ 
\end{theorem}

\begin{proof}
Fix $x\in\bfM(\sG)$ and a horizontal vector $w\in\cH$. Let $(J,\Phi)$ be a Higgs bundle solving the Hitchin equations with $x=[(J,\Phi)]$ and $(\mu,\beta,\psi)$ be a semiharmonic tangent vector for $(J,\Phi).$ By Item (2) of Lemma \ref{lem norm of min norm vect}, we have 
\[\Vert w\Vert ^2_{h_0}:=\Vert (\mu,\beta,\psi)\Vert ^2_{h_0}=-\Vert \pr_{\ker(D')}\left(\tfrac{1}{2i}\Phi\mu\right)\Vert ^2\leq 0.\]
Hence, $h_0$ is nonpositive on the horizontal distribution and 
\[\cK_x=\{w\in\cH~|~\Vert w\Vert ^2_{h_0}=0\}~.\]
The signature of $h_0$ at points where it is nondegenerate follows immediately.

It remains to prove $\cK_x= \cD_x\cap \I(\cD_x)$. The inclusion $\cD_x\cap \I(\cD_x)\subset \cK_x$ is proven in Lemma \ref{lem I-inv D in K}. We now prove the opposite inclusion. Since $\cK_x$ is complex, it suffices to prove $\cK_x\subset\cD_x.$
Suppose $w\in\cK_x$. By the above, $w$ is in the horizontal distribution and has $\Vert w\Vert ^2_{h_0}=0.$   
Hence, by Lemmas \ref{lem norm of min norm vect} and \ref{lem D' and D'' eta}, $\frac{1}{2i}\Phi\mu\in \mathrm{Im}(D'')$  and there is a unique $\eta\in \Omega^0(\Sig,P(\gfrak))$ such that 
\[\xymatrix{D'\eta=\beta+\psi&\text{and}& D''\eta=\tfrac{1}{2i}\Phi\mu}.\]
Since $\tfrac{1}{2i}\Phi\mu$ is a $(0,1)$-form, we have $\tfrac{1}{2i}\Phi\mu=\dbar_J\eta$ and $[\tfrac{1}{2i}\Phi,\eta]=0$. Hence,
\[D'(\eta^*)=\partial_{A_J}\eta^*+[\left(\tfrac{1}{2i}\Phi\mu\right),\eta^*]=(\dbar_J \eta)^*=\left(\tfrac{1}{2i}\Phi\mu\right)^*.\]
Similarly, since $\beta$ and $\psi$ have type $(0,1)$ and $(1,0)$, respectively, we have $\beta=[\left(\tfrac{1}{2i}\Phi\right)^*,\eta]$ and $\psi=\partial_{A_J}\eta$. 
Thus, $\beta^*=-[\frac{1}{2i}\Phi,\eta^*]$, and 
\[D''\eta^*+\beta^*=\bar\partial_J(\eta^*)=(\partial_{A_J}\eta)^*=\psi^*.\]

By the above calculations, we have
$$	\rd\rH_{(J,\Phi)}(\mu,\beta,\psi)=\tfrac{1}{2i}\left(\beta+\psi+\tfrac{1}{2i}\Phi\mu+\beta^*-\psi^*-\left(\tfrac{1}{2i}\Phi\mu\right)^*\right)=D\left(\tfrac{1}{2i}\left(\eta-\eta^*\right)\right)~.
$$
%\begin{align*}
%	\rd\rH_{(J,\Phi)}(\mu,\beta,\psi)&=\tfrac{1}{2i}\left(\beta+\psi+\tfrac{1}{2i}\Phi\mu+\beta^*-\psi^*-\left(\tfrac{1}{2i}\Phi\mu\right)^*\right)\\&=D\left(\tfrac{1}{2i}\left(\eta-\eta^*\right)\right)~.
%\end{align*}
By Lemma \ref{lem isomon vect on config}, this implies the vector $w$ is in the isomonodromic distribution $\cD_x$, as desired. 
\end{proof}

\subsection{The energy function on $\bfM(\sG)$}
Recall that we have fixed a structure group reduction $P_\sK\subset P$ to the maximal compact subgroup. 
Let $(J,\Phi)$ be a stable $\sG$-Higgs bundle which solves the Hitchin
equation for the fixed reduction $P_K.$  The $L^2$-norm of the Higgs field
defines the \emph{energy function}
$\rE:\bfM(\sG)\to \R$. Explicitly,
\begin{equation}
	\label{eq energy function joint}\rE([J,\Phi])=\Vert
    \Phi\Vert^2=-2\int_\Sig\kappa_\gfrak\left(\Psi\wedge\Psi\circ j\right),
\end{equation}
where $\Psi=\tfrac{1}{2i}(\Phi-\Phi^*)$ and $j=\pi(J)$ is the induced complex structure on $\Sig.$  

\begin{definition}
	Fix a representation $\rho\in\bfX(\sG)$. Let $s_\rho:\bfT(\Sig)\to
    \bfM(\sG)$ be the section whose image is the isomonodromic leaf of
    $\rho.$ Then the energy $\cE_\rho$ of $\rho$ is defined to be the function
	\[\cE_\rho=\rE\circ s_\rho:\bfT(\Sig)\to \R.\]
\end{definition}

\begin{remark}
    As mentioned in the introduction, $\Ecal_\rho([X])$ is the energy of
    the (unique) $\rho$-equivariant harmonic map $\widetilde X\to \G/\K$
    (cf.\ \cite{canonicalmetrics,harmoicmetric}).
\end{remark}

 The first and second variation formula for harmonic maps is classical 
 (see \cite[\S 6]{EellsSampson:64}). Here we give a formulation of this for Higgs bundles (see
also \cite{Tosicpluri}).

Let $(J_t,\Phi_t)$ be a family of stable Higgs bundles which solve the Hitchin equations.  Denote the tangent vector at $t=0$ by $(\dt J, \dt \Phi)$ and let $m=\rd\pi_{\cJ}(\dt J)\in\T_j\J(\Sig)$.  The associated first and second variation of the energy is given by
\begin{equation}
 	\label{eq 1st variation of Energy}
 	\dt\rE(J_t,\Phi_t)=-2\int_\Sig\kappa_\gfrak\left(2\Psi\wedge\dt\Psi\circ j+\Psi\wedge\Psi\circ m\right),
 \end{equation}
 \begin{equation}
 	\label{eq 2nd variation of Energy}
 	\ddt\rE(J_t,\Phi_t)=-2\int_\Sig\kappa_\gfrak\left(2\Psi\wedge\ddt \Psi\circ j+2\dt \Psi\wedge\dt\Psi\circ j+4\Psi\wedge\dt \Psi\circ m +\Psi\wedge\Psi\circ\dt m\right)~.
 \end{equation}

The first variation of $\rE$ in semiharmonic directions is given by the following lemma. 
\begin{lemma}\label{lem general 1st variation}
	Let $(\mu,\beta,\psi)$ be a semiharmonic vector at $(J,\Phi)$, then the first variation of $\rE$ in the direction $(\mu,\beta,\psi)$ is given by 
	\[\dt\rE(\mu,\beta,\psi)=2\mathrm{Re}\left\langle
    \Phi,\psi-\left(\tfrac{1}{2i}\Phi\mu\right)^* \right\rangle +
    2\mathrm{Re}\left\langle \Phi,\left(\tfrac{1}{2i}\Phi\mu\right)^*
    \right\rangle= 2\mathrm{Re}\left\langle \Phi,\psi\right\rangle .\]
\end{lemma}
\begin{proof}
	By Theorem \ref{thm:MetricvariationSlice},
    the first variation of the metric in semiharmonic directions vanishes. So, 
	\[\xymatrix{\dt\Phi=\psi+\frac{1}{2i}\Phi\mu&\text{and}&\dt\Psi=\frac{1}{2i}\left(\psi+\frac{1}{2i}\Phi\mu-\psi^*-\left(\frac{1}{2i}\Phi\mu\right)^*\right)}.\] 
	Since $\psi$ and $\frac{1}{2i}\Phi\mu$ have types $(1,0)$ and $(0,1)$ respectively, we have
	\begin{align*}
		-\int_\Sig\kappa_\gfrak\left(2\Psi\wedge\dt\Psi\circ j\right)&=\frac{1}{2} \int\kappa_\gfrak \left(\left(\Phi-\Phi^*\right)\wedge \left(i\psi-i\tfrac{1}{2i}\Phi\mu+i\psi^*-i\left(\tfrac{1}{2i}\Phi\mu\right)^*\right)\right) \\
		& = \frac{i}{2}\int_\Sig\kappa_\gfrak\left(\Phi\wedge
        \left(\psi^*-\tfrac{1}{2i}\Phi\mu\right)\right)+\frac{i}{2}\int\kappa_\gfrak\left(\left(\psi-\left(\tfrac{1}{2i}\Phi\mu\right)^*\right)\wedge \Phi^*\right)\\
		&=\mathrm{Re}\left\langle \Phi,\psi-\left(\tfrac{1}{2i}\Phi\mu\right)^* \right\rangle
	\end{align*}
	Similarly, since $\Psi\circ m=\tfrac{1}{2i}\Phi\mu+\left(\tfrac{1}{2i}\Phi\mu\right)^*$ we have 
	\begin{align*}
		-\int_\Sig\kappa_\gfrak\left(\Psi\wedge\Psi\circ
        m\right)&=\frac{i}{2}\int_\Sig\kappa_\gfrak\left(\Phi-\Phi^*\right)\wedge\left(\tfrac{1}{2i}\Phi\mu+\left(\tfrac{1}{2i}\Phi\mu\right)^*  \right)\\
		&= \frac{i}{2}\int_\Sig\kappa_\gfrak\left(\Phi\wedge
        \tfrac{1}{2i}\Phi\mu\right)+\frac{i}{2}\int_\Sig\kappa_\gfrak\left(\left(\tfrac{1}{2i}\Phi\mu\right)^*\wedge\Phi^*\right)\\
		&=\mathrm{Re}\left\langle\Phi,\left(\tfrac{1}{2i}\Phi\mu\right)^*\right\rangle.
	\end{align*}
	Adding the two terms and using \eqref{eq 1st variation of Energy} completes the proof. 
\end{proof}
Along the horizontal and isomonodromic distributions we have the following.

\begin{lemma}\label{lem  1st var hor and iso}
	Let $x=[J,\Phi]\in \bfM(\sG)$ and let $w_{[\mu]}\in\cH_x$ and $\ell_{[\mu]}\in\cD_x$ be horizontal and isomonodromic vectors, respectively. Then the derivative of $\rE$ in the directions $w_{[\mu]}$ and $\ell_{[\mu]}$ are
	\[\xymatrix{\rd \rE\left(w_{[\mu]}\right)=0&\text{and}&\rd
    \rE(\ell_{[\mu]})=2\mathrm{Re}\left\langle\Phi,\left(\tfrac{1}{2i}\Phi\mu\right)^* \right\rangle}.\]
 \end{lemma}
 \begin{proof}
 Let $(\mu,\beta,\psi)$ be a semiharmonic representative of the horizontal vector $w_{[\mu]}$. By Lemma \ref{lem D' and D'' eta}, $\psi^*=\dbar_J\eta$ for some $\eta\in\Omega^0(\Sig,P(\gfrak))$. By Lemma \ref{lem general 1st variation} and Stokes' theorem we have $\dt\rE(\mu,\beta,\psi)=0$. 

Similarly, let $(\mu,\beta,\psi)$ be a semiharmonic representative b of the isomonodromic vector $\ell_{[\mu]}$. By Lemma \ref{lem semiharmonic iso gauge hermitian}, $\psi^*-(\frac{1}{2i}\Phi\mu)=\dbar \zeta$ for some $\zeta\in \Omega^0(\Sig,P(\gfrak))$. Hence, by Stokes' theorem, along the isomonodromic distribution we have 
\begin{equation}
	\label{eq vert variation zero on iso} 
    -\int_\Sig \kappa_\gfrak\left(2\Psi\wedge\dt\Psi\circ j\right)=0.
\end{equation}
Lemma \ref{lem general 1st variation} now implies $\dt\rE(\mu,\beta,\psi)=2\mathrm{Re}\langle \Phi,\left(\tfrac{1}{2i}\Phi\mu\right)^*\rangle$. 
 \end{proof}
 \begin{remark}
     The second equation in Lemma \ref{lem  1st var hor and iso}
agrees with the well-known first variation of the energy $\cE_\rho$, 
     the formula for which goes back to Douglas (cf.\ \cite[Equation 12.29]{Douglas:39}).
 \end{remark}

 We will also need the second variation of $\rE$ along the isomonodromic distribution.

\begin{lemma}\label{lem 2nd variation}
	Let $(J_t,\Phi_t)$ be a path of Higgs bundles whose tangent vector is an isomonodromic semiharmonic vector $(\mu,\beta,\psi)$. Then the second variation of $\rE$ is given by 
	\[\ddt\rE(J_t,\Phi_t)=-2\int_\Sig\kappa_\gfrak \left(2\Psi\wedge\dot\Psi\circ
    m+\Psi\wedge\Psi\circ\dot m\right)~ = ~
    8\Real\left\langle\partial_{A_J}(i\zeta),\left(\tfrac{1}{2i}\Phi\mu\right)^*\right
    \rangle -2\int_\Sig\kappa_\gfrak\left(\Psi\wedge\Psi\circ\dot m\right),\]
	where $\zeta\in\Omega^0(\Sig,P(\gfrak))$ is given by Lemma \ref{lem semiharmonic iso gauge hermitian}.
\end{lemma}

\begin{proof}
	The second variation is given by Equation \eqref{eq 2nd variation of Energy}. By differentiating Equation \eqref{eq vert variation zero on iso}, along the isomonodromic distribution, we have
	\[0=-\int_\Sig\kappa_\gfrak\left(2\Psi\wedge\ddt \Psi\circ j+2\dt \Psi\wedge\dt\Psi\circ j+2\Psi\wedge\dt \Psi\circ m\right).\]
	Hence, for a path $(J_t,\Phi_t)$ which is tangent to the isomonodromic distribution at $t=0$, we have
	\[\ddt\rE(J_t,\Phi_t)= -2\int_\Sig\kappa_\gfrak\left(2\dt \Psi\wedge \Psi\circ m+\Psi\wedge\Psi\circ\dot m\right).\]

	To complete the proof, we compute $-\int_\Sig\kappa_\gfrak\left(2\dt \Psi\wedge \Psi\circ m\right)$. By Lemma \ref{lem isomon vect on config} and Lemma \ref{lem semiharmonic iso gauge hermitian}, there is a unique $\zeta\in\Omega^0(\Sig,P(\gfrak))$  such that $\zeta=\zeta^*$ and $\dt\Psi=\rd_{A_J}\zeta$. Hence, 
	\begin{align*}
		-\int_\Sig\kappa_\gfrak\left(2\dt \Psi\wedge \Psi\circ m\right)&=
        -2\int_\Sig\kappa_\gfrak\left(\rd_{A_J}\zeta\wedge \left(\tfrac{1}{2i}\Phi\mu+\left(\tfrac{1}{2i}\Phi\mu\right)^*\right)\right)\\ 
		&= 2i\int_\Sig\kappa_\gfrak\left(\rd_{A_J}(i\zeta)\wedge
        \tfrac{1}{2i}\Phi\mu\right)-2i \int_\Sig\kappa_\gfrak\left(\left(\tfrac{1}{2i}\Phi\mu\right)^*\wedge \rd_{A_J}(i\zeta)\right)\\
		&=2\langle \rd_{A_J}(i\zeta),\left(\tfrac{1}{2i}\Phi\mu\right)^*\rangle +2 \left\langle \left(\tfrac{1}{2i}\Phi\mu\right)^*,\rd_{A_J}(i\zeta)\right\rangle\\
		&=4\Real\left\langle \partial_{A_J}(i\zeta),\left(\tfrac{1}{2i}\Phi\mu\right)^*\right\rangle,
	\end{align*}
	where we used that $\zeta=\zeta^*$ implies $\left(\rd_{A_J}(i\zeta)\right)^*=-\rd_{A_J}(i\zeta)$.
\end{proof}

 \subsection{Proof of Theorem \ref{MainThmC}}
 Fix a representation $\rho\in\bfX(\sG)$. The complex Hessian of the energy function $\cE_\rho:\bfT(\Sig)\to \R$ of $\rho$ is defined by $\partial\dbar \cE_\rho.$ The following is equivalent to Theorem \ref{MainThmC} of the introduction. 

 \begin{theorem}\label{thm Non intro Thm C}
Fix a representation $\rho\in \bfX(\sG)$ and let $\cE_\rho:\bfT(\Sig)\to \R$ be the energy function of $\rho$. Fix a point $[j]\in\bfT(\Sig)$ and let $x\in\bfM(\sG)$ be the point in the isomonodromic leaf of $\rho$ in the fiber over $[j]$. 
Then for each tangent vector $[\mu]\in \T_{[j]}\bfT(\Sig)$, the complex Hessian of $\cE_\rho$ at $[j]$ along the complex line spanned by $[\mu]$ is
\[\partial_{[\mu]}\partial_{[\overline\mu]}\cE_\rho=-2\Vert w_{[\mu]}\Vert
     _{h_0}^2=2\Vert\Theta([\mu])\Vert^2,\]
where $w_{[\mu]}$ is the horizontal vector associated $[\mu],$ and $\Theta$ is the holomorphic section from Theorem \ref{theo:ModuliSpace}.
\end{theorem}

\begin{remark}
Note that the complex Hessian of $\rE$ along the complex line spanned by
    the isomonodromic vector $\ell_\mu$ is different than what is computed
    above since the isomonodromic distribution is not complex. 
\end{remark}

Since $\Vert\Theta([\mu])\Vert_{h_0}^2\geq 0$,
the following corollary is immediate from Theorem \ref{thm Non intro Thm B}.

\begin{corollary} \label{cor:toledo}
	Let $\rho\in\bfX(\sG)$. Then the energy function $\cE_\rho$ is plurisubharmonic. 
	Moreover, for $[j]\in\bfT(\Sig)$, the set of directions $[\mu]$ in
    which $\cE_\rho$ is not strictly plurisubharmonic is identified with
    the intersection of $\ker(\Theta)$ and the tangent space to $s_\rho([j])$. 
\end{corollary}

\begin{remark}
    The first statement in Corollary \ref{cor:toledo} was proven by Toledo
    \cite{Toledo}, in a more general context, and later also by To\v si\'c
    \cite{Tosicpluri}. The formula in Theorem \ref{thm Non intro Thm C}
    above makes plurisubharmonicity manifest.  
    In Appendix \ref{sec:tosic-comparison}, we show that this formula agrees
    with the expression found in \cite[Theorem 1.10]{Tosicpluri}.    Additionally, in
\cite[Theorem 1.6]{Tosicpluri} 
   it is shown that $\mu$ is in the kernel of the complex Hessian of the energy function
    if and only there exists a $\xi_\mu\in\Omega^0(P(\gfrak))$ 
	\begin{equation}
	 	\label{eq :Tosic} 
	 	\xymatrix{\dbar_J\xi_\mu=\Phi\mu&\text{and}&[\xi_\mu,\Phi]=0.}
	 \end{equation}
This, too, is clearly  equivalent to being in the kernel of $\Theta$.
    Indeed,  by definition, $\mu\in\ker(\Theta_{(J,\Phi)})$ if and only if
    $\tfrac{1}{2i}\Phi\mu$ is in image of $D''$, which is equivalent to the
    existence of such a $\xi_\mu.$  In this way, 
    the holomorphicity of the degeneracy locus is also immediate. 
\end{remark}
\begin{proof}
 	 Let $j(s,t)$ be a two parameter family of complex structures on $\Sig$  which is holomorphic in the variable $s+it$ and satisfies  $j(0,0)=j$. If $m_s,m_t$ are first variations of $j(s,t)$ at $0$ with respect to $s,t$, respectively, then the complex Hessian of $\cE_\rho$ at $j$ in complex line $m_s+im_t$ is given 
 \begin{equation} \label{eqn:complex-hessian}
    \partial_\mu\partial_{\bar \mu} E_\rho=
        \frac{1}{4}\left(
    \frac{\rd^2\cE_\rho}{\rd s^2}+ \frac{\rd^2\cE_\rho}{\rd t^2}
        \right)~.
    \end{equation} 
  Let $\dot m_s,\dot m_t$ be the second variations at $0$ of $j(s,t)$ with respect $s$ and $t$, respectively. Then, as in \cite[Equation 17]{Toledo}, differentiating the Cauchy-Riemann equations implies 
    \begin{equation}\label{eq second var of j}
    	\dt m_s+\dt m_t= 2 j m_s^2 .
    \end{equation}

Hence, the complex Hessian is computed by adding the two second variations from Lemma \ref{lem 2nd variation} for the isomonodromic vectors $\ell_{[\mu]}$ and $\ell_{[i\mu]}$, where the sums of the two $\dot m$ terms satisfy \eqref{eq second var of j}. Furthermore, it is straight forward to check that $2jm^2=2|\mu|^2j.$ 

Let $\zeta_1,\zeta_2\in\Omega^0(\Sig,P(\gfrak))$ be the associated sections from Lemma \ref{lem isomon vect on config} for the semiharmonic isomonodromic vectors $(\mu,\beta_1,\psi_1)$ and $(i\mu,\beta_2,\psi_2)$, respectively. We have 
    \[ \tfrac{1}{2}\partial_{[\mu]}\partial_{[\overline\mu]}\cE_\rho =
    -\frac{1}{2}\int_\Sig\kappa_\gfrak\left(\Psi\wedge\Psi\circ |\mu|^2j\right)+\Real\left\langle \partial_{A_J}(i\zeta_1),\left(\tfrac{1}{2i}\Phi\mu\right)^*\right\rangle + \Real\left\langle \partial_{A_J}(i\zeta_2),\left(\tfrac{1}{2i}\Phi i \mu\right)^*\right\rangle.\]  

For the first term we have
\begin{align*}
	-\frac{1}{2}\int_\Sig\kappa_\gfrak\left(\Psi\wedge\Psi\circ
    |\mu|^2j\right)&=-\frac{1}{2}\int_\Sig|\mu|^2\kappa_\gfrak\left(-i\tfrac{1}{2i}\Phi\wedge \left(\tfrac{1}{2i}\Phi\right)^*+i\left(\tfrac{1}{2i}\Phi\right)^*\wedge  \tfrac{1}{2i}\Phi\right)
	\\&= i\int_{\Sig}|\mu|^2\kappa_\gfrak\left(\tfrac{1}{2i}\Phi\wedge \left(\tfrac{1}{2i}\Phi\right)^*\right)
	\\&=\Vert \tfrac{1}{2i}\Phi\mu\Vert ^2.
\end{align*}

Using Stokes' theorem and
    $D'\left(\tfrac{1}{2i}\Phi\mu\right)=\rd_{A_J}\left(\tfrac{1}{2i}\Phi\mu\right)$,
    we have 
\begin{align*}
    \tfrac{1}{2}
\partial_{[\mu]}\partial_{[\overline\mu]}\cE_\rho
    &= \Vert \tfrac{1}{2i}\Phi\mu\Vert ^2+
	\Real\langle \partial_{A_J}(i\zeta_1-\zeta_2)\wedge \left(\tfrac{1}{2i}\Phi\mu\right)^*\rangle \\ 
	&=\Vert \tfrac{1}{2i}\Phi\mu\Vert ^2+\Real \left\{i
    \int_\Sig\kappa_\gfrak\left(\partial_{A_J}(i\zeta_1-\zeta_2)\wedge \tfrac{1}{2i}\Phi\mu\right) \right\}\\&
	=\Vert \tfrac{1}{2i}\Phi\mu\Vert ^2-\Real \left\{i
    \int_\Sig\kappa_\gfrak\left((i\zeta_1-\zeta_2)\wedge D'\left(\tfrac{1}{2i}\Phi\mu\right)\right) \right\}~.
\end{align*}
Let $w_\mu=(\mu,\beta,\psi)$ be a semiharmonic horizontal vector associated to $\mu$. Recall that 
    $$D'\left(\tfrac{1}{2i}\Phi\mu\right)=-D''(\beta+\psi)=-\dbar\psi-[\Phi,\beta]\
    ,$$ and, by invariance of the Killing form, we have 
\[\kappa_\gfrak\left((i\zeta_1-\zeta_2)\wedge
    \left[\tfrac{1}{2i}\Phi,\beta\right]\right)= \kappa_\gfrak\left(\left[\tfrac{1}{2i}\Phi,(i\zeta_1-\zeta_2)\right]\wedge\beta\right).\] 
Using $D''=\dbar+\tfrac{1}{2i}\Phi$, Stokes' theorem and \eqref{eq D'' horizontal}, we have
\begin{align*}
	\tfrac{1}{2}\partial_{[\mu]}\partial_{[\overline\mu]}\cE_\rho
    &=\Vert\tfrac{1}{2i}\Phi\mu\Vert ^2+\Real \left\{i \int_\Sig\kappa_\gfrak\left((i\zeta_1-\zeta_2)\wedge D''(\beta+\psi)\right) \right\}\\
	& = \Vert\tfrac{1}{2i}\Phi\mu\Vert ^2-\Real \left\{i \int_\Sig\kappa_\gfrak\left(D''(i\zeta_1-\zeta_2)\wedge (\beta+\psi)\right) \right\}\\
	& =\Vert\tfrac{1}{2i}\Phi\mu\Vert ^2-\Real \left\{i \int_\Sig\kappa_\gfrak\left((-\psi^*+\beta^*)\wedge (\beta+\psi)\right) \right\}\\
	& =\Vert \tfrac{1}{2i}\Phi\mu\Vert ^2-\Vert \beta\Vert ^2-\Vert \psi\Vert ^2~.
\end{align*}
This completes the proof since $\Vert w_{\mu}\Vert ^2_{h_0}=\Vert \psi\Vert ^2+\Vert \beta\Vert ^2-\Vert \tfrac{1}{2i}\Phi\mu\Vert ^2.$
 \end{proof}

\section{Stratification and isomonodromic leaves}\label{sec:Stratification}

We now discuss the $\C^*$-action on the joint moduli space $\bfM(\sG)$, Higgs bundles for real forms of $\sG$ and the $\C^*$ and mapping class group invariant  stratification of $\bfM(\sG)$ defined by the rank of the kernel of the hermitian form $h_0$, or equivalently, the 2-form $\omega_0.$ We then prove Theorems \ref{nonintro Thm D} and \ref{nonintro Thm E} which are equivalent to Theorems D and E from the introduction, respectively.

 \subsection{$\C^*$-action}
Given a complex manifold $N$, we denote by $\CBbb^\ast[N]$ the group of nowhere
vanishing holomorphic functions on $N$. When $N=\J(\Sig)$, the group
$\CBbb^\ast[\Jcal(\Sigma)]$ acts on the configuration space $\cC(P)$ via
$$
\Xi:  \CBbb^\ast[\Jcal(\Sigma)]\times \cC(P)  \longrightarrow  \cC(P)  ~: ~
     (f,(J,\Phi))  \longmapsto  (J,f(\pi_{\cJ}(J))\Phi) \ .
$$
This action preserves the space $F^{-1}(0)^s$ of stable Higgs bundles. If
we denote by $p: \J(\Sig) \to \bfT(\Sig)$ the quotient map, the pullback
gives
$p^* : \CBbb^*[\bfT(\Sig)] \to\CBbb^*[\J(\Sig)]$.
We can thus define the action of an element $f\in \CBbb^*[\bfT(\Sig)]$ on $x\in \bfM(\sG)$ by
\[f\cdotp x := \left[\Xi(p^*f, x_0)\right]~,\]
where $x_0 \in F^{-1}(0)^s$ is any lift of $x$. 
 % (the action is independent of the choice of lift). 
 Restricting the action to constant functions yields a $\C^*$-action on $\bfM(\G)$ that
restricts to the usual $\C^*$-action of Hitchin on each fiber.

\begin{proposition}
    The action of $\CBbb^*[\bfT(\Sig)]$ on $\bfM(\sG)$ described above is holomorphic.
\end{proposition}

\begin{proof}
For a Higgs bundle $x=(J,\Phi)\in \cC(P)$ and a vector $(\dot
    f,(\mu,\beta,\theta))\in \T_f\CBbb^*[\J(\Sig)] \times \T_x\cC(P)$,
    (where canonically, $\T_f\CBbb^*[\J(\Sig)]\simeq \CBbb[\J(\Sig)]$)  
    the differential of $\Xi$ is given by
\[\rd_{(f,x)}\Xi (\dot f,(\mu,\beta,\theta)) = (\mu,\beta,\dot f\Phi+f\theta) \]
and so $\Xi$ is holomorphic. 

The pullback map $p^*$ and the projection $\pi: F^{-1}(0)^s \to \bfM(\sG)$ are both holomorphic.
Around any point $y\in \bfM(\sG)$, one can find a holomorphic slice $\cS$
    in $F^{-1}(0)^s$ such that the restriction of $\pi$ to $\cS$ is a
    biholomorphism onto a neighborhood of $y$. The action of
    $\CBbb^*[\bfT(\Sig)]$ on $\bfM(\sG)$ is then locally given by the
    composition $\pi\circ\Xi\circ \pi|_{\cS}^{-1}$, which is thus holomorphic.
\end{proof}

\subsubsection{The horizontal distribution} We now show that the restriction of this action to $\sU(1)$ preserves the horizontal distribution $\cH$ from Definition \ref{def:horizontal dist}.

\begin{proposition}\label{prop:U(1)ActionPreservesHorizontal}
For $\lambda\in\sU(1)$ denote the multiplication map $m_\lambda:\bfM(\sG)\to\bfM(\sG)$. Then, the derivative $\rd m_\lambda$ preserves the horizontal distribution $\cH.$ 
\end{proposition}

\begin{proof}
The map $m_\lambda$ is induced from the map $m_\lambda:\cC(P)\to\cC(P)$ defined by $m_\lambda(J,\Phi)=(J,\lambda\Phi)$. 
Let $x=(J,\Phi)$ be a stable Higgs bundle which solves the Hitchin equations. Then for $\lambda\in\sU(1)$, $y=(J,\lambda\Phi)$ also solves the Hitchin equations. 
It suffices to prove  $\rd_x m_\lambda:\T_{x}\cC(P)\to \T_{y}\cC(P)$ sends horizontal semiharmonic vectors at $x$ to horizontal semiharmonic vectors at $m_\lambda(x).$

For a tangent vector $v=(\mu,\beta,\psi)\in \T_x\cC(P)$, we have $\rd_xm_\lambda (v)=(\mu,\beta,\lambda\psi).$ 
By Lemma \ref{lem D' and D'' eta}, $v$ is semi-harmonic and horizontal if and only if there exists $\eta\in\Omega^0(P[\gfrak])$ such that 
\[\xymatrix{D'_x(\eta)=(\beta,\psi)&\text{and}&D''_x(\beta,\psi)+D'_x(\tfrac{1}{2i}\Phi\mu)=0,}\]
where $D'_x= \partial_{A_J} + (\tfrac{1}{2i} \Phi)^*$ and $D''_x = \overline\partial_{A_J} + \frac{1}{2i}\Phi$. 
Setting 
\[\xymatrix{D'_y= \partial_{A_J} + (\tfrac{1}{2i} \lambda\Phi)^* &\text{ and } &D''_y = \overline\partial_{A_J} + \tfrac{1}{2i}\lambda\Phi~,}\] we have
\begin{align*}
	 \lambda D''_x(\beta,\psi)+\lambda D'_x(\tfrac{1}{2i}\Phi\mu)&= \overline\partial_{A_J} \lambda\psi + \tfrac{1}{2i}[\lambda\Phi,\beta] + \tfrac{1}{2i}\partial_{A_J} \lambda\Phi\mu \\ &= D''_y(\beta,\lambda \psi) +D'_y(\lambda\tfrac{1}{2i}\Phi\mu).
\end{align*}
 Since $\lambda\in\sU(1)$,  we have
\[	D_y'(\lambda\eta)  =  \partial_{A_J}\lambda\eta
    +[(\tfrac{1}{2i}\lambda\Phi)^*,\lambda \eta] = \lambda
    \partial_{A_J}\eta +  [(\tfrac{1}{2i}\Phi)^*,\eta]= \lambda\psi+\beta. \]
Hence, $(\mu,\beta,\lambda\psi)$ is a semiharmonic horizontal vector in $\T_y\cC(P)$ if and only if $(\mu,\beta,\psi)$ is a semiharmonic horizontal vector in $\T_x\cC(P)$. This completes the proof.
\end{proof}

\subsubsection{$\C^*$-fixed points and cyclic Higgs bundles}
The fixed points of the $\C^*$-action on the moduli $\bfM_X(\sG)$ of Higgs bundles for a fixed Riemann surface $X$ are important from many perspectives. For example, 
Simpson showed that, under the nonabelian Hodge correspondence, $\C^*$-fixed points correspond to complex variations of Hodge structure \cite{SimpsonVHS,Simpson:94b}. 
In addition, Hitchin showed that $\C^*$-fixed points are the critical points of the energy function $\rE_X$ on $\bfM_X(\sG)$ \cite{liegroupsteichmuller}.  This also holds on the joint moduli space.
\begin{proposition}\label{prop crit points of E}
	A point $x\in\bfM(\sG)$ is a critical point of the energy function $\rE$ from \eqref{eq energy function joint} if and only if it is a $\C^*$-fixed point.
\end{proposition}
\begin{proof}By Lemma \ref{lem  1st var hor and iso}, the first variation of the energy is always zero in the horizontal distribution. 
By the above discussion, the first variation of the energy in the vertical directions vanishes exactly at the $\C^*$-fixed points. This completes the proof. 
\end{proof}

In \cite{Simpson:94b}, Simpson proves that a Higgs bundle is a  $\C^*$-fixed point if and only if it is fixed by an infinite order element of $\sU(1)<\C^*$. As a result, we denote the set of $\C^*$-fixed points by
\[\bfY_\infty=\bfM(\sG)^{\C^*}~.\]  
The set of Higgs bundles fixed by a cyclic subgroup of $\C^*$ are also of interest.

\begin{definition} Let $k \in \N$ and $\Z_k<\C^*$ be the subgroup of $k^{th}$-roots of unity.  A \emph{$k$-cyclic Higgs bundle} is a point in $\bfM(\G)$ fixed by the action $\Z_k$. Denote the set of $k$-cyclic Higgs bundles by $\bfY_k$.
\end{definition}

We have the following proposition concerning the sets of fixed points.

\begin{proposition}\label{prop:CyclicSubbundle}
For any $k\in \N\cup \{\infty\}$, the set $\bfY_k$ is a holomorphic subbundle of $\pi : \bfM(\G) \to \bfT(\Sig)$ such that for any $x$ in $\bfY_k$, the horizontal space $\cH_x$ is contained in $\T_x\bfY_k$.
\end{proposition} 
\begin{proof}
Let $\zeta$ be a unit norm complex number generating $\Z_k$. So 
\[\bfY_k = \{ x\in \bfM(\G)~ \vert ~\zeta x = x\}\]
is the set of fixed points  of a holomorphic action, hence is a holomorphic submanifold.
 Given $x\in \bfY_k$, we have
\[\T_x\bfY_k = \ker(\rd m_\zeta - \Id) ~.\]
By Proposition \ref{prop:U(1)ActionPreservesHorizontal}, the horizontal space $\cH_x$ is preserved by $\rd m_\zeta$. The fact $\rd m_\zeta$ is the identity on $\cH_x$ follows directly from $\pi(\zeta y)=\pi(y)$ for any $y\in\bfM(\G)$. In particular
\[\cH_x \subset \ker(m_\zeta-\Id) = \T_x \bfY_\zeta~.\]
Since $\rd_x\pi$ restricts to an isomorphism on $\cH_x$, $\pi$ restricts to a holomorphic submersion on $\bfY_\zeta$.
\end{proof}

Let us now describe a nice application of Proposition \ref{prop:CyclicSubbundle}. The uniformization theorem associates to any point $X\in \bfT(\Sig)$ a Fuchsian representation $\rho_X\in \bfX(\mathsf{PSL}_2\C)$, and hence a \emph{uniformizing Higgs bundle} defined by 
\[u(X) := \rH_X^{-1}(\rho_X) \in \bfM_X(\mathsf{PSL}_2\C)~.\]
We thus obtain the \emph{uniformizing section} $u: \bfT(\Sig) \to \bfM(\mathsf{PSL}_2\C)$.  

\begin{proposition}
The uniformizing section is holomorphic and horizontal.
\end{proposition}

\begin{proof}
For any choice $X\in \bfT(\Sig)$, the uniformizing Higgs bundle $u(X)$ is an isolated $\C^*$-fixed point in $\bfM_X(\mathsf{PSL}_2\C)$, \cite{selfduality}. Hence, the image of the uniformizing section is a connected component of $\bfY_\infty$, and is thus holomorphic and horizontal by Proposition \ref{prop:CyclicSubbundle}.
\end{proof}

\subsection{Stratification} Recall from Theorem \ref{thm Non intro Thm B} that the kernel $\cK_x$ of the hermitian form $h_0$ at a point $x\in\bfM(\sG)$ is identified with the kernel of the holomorphic section $\Theta$ at $x$. The dimension of the kernel of $\Theta$ defines a stratification of the moduli space $\bfM(\sG)$ which is preserved by the actions both $\C^*$ and the mapping class group of $\Sig$.

\begin{theorem}\label{thm:stratification}Let 
$\bfM_d=\{x\in \bfM(\sG)~|~\dim(\cK_x)=d\}~.$
	Then 
	\begin{enumerate}
		\item $\bfM_d$ is a $\C^*$-invariant complex subvariety of $\bfM(\sG)$ which is mapping class group invariant and empty if $d>3g-3$,
	\item the closures of the subsets are nested, i.e., they satisfy
	$\overline\bfM_d\subset \coprod_{d\leq c}\bfM_c,$
	\item  $\bfM_0$ is nonempty, open and dense, and 
	\item $\bfM_{3g-3}$ is nonempty and closed.
	 \end{enumerate}
	 In particular, $\bfM(\sG)=\coprod\limits_{0\leq d\leq 3g-3}\bfM_d$ is a $\C^*$-invariant stratification. 
\end{theorem}
\begin{proof}
The proof mainly follows from the identification between the kernel of the hermitian form $h_0$ the kernel of the holomorphic section $\Theta$ from Theorem \ref{MainThmB} (see Theorem \ref{thm Non intro Thm B}). 
The points that do not immediately follow from Theorem \ref{MainThmB} are the nonemptiness of $\bfM_0$ and $\bfM_{3g-3}.$  
For the nonemptiness of $\bfM_0$, the points in the so called Hitchin section are always in $\bfM_0$, see Example \ref{example sl2 Slodowy} below. For nonemptiness of $\bfM_{3g-3}$,  note that $x\in\bfM_{3g-3}$ for all points $x$ where the Higgs field $\Phi$ is zero since the map $\Theta_x$ is identically zero. 
 \end{proof}

Recall that Corollaries \ref{Cor open stratum} and \ref{cor closed stratum} give equivalent characterizations of the open and closed strata, respectively, in terms of the isomonodromic distribution. They both follow directly from Theorems \ref{thm Non intro Thm B}, \ref{thm Non intro Thm C} and \ref{thm:stratification}.

The following is Theorem \ref{MainThmE} from the introduction.
\begin{theorem}\label{nonintro Thm E}
For a representation $\rho\in\bfX(\sG),$ the following are equivalent: 
\begin{enumerate}
	\item The isomonodromic leaf $\cL_\rho$ is contained in the closed stratum $\bfM_{3g-3}$.
	\item The isomonodromic leaf $\cL_\rho$ is a holomorphic submanifold of $\bfM(\sG)$.
	\item The energy function $\cE_\rho$ is constant. 
\end{enumerate}
\end{theorem}
\begin{proof}
By Theorem \ref{thm Non intro Thm B}, the isomonodromic leaf $\cL_\rho$ is contained in the closed stratum if and only if its tangent space $\T_x\cL_\rho=\cD_x$ is a complex subspace of $\T_x\bfM(\sG)$. So (1) and (2) are equivalent. 
If the energy function $\cE_\rho$ is constant, then, at every $x\in\cL_\rho$, its complex Hessian is zero at every point. So, (3) implies (1) by Theorem \ref{thm Non intro Thm C}. 
Finally, if $\cL_\rho\subset\bfM_{3g-3}$, then the first variation of the energy function $\cE_\rho$ is zero for all $x\in\cL_\rho$ by Lemma \ref{lem  1st var hor and iso}, so (1) implies (3).
\end{proof}

The following corollary is immediate from Proposition \ref{prop crit points of E}.
 \begin{corollary}
 	If an isomonodromic leaf is fixed pointwise by the $\C^*$-action then it is holomorphic.
 \end{corollary} 

\subsection{Points in $\bfM_0$ and Proof of Theorem \ref{MainThmD}}
 % We also describe a number of mapping class group invariant subvarieties of $\bfM_0$ which have been studied in the literature, and prove Theorem D from the introduction. 
 We start with the following sufficient condition for a Higgs bundle $[J,\Phi]$ to be in the open stratum $\bfM_0$. 
\begin{proposition}
    \label{lem suff cond open strat}
	Suppose $\sH<\sG$ is a complex reductive subgroup such that the Lie algebra $\gfrak$ has an $\sH$-invariant splitting $\gfrak=U_1\oplus U_2$ with $\dim(U_1)=1$. Let $(J,\Phi)$ be a stable $\sG$-Higgs bundle such that $(P,J)$ admits a holomorphic reduction $P_\sH$ to $\sH$, and write 
	\[\Phi=(\Phi_1,\Phi_2)\in H^0(P_\sH(U_1)\otimes K)\oplus H^0(P_\sH(U_2)\otimes K)\cong H^0(P(\gfrak)\otimes K).\]
	If $\Phi_1$ is nowhere vanishing, then $x=[J,\Phi]$ is in the open stratum $\bfM_0$. 
\end{proposition}
\begin{proof}
Let $[J,\Phi]$ be a stable $\sG$-Higgs bundle  and let $P_\sH\subset P$ be
    a holomorphic reduction to $\sH$, as in the statement. Consider $\mu$ a Beltrami differential and a $\xi\in\Omega^0(P(\gfrak))$ such that $\frac{1}{2i}\Phi\mu=\dbar_J\xi$. Since the reduction $P_\sH\subset P$ is holomorphic, we have  $\xi=(\xi_1,\xi_2)$ for $\xi_j\in \Omega^0(P_\sH(U_j))$ and 
	\[\dbar_J\xi=(\dbar_J\xi_1,\dbar_J\xi_2)=(\tfrac{1}{2i}\Phi_1\mu,\tfrac{1}{2i}\Phi_2\mu).\]
Assume that $\Phi_1$ is a nowhere vanishing section of the line bundle $P_\sH(U_1)\otimes K$, hence
    $$\dbar_j(2i\xi_1\otimes\Phi_1^{-1})=\mu.$$
It follows that $\mu$ defines a zero cohomology class, that is $[\mu]=0\in \T_{[j]}\bfT(\Sig)$. So $[J,\Phi]\in\bfM_0$.
\end{proof}

\subsubsection{Examples of points in the open stratum}
There are many examples of Higgs bundles that satisfy the assumptions of
Proposition \ref{lem suff cond open strat}.
Below we describe two  families of interest.

\begin{example}\label{ex:VectbundleOpen} An $\sSL_n\C$-Higgs bundle $(J,\Phi)$ 
    can be viewed equivalently as a pair $(E,\Phi)$, where $E$ is a rank
    $n$ holomorphic vector bundle with fixed trivial determinant, and
    $\Phi:E\to E\otimes K$ is a traceless holomorphic bundle map. 
    Suppose $E$ decomposes holomorphically as
\[E=E_1\oplus\cdots\oplus E_\ell.\]
With respect to this decomposition a Higgs field $\Phi$
decomposes as $\Phi_{ab}:E_a\to E_b\otimes K$. 
If $E_a\cong E_b\otimes K$ for some $a,b$ and $\Phi_{ab}$ is an
    isomorphism, then the Higgs bundle $(E,\Phi)$ satisfies the assumptions
    of Proposition \ref{lem suff cond open strat}. 
    To see this, it suffices to consider the case when $E=E_1\oplus E_2$ with $E_2= E_1\otimes K^{-1}$ and $\Phi_{12}=\Id$. In this situation the structure group of the bundle reduces to 
\[\sH=\left\{\begin{pmatrix}A&\\&\lambda A
\end{pmatrix}\in\sSL_{2k}\C~|~A\in\sGL_k\C\text{ and }\lambda=\det(A)^{-2}\right\}\] 
This subgroup preserves the splitting $\slfrak_n\C=U_1\oplus U_2$, where 
\[U_1=\left\langle \begin{pmatrix}
	0&0\\\Id&0
\end{pmatrix}\right\rangle\ \ \ \ \text{and}\ \ \ \  U_2=\left\{\begin{pmatrix}
	X&Y\\Z&W
\end{pmatrix}~|~\Tr(Z)=0\right\}.\]
\end{example}
\begin{remark}\label{rem:cyclicVBopen}
Let $(E,\Phi)$ be a stable $\sSL_n\C$-Higgs bundle which is $k$-cyclic and not a $\C^*$-fixed point, i.e., a point in $\bfY_k\setminus\bfY_\infty.$
    By a result of Simpson \cite{KatzMiddleInvCyclicHiggs},
    $E$ decomposes as 
\[E=E_1\oplus\cdots\oplus E_k,\]
and, with respect to  this decomposition,  the components of the Higgs field $\Phi_{ab}:E_a\to E_b\otimes K$ are zero when $b-a\neq 1 \mod~ k.$ 
In particular, if $(E,\Phi)\in\bfY_k\setminus\bfY_\infty$, and one of the nonzero maps 
    \[\Phi_{ab}:E_a\xrightarrow{\hspace*{.4cm}\cong\quad}E_{b}\otimes K\]
 is an isomorphism, then $(E,\Phi)$ is in the open stratum $\bfM_0.$ 
	See \S \ref{ss:CyclicHB} for many examples of Higgs bundles that are
    special cases of Example \ref{ex:VectbundleOpen} and  which have been studied in the literature. 
\end{remark}
The second class of examples come from the so called Slodowy slice construction for an $\slfrak_2\C$ subalgebra of $\gfrak.$ The Higgs bundle analogue below  generalizes Hitchin's description \cite{liegroupsteichmuller} of the Hitchin section, and was introduced in \cite{ColSandGlobalSlodowy}. 
\begin{example}\label{example sl2 Slodowy}
	Suppose $\{f,h,e\}\subset\gfrak$ is a subalgebra isomorphic to $\slfrak_2\C$, where 
	\[\xymatrix{[h,e]=2e,& [h,f]=-2f,&\text{and}&[e,f]=h.}\] Let $\sT<\sG$ be the subgroup defined by exponentiating $\langle h\rangle$ and $\sC<\sG$ be the centralizer of $\{f,h,e\}$. Finally, let $\sH<\sG$ be the subgroup generated by $\sC$ and $\sT.$ Then $\gfrak$ admits a $\sH$-invariant splitting 
	\[\gfrak=W\oplus \langle f\rangle \oplus \bigoplus_{j=0}^NV_j,\]
	where $V_j=\{v\in\gfrak~|~ [e,v]=0 \text{ and } [h,v]=j v\}$. The (Lie algebra) Slodowy slice of $\{f,h,e\}$ is 
	\[\mathfrak{s}_f=\{f+v_0+\cdots+v_N|~v_j\in V_j\}\subset\gfrak.\]

	We now recall a Higgs bundle analogue. 
	Since $\sT$ and $\sC$ are commuting subgroups, the multiplication map $\sC\times\sT\to\sG$ is a group homomorphism. 
	Let $P_\sT$ and $P_\sC$ be holomorphic principal $\sT$ and $\sC$ bundles, respectively, and $P_\sH$ be the holomorphic principal $\sH$ obtained by extending the structure group with the multiplication map. 
	If $P_\sT$ is the holomorphic frame bundle of a square root of $K$, then $P_\sH(\langle f\rangle)\cong K^{-1}$. In particular, $f$ defines a section
	\[f\in H^0\left(P_{\sH}\left(\langle f\rangle\right)\otimes K\right).\]
	The \emph{(Higgs bundle) Slodowy slice} of the $\slfrak_2$-triple $\{f,h,e\}$ is given by 
\[\widehat\cS_f=\{(P,\Phi)=(P_\sH(\sG), f+\psi_0+\cdots \psi_N)~|~\psi_j\in H^0(P_\sH(V_j)\otimes K)\}~.\]

The subset of such Higgs bundles which are stable satisfy the
    assumptions of Proposition \ref{lem suff cond open strat},  and so define a mapping class group invariant subset  of the open stratum $\bfM_0$. 
 For principal $\slfrak_2\C$-triples, the components of the Higgs bundle Slodowy slice are called  Hitchin sections. All Higgs bundles in the Hitchin sections are stable \cite{liegroupsteichmuller}. 
\end{example}

\subsubsection{ Proof of Theorem \ref{MainThmD}}
Fix an $\slfrak_2\C$-triple $\{f,h,e\}\subset\gfrak$, and let $\cS_f(X)\subset\bfM_0$ denote the Higgs bundle in the Slodowy slice from Example \ref{example sl2 Slodowy} on a fixed Riemann surface $X$ which are stable. 
Furthermore, define $\cS_f^0(X)\subset\cS_f(X)$ to be the subset where the holomorphic section $\psi_0$ vanishes
\[\cS_f^0(X)=\{[J,\Phi]\in\cS_f(X)~|~\psi_0=0\}.\]

Applying the nonabelian Hodge map \eqref{eq NAH map}, defines subsets of the character variety
\[\rH(\cS_f^0(X)) \subset \rH(\cS_f(X))\subset \bfX(\sG).\]
 In general, it is not clear how these subsets depend on the choice Riemann surface $X$. 
However, for the special class of \emph{magical $\slfrak_2$-triples} introduced in \cite{BCGGOgeneralCayley}, we have the following. 
\begin{theorem}[\cite{BCGGOgeneralCayley}]\label{thm:BCGGOmag}
 	Let $\{f,h,e\}\subset\gfrak$ be a magical $\slfrak_2$-triple and $X\in\bfT(\Sig)$ be a Riemann surface. Then $\rH(\cS_f^0(X))\subset\bfX(\sG)$ is independent of $X.$
 \end{theorem} 

More precisely, for each magical $\slfrak_2$-triple $\{f,h,e\}$ it was shown that there is a canonical real form $\sG^\R_f<\sG$ such that the polystable points of $\widehat\cS_f^0(X)$ define a union of connected components of the  $\sG^\R_f$-Higgs moduli space on $X$.
In \cite{BCGGOgeneralCayley}, the connected components of the $\sG^\R_f$-Higgs bundle moduli space defined by the polystable points of $\widehat\cS_f^0(X)$ (and their image in the $\sG^\R$-character variety) are called \emph{Cayley components}. As a result, we make the following definition.
 
 \begin{definition}
 	We will say that a representation $\rho\in\bfX(\sG)$ lies in a Cayley component if there is a magical $\slfrak_2$-triple $\{f,h,e\}$  such that $\rho\in \rH(\cS_f^0(X))$ for some (and hence any) Riemann surface $X.$
 \end{definition}
 
 Recall that $\omega_0$ defines a symplectic form on the open stratum $\bfM_0$ which is compatible with $\I$, and that the hermitian form $h_0$ is nondegenerate with signature $(\dim\bfX(\sG),3g-3)$. 
Writing $h_0=2(g_0+i\omega_0)$, $g_0$  is a pseudo-K\"ahler metric on $\bfM_0$. 
  The following theorem is equivalent to Theorem \ref{MainThmD} from the introduction.
\begin{theorem}\label{nonintro Thm D}
	Let $\rho\in\bfX(\sG)$ be  in a Cayley component and $\cL_\rho\subset\bfM(\sG)$ be its isomonodromic leaf. Then, with the notation above, we have
	\begin{enumerate}
	 	\item $\cL_\rho$ is contained in the open stratum $\bfM_0,$
	 	\item $\cL_\rho$ is a $\omega_0$-symplectic submanifold $\bfM_0$ which is totally real and $g_0$-isotropic, and 
	 	\item $h_0$ has signature $(3g-3,3g-3)$ on $\T\cL_\rho\oplus\I(\T\cL_\rho)$.
	 \end{enumerate} 
\end{theorem}
\begin{proof}
	By Theorem \ref{thm:BCGGOmag}, $\rho$ being in a Cayley component implies $\rho\in\rH(\cS_f^0(X))$ for every $X\in\bfT(\Sig)$.
	Since, $\cS_f^0(X)\subset\bfM_0$ for all $X\in\bfT(\Sig)$, the isomonodromic leaf $\cL_\rho$ is contained in $\cM_0.$ 
	Items (2) and (3) now follow from the definition of $\bfM_0$ and Lemma \ref{lem D isotr}.
	\end{proof}

\section{Hitchin map, minimal surfaces and cyclic locus}\label{sec:MinimalSurfaces}

In this section, we construct the Hitchin map in the joint setting and give a different interpretation of the section $\Theta$. We then focus on
the moduli spaces of equivariant minimal surfaces and cyclic Higgs bundles, and prove Theorems
\ref{MainThmF} and \ref{MainThmG} and \ref{MainThmH}.

\subsection{The Hitchin map} 

\subsubsection{The bundle of holomorphic $k$-differentials}
The bundle of holomorphic $k$-differentials over Teichm\"uller space is
defined by taking direct images of the $k$-th power of the relative
cotangent sheaf of the universal curve. A direct construction of this
bundle from the point of view of the Ahlfors-Bers complex structure is
given, for example,  by Bers \cite{Bers:75}. In this section, we briefly describe a construction
in manner of the previous sections.  The result is  the following.
\begin{theorem}
    Fix $k\geq 1$. There is a complex manifold $\Bbold^{(k)}$ and
    holomorphic submersion $\pi_\Bbold:\Bbold^{(k)}\to \Tbold(\Sigma)$ 
    such that the points of the  fiber $\Bbold_X^{(k)}=\pi_\Bbold^{-1}(X)$ 
    consist precisely of the holomorphic $k$-differentials on $X$.
    Moreover, there is an action of $\Mod(\Sigma)$ on $\Bbold^{(k)}$
    by biholomorphisms for which
    $\pi_\Bbold$ is equivariant.
\end{theorem}

Since many of the technical details follow very closely the discussion in
\S \ref{s:ConstructingModuliSpace}, in the following  we shall only give the set-up for the
construction and omit explicit proofs. 

Given $j\in \Jcal(\Sigma)$, let $j^{(k)}\in
\End((T^\ast_\CBbb\Sigma)^{\otimes k})$ be defined by:
$$
j^{(k)}(\alpha_1\otimes\cdots\otimes\alpha_k)=\sum_{p=1}^k \alpha_1\otimes
\cdots\otimes (\alpha_p\circ j)\otimes \cdots\otimes \alpha_k.
$$
Then $j^{(k)}$ preserves $\Sym^k(T^\ast_\CBbb\Sigma)\subset
\End((T^\ast_\CBbb\Sigma)^{\otimes k})$, and there is an eigenspace decomposition
$$
\Sym^k(T^\ast_\CBbb\Sigma) = \bigoplus_{p=0}^k E^{(k)}_{-k+2p}[j],
$$
where $j^{(k)}$ acts by $ip$ on  $E_p^{(k)}[j]$. Let 
$$
\Ebold^{(k)}_p:=\left\{ (j,\alpha)\in \Jcal(\Sigma)\times\Omega^0(\Sigma,
\Sym^k(T^\ast_\CBbb\Sigma))
\mid \alpha\in \Omega^0(\Sigma,E_p^{(k)}[j])\right\}.
$$
This is clearly the total space of a $C^\infty$ bundle on $\Jcal(\Sigma)$
with the Fr\'echet
topology. Moreover, it carries a natural complex structure.
We call $\Ebold^{(k)}_k$ the \emph{bundle of $k$-differentials} on $\Jcal(\Sigma)$.

Given a smooth family $(j(t))_{t\in (-\epsilon,\epsilon)}$ in $\J(\Sig)$ tangent to $\mu \in \T_j\J(\Sig)$, and $\ z_t$ is a local holomorphic 1-form on $(\Sig,j(t))$, then $\dt{\wideparen{dz^k}}=k dz^{k-1}\mu(\partial_z)$. In particular
$$\dt{\wideparen{j^{(k)}}}: E^{(k)}_{k}[j]\lra E^{(k)}_{k-2}[j]~.$$
As a result, if $(j(t),q(t))_{t\in(-\epsilon,\epsilon)}$ is a smooth family in $\Ebold^{(k)}_k$, then $\dt q=(\dt q)_k+
(\dt q)_{k-2}$, where $(\dt q)_k$ is a $k$-differential, and $(\dt q)_{k-2}$ is a
$(k-2)$-differential given by:
$$
(\dt q)_{k-2} =-\tfrac{i}{2}\dt{\wideparen{j^{(k)}}} q~.
$$
The holomorphicity condition is given by the zeros of the map:
$$
F:  \Ebold^{(k)}_k\lra \Ebold^{(k+1)}_k
: q_k\longmapsto \dbar_j q_k
$$
induced by the exterior derivative on forms, and
where we regard 
$$\Lambda^2 T^\ast_\CBbb\Sigma\simeq T^{1,0}\Sigma\otimes T^{0,1}\Sigma\subset
\Sym^2(T^\ast_\CBbb\Sigma)\ .$$
\begin{proposition} 
    The zero set $Z^{(k)}$ of $F$ is a smooth submanifold
    of $\Ebold^{(k)}_k$, and $\Diff_0(\Sigma)$ acts properly
    discontinuously on $Z^{(k)}$.   
\end{proposition}

We then set: $\Bbold^{(k)}=Z^{(k)}/\Diff_0(\Sigma)$.
The structure of a complex manifold then follows along the lines of \S \ref{s:ConstructingModuliSpace}.

\subsubsection{The Hitchin map}
Let us first briefly recall the definition of the Hitchin map on $\bfM_X(\G)$.
For a fixed Riemann surface $X$, let $\Bbold^{(k)}_X=H^0(X,K_X^{\otimes k})$.
Given a semisimple complex Lie group  $\G$ of rank $\ell$, the
\emph{Hitchin base} is defined by
\begin{equation} \label{eqn:hitchin-base}
\Bbold_X(\G) =\bigoplus_{i=1}^\ell \Bbold^{(m_i+1)}_X
\end{equation}
 where $m_1,\ldots, m_\ell$ are the exponents of $\G$.
  For each exponent $m_i$ of $\G$, let
$p_i$ denote a choice of nonzero invariant polynomial on $\gfrak$ of
degree $m_i+1$. The \emph{Hitchin map} is then
$$
\varpi_X:\Mbold_X(\G)\lra \Bbold_X(\G) :
[\Ecal,\Phi]\mapsto(p_1(\Phi),\ldots,p_\ell(\Phi))
~.$$

Now 
let $\Bbold(\G)$ denote the fiber product over
$\Tbold(\Sigma)$ of the spaces $\Bbold^{(m_i+1)}$,
$i=1,\ldots, \ell$, from the previous section. 
Let us continue to denote the projection by $\pi_\Bbold$. 
Then $\pi_\Bbold:\Bbold(\G)\to \Tbold(\Sigma)$ is called the \emph{joint Hitchin
base}. 

We can then define:
\begin{equation} \label{eqn:hitchin-map}
\varpi([J,\Phi])=[(j,p_1(\Phi),\ldots,p_\ell(\Phi))]
\end{equation}
Using the holomorphicity condition on $\Phi$, the map $\varpi$ takes value in $\Bbold(\G)$ and is holomorphic since the $p_\alpha$ are. This is the desired relative Hitchin map. 
We summarize with the following.

\begin{theorem}
There is a holomorphic map $\varpi$ making the diagram
$$
    \begin{tikzcd}
        \Mbold(\G) \arrow[rr,"\varpi"]\arrow[dr,"\pi"'] && \Bbold(\G)\arrow[dl,"\pi_{\Bbold}"] \\
        & \Tbold(\Sig) &
    \end{tikzcd}
    $$
    commute. Moreover,  for every $X\in \Tbold(\Sigma)$, the restriction of
    $\varpi$ to $\pi^{-1}(X)$ is
    the Hitchin map $\varpi_X$.
\end{theorem}

\subsubsection{The quadratic part of $\varpi$}

Note that for any semisimple $\G$, we have $m_1=1$. In particular, the
Hitchin map has a quadratic piece $\varpi^{(2)}: \bfM(\G) \to
\Bbold^{(2)}$. This quadratic part is closely related to $\Theta$ as we now
explain\footnote{We thank Nigel Hitchin for pointing out this interpretation.} (see also \cite{Chen:12}). Here we normalize $\varpi^{(2)}$ by
\[\varpi^{(2)}([J,\Phi]) =
\left(\pi(J),\tfrac{1}{4}\kappa_\gfrak(\Phi\otimes \Phi)\right)~.\]

Using the identification between $\Bbold_X^{(2)}$ and $\T_X^*\bfT(\Sig)$, any $\mu \in \T_X\bfT(\Sig)$ defines a function $f_\mu$ on $\bfM_X(\G)$ by
\[f_\mu([J,\Phi]) = \left\langle \varpi^{(2)}_X([J,\Phi]) , \mu \right\rangle~.\]
Here $\langle\cdotp,\cdotp\rangle$ is the natural pairing between $H^0(K^2_X)$ and $H^1(X, \T X)$ given by
\[\langle q_2, \mu\rangle = \int_\Sig q_2\mu~,\]
where $q_2\mu$ is seen as a $(1,1)$-form on $\Sig$.

Recall that $\bfM_X(\G)$ is equipped with a holomorphic symplectic form $\omega_I^\C$ defined by

\[\omega_I^\C \left((\beta_1,\psi_1),(\beta_2,\psi_2) \right) = i \int_\Sig
\kappa_\gfrak(\psi_2\wedge \beta_1-\psi_1\wedge\beta_2)~.\]

\begin{proposition}\label{prop:dimCoker}
    For any $X\in \bfT(\Sig)$ and $[\mu] \in \T_X\bfT(\Sig)$, the vector
    field $\Theta([\mu])$ on $\bfM_X(\G)$ is the Hamiltonian vector field
    of $f_\mu$ with respect to $\omega_I^\C$. In particular, for any $x\in \bfM(\G)$ we have 
    \[\dim \coker \rd_x\varpi^{(2)}=\dim\ker\Theta_x.\]
\end{proposition}

\begin{proof}
For $x=[J,\Phi]\in \bfM(\G)$ we have
\[\rd_x \varpi_X^{(2)}(\beta,\psi) = \tfrac{1}{2}\kappa_\gfrak(\Phi\otimes\psi)~.\]
Thus
\[\rd_x f_\mu(\beta,\psi)  = \frac{1}{2} \int_\Sig
    \kappa_\gfrak(\Phi\otimes\psi)\mu \\
 = i\int_\Sig \kappa_\gfrak\left(\psi\wedge \tfrac{1}{2i}\Phi\mu\right) \\
 = 4 \omega_I^\C \left(\Theta(\mu),(\beta,\psi)\right)~.\]
This proves the first statement. To show the second statement, first observe that, evaluating $\varpi^{(2)}$ along horizontal vectors, one obtains $\coker \rd_x \varpi^{(2)} = \coker \rd \varpi_X^{(2)}$. The result follows from the duality between $\rd_x f_\mu$ and $\Theta_x(\mu)$.
\end{proof}

\subsubsection{The nilpotent cone} 
We shall not discuss the rich geometry of the Hitchin map, other than to point to
an essential consequence.
For a fixed $X\in\T(\Sig)$, the \emph{nilpotent cone $\bfC_X(\G)$} is the preimage of
$0\in \bfB_X(\G)$ by the Hitchin map $\varpi_X$. In the joint setting, we define
\emph{the (relative) nilpotent cone $\bfC(\G)$} as the preimage by $\varpi$ of the 
zero section in $\bfB(\G)$. 
We record the following:

\begin{proposition}
The nilpotent cone $\bfC(\G)$ is a complex analytic subvariety of 
    $\Mbold(\G)$.
\end{proposition}

Recall the definition of the set $\Ucal_s$ from Theorem
\ref{theo:ModuliSpace}.

\begin{theorem} \label{thm:nilpotent-kahler}
There is $s_0\geq 0$ such that for each $s>s_0$, the set $\Ucal_s$ is an open
    neighborhood of $\Cbold(\G)$.
\end{theorem}

Now let $\vec q\in \Bbold_X(\G)$. Then $\vec q=(q_{k_1},\ldots, q_{k_\ell})$,
where $k_i=m_i+1$ and $q_{k_i}$ is a
holomorphic $k_i$-differential on $X$.
Define the pointwise norm
$$
|\vec q|_\rho^2(z) :=\sum_{i=1}^\ell |q_{k_i}|_\rho^2(z)
$$
where $\rho$ is the hyperbolic metric on $X$ (see Appendix
\ref{sec:app-kahler}).
We set
$$
\Vert\vec q\Vert_{\infty}:= \sup_{z\in X} |\vec q|_\rho^2(z)
$$
Finally, note that given a Cartan involution $\tau$
we can define a pointwise norm for the Higgs field
$|\Phi|_{\rho,\tau}(z)$ as in Appendix \ref{sec:app-kahler}. With this
understood, we have the following.

\begin{lemma} \label{lem:uniform}
    Fix $c>0$. Then there is $K(c)>0$ with the following
    significance. If $(\Ecal,\Phi)$ is a solution to the Hitchin equations
    on a Riemann surface $X$ with harmonic metric $\tau$, and
    $\Vert \varpi_X(J,\Phi)\Vert_\infty\leq c$, then
    $\displaystyle\sup_{z\in X}
    |\Phi|_{\rho,\tau}(z)\leq
    4K(c)$. 
\end{lemma}
The key point is that $K(c)$ is uniform independent of $X$. 
Lemma \ref{lem:uniform} is due to Simpson \cite[Lemma 2.7]{localsystems}
(see also the proof of \cite[Proposition 4.2.21]{Wentworth:16}).

\begin{proof}[Proof of Theorem \ref{thm:nilpotent-kahler}]
    For $c>0$, it is clear that
$$
    \Mbold^c(\G)=
    \left\{ [(J,\Phi)]\in \Mbold(\G) \mid \Vert \varpi_X(J,\Phi)\Vert_\infty< c \right\}
$$
    is an open neighborhood of the nilpotent cone. If we  take
    $s_0=\inf_{c>0} K(c)$, then for any $s>s_0$, there is $c>0$ such that
    $K(c)<s$, and so
    $\Mbold^c(\G)\subset\Ucal_s$. This completes the proof.
\end{proof}

\subsection{Equivariant minimal surfaces}

Any point $x=(J,\Phi)$ in $\bfM(\G)$ corresponds to an equivariant harmonic map $u_x$ 
from the universal cover of $X=\pi(x)$ to the symmetric space of $\sG$. The
following is classical.

\begin{proposition}
    For any point $x$ in $\bfM(\G)$, the holomorphic quadratic differential $\varpi^{(2)}(x)$ is
    proportional to the Hopf differential  of the associated equivariant
    harmonic map. 
\end{proposition}

Recall that $u_x$ is a \emph{branched minimal immersion} if and only if its Hopf differential vanishes
(cf.\ \cite{GOR:73,SacksUhlenbeck:81}.
As a result, the set
\[\bfW(\G) = \{ (J,\Phi) \in \bfM(\G)\mid p_1(\Phi)=0\}\]
is \emph{the space of equivariant minimal immersions}.

\begin{proposition}\label{prop:SpaceOfMinSurf}
The set $\bfW(\G)$ is a complex analytic subvariety of $\pi: \bfM(\G) \to \bfT(\Sig)$.
    The restriction $\Wbold_0:=\Wbold(\G)\cap \Mbold_0$ is a smooth complex
    manifold.
\end{proposition}

\begin{proof}
If $Z$ denotes the image of the zero section in $\T^*\bfT(\Sig)$, then $\bfW(\G)= (\varpi^{(2)})^{-1}(Z)$ and the first statement follows by holomorphicity of $\varpi^{(2)}$. 

For the second statement, since points in $\bfM_0$ correspond to points on
    which the energy is strictly plurisubharmonic (by Corollary \ref{Cor
    open stratum}) and that the complex Hessian of the energy is related to the norm of $\Theta$ by Theorem \ref{thm Non intro Thm C}, we get that $\Theta$ (seen as a section of $\Hom(\pi^*\T\bfT(\Sig),\V \bfM(\G)))$ is injective at every point of $\bfM_0$. The equality between $\dim \coker \rd_x \varpi^{(2)}$ and $\dim \ker \Theta$ of Proposition \ref{prop:dimCoker} then implies that $\varpi^{(2)}$ is a submersion on $\bfM_0$. The result follows.
\end{proof}

We now prove Theorem \ref{MainThmF} of the introduction.

\begin{theorem}\label{thm:NewThmF}
The restriction of $\omega_0$ to $\bfW_0$ is equal to the pullback of the Atiyah-Bott-Goldman symplectic form via (the restriction of) the nonabelian Hodge map.
\end{theorem}

\begin{proof}
Since
$\bfW_0= \left\{(J,\Phi)\in \bfM_0~\vert~\kappa_\gfrak(\Phi\otimes\Phi)=0 \right\}$,
for $x=(J,\Phi) \in \bfW(\sG)$, we have
\[\T_x \bfW_0 = \big\{ (\mu,\theta,\beta) \in \T_x
    \bfM_0~\vert~\kappa_\gfrak(\Phi\otimes\theta)=0\big\}~. \]
Since $\kappa_\gfrak(\Phi\otimes \Phi)=0$, we have for any $(\mu,\theta,\beta) \in \T_x\bfM_0$ that
    $\kappa_\gfrak(\Phi\otimes\Phi\mu)=0$. Writing 
    $\theta=\psi+\frac{1}{2i}\Phi\mu$ with $\psi$ of type
    $(1,0)$, we get that $\kappa_\gfrak(\Phi\otimes \psi)=0$. In particular, for any Beltrami differential $\nu$ on $X=\pi(x)$,  we get we have $\kappa_\gfrak(\psi\wedge\Phi\nu)=0$ (this can be easily seen in local coordinates).
    
From equation (\ref{eq pullbackABG}), the $(2,0)$-part of $\rH^*\omega_{ABG}$ is given by
\[\left(\rH^*\omega_{ABG}\right)^{2,0}_x\big((\mu,\theta_1,\beta_1),(\mu,\theta_2,\beta_2)\big)
    = -\frac{1}{4}\int_\Sig \kappa_\gfrak(\theta_1\wedge \theta_2)  =
    -\frac{1}{8i}\int_\Sig \kappa_\gfrak(\psi_1\wedge\Phi\mu_2 - \psi_2\wedge\Phi\mu_1)~,\]
and so vanishes on $\T_x\bfW_0$. The result follows.
\end{proof}

Since $\bfM_0$ is pseudo-K\"ahler and not K\"ahler, the restriction of $h_0$ to $\bfW_0$ is not necessarily nondegenerate (equivalently, the restriction of $\rH$ to $\bfW_0$ is not necessarily a local symplectomorphism). Nevertheless, we can prove the result on the cyclic locus $\bfY_k\cap \bfW_0$ for $k>2$. This is the content of Theorem \ref{MainThmG} of the introduction.

\begin{theorem}\label{thm:NewThmG}
Let $k>2$ and $x$ be a point in the $k$-cyclic locus $\bfY_k\cap \bfW_0$ of $\bfW_0$. Then the restriction of the nonabelian Hodge map $\rH$ to $\bfW_0$ is a local symplectomorphism around $x$. In particular, $h_0$ restricts to a pseudo-K\"ahler structure on a neighborhood of $x$ in $\bfW_0$.
\end{theorem}
\begin{remark}\label{remark:SlodowyCyclic}
	Recall from Example \ref{example sl2 Slodowy}, that the points of any Slodowy slice $\cS_f$ lie in $\bfM_0.$ 
	Hence, Theorem \ref{thm:NewThmG} applies to all $k$-cyclic Higgs bundle, with $k\geq 3,$ which lie in some Slodowy slice. For example, all $k$-cyclic Higgs bundles in Hitchin components for $k\geq 3.$
\end{remark}
\begin{proof}

Denote by $\V_x \bfY_k$ and $\V_x \bfW_0$ the intersection of the vertical space $\V_x \bfM$ with $\T_x \bfY_k$ and $\T_x\bfW_0$ respectively. The restriction of $h_0$ to $\bfW_0$ and $\bfY_k$ yields orthogonal decompositions
\[\T_x \bfW_0 = \V_x \bfW_0 \oplus \left(\V_x \bfW_0 \right)^\perp~\ \text{ and }~\ \T_x \bfW_0 = \V_x \bfY_k \oplus \left(\V_x \bfY_k \right)^\perp~.\]

The joint Hitchin map $\varpi$ is $\C^*$-equivariant, where $\C^*$ acts with weight $k$ on the bundle $\Bbold^{(k)}$ of holomorphic $k$-differential. In particular, if $\alpha$ is not a multiple of $k$, then $\varpi^{(\alpha)}(x)=0$ for any $x\in \bfY_k$. If follows that $\bfY_k\cap \bfM_0$ is a submanifold of $\bfW_0$, we get that $\V_x \bfY_k \subset \V_x\bfW_0$ and so since the dimension coincide, we get $\left(\V_x \bfW_0 \right)^\perp=\left(\V_x \bfY_k \right)^\perp$. Proposition \ref{prop:CyclicSubbundle} the implies
\[\cH_x = \left(\V_x\bfM_0\right)^\perp = \left(\V_x \bfW_0 \right)^\perp=\left(\V_x \bfY_k \right)^\perp~.\]
But by definition, $x\in \bfM_0$ if and only if $h_0$ is non-degenerate on $\cH_x$. The result follows.
\end{proof}

\subsection{Cyclic Higgs bundles}\label{ss:CyclicHB}
We now describe many examples  of cyclic Higgs bundles that appeared in the literature. 
One can show that each of these examples are in some Slodowy slice for the relevant complex group. 
As a result, they are all in $\bfW_0$, see Remark \ref{remark:SlodowyCyclic} (alternatively, since these examples are described in terms of holomorphic vector bundles, and one can also apply Proposition \ref{lem suff cond open strat} and Remark \ref{rem:cyclicVBopen}). Theorem \ref{thm:NewThmG} then gives the following result.

\begin{theorem}\label{thm:FamilyOfCyclicHB}
Let $\bfZ$ be one of the complex submanifolds of $\bfY_k\cap\bfM_0$ described in the following list. Then the nonabelian Hodge map defines a symplectic immersion to the character variety. In particular, real part of $h_0$ restricts to a nondegenerate pseudo-K\"ahler metric on $\bfZ$.
\end{theorem}

\begin{enumerate}
\item \textbf{Hitchin cyclic.} The first extensive study of cyclic Higgs bundles was carried out by Baraglia in his thesis \cite{g2geometry}, and followed by Labourie in \cite{cyclicSurfacesRank2}. These works focused on the $k$-cyclic Higgs bundles associated to Hitchin components of the split real form of $\sG,$ when $k-1$ is the length of the longest root in $\gfrak$. In this case, Theorem \ref{thm:FamilyOfCyclicHB} recovers Labourie's result that $\rH$ restricts to an immersion \cite[Theorem 1.5.1]{cyclicSurfacesRank2}.

\item \textbf{Maximal representations in rank $2$.} This  second example was considered by Tholozan and the first two authors in \cite{CTTmaxreps}. The $4$-cyclic $\sSO_{n+2}\C$-Higgs bundles which give rise to maximal representation into $\sSO_{2,n}$.
Such cyclic Higgs bundles are a Gauss lift of maximal surfaces in the pseudo-hyperbolic space $\mathbf{H}^{2,n-1}$. By Theorem \ref{thm:FamilyOfCyclicHB}, the corresponding component $\bfZ \subset \bfW(\sG)$ is locally symplectomorphic to the components of maximal representations in $\bfX(\sG)$.

\item \textbf{Alternating holomorphic curves in $\mathbf{H}^{4,2}$.} In \cite{CollierToulisse}, the first two authors described a one-to-one correspondence between some $6$-cyclic Higgs bundles for the exceptional Lie group $\sG_2'$ and some alternate holomorphic curves into the almost-complex pseudo-hyperbolic space $\mathbf{H}^{4,2}$. The corresponding Higgs bundles satisfy the condition of  Proposition \ref{lem suff cond open strat}, hence Theorem \ref{thm:FamilyOfCyclicHB} provides a stronger result than \cite[Theorem C]{CollierToulisse}.

\item \textbf{A-surfaces in $\bf H^{n,n-1}$.} In \cite{NieASurfaces}, Nie describes a correspondence between a certain classes of $2n$-cyclic Higgs bundles for $\SO_{n,n+1}$ and  alternating maximal surfaces in $\mathbf H^{n,n}$ or $\mathbf H^{n-1,n+1}$, depending on the parity of $n$. These  Higgs bundles again satisfy the hypothesis of Proposition \ref{lem suff cond open strat}, and so Theorem \ref{thm:FamilyOfCyclicHB} recovers Nie's infinitesimal rigidity \cite[Theorem C]{NieASurfaces}.

\item \textbf{Non-maximal in $\sSO_{2,3}$.} A special class of $4$-cyclic Higgs bundles for $\sSO_{2,3}$ was considered in \cite{CTTmaxreps} in relation to geometric structures. These Higgs bundles were shown to define Anosov representations by Filip \cite{filip} and Zhang \cite{Zhang:2024}. These results can also be generalized to $\sSO_{2,n}$, see \cite[Remark 5.7]{Zhang:2024}. Theorem \ref{thm:FamilyOfCyclicHB} applies in this situation and yields a new result on this class of cyclic Higgs bundles. 

\item \textbf{Non-Hitchin in $\SL_3(\R)$.} In a recent paper
    \cite{BronsteinDavalo}, Bronstein and Davalo considered a special class
        of 3-cyclic Higgs bundles for $\sG=\SL_3(\R)$ and prove the
        corresponding representations are Anosov. By Theorem
        \ref{thm:FamilyOfCyclicHB}, these cyclic Higgs bundles define
        symplectic immersions into the character variety.

\item \textbf{Non-Hitchin in $\SL_{2n+1}(\R)$.} In \cite{TamburelliRungi},
    Tamburelli and Rungi gives a correspondence between some
        $(2n+1)$-cyclic Higgs bundles for $\SL_{2n+1}(\R)$ and some special
        surfaces in the para-complex hyperbolic space. Theorem
        \ref{thm:FamilyOfCyclicHB} also applies in that case and defines
        symplectic immersions into the character variety.
\end{enumerate}

Finally, as a corollary of Theorem \ref{thm:FamilyOfCyclicHB}, we obtain Theorem \ref{MainThmH} of the introduction.

\begin{corollary}\label{thm:NewThmH}
Let $\bfY\subset\bfX(\sG)$ either be the submanifold Hitchin representations into the split real form when $\mathrm{rk}(\sG)=2$ or the submanifold of maximal $\sSO_{2,n}$-representations when  $\sG=\sSO_{2+n}\C$. 
Let $\bfZ\subset\bfW_0$ be the submanifold of minimal surfaces equivariant with respect to the points of $\bfY$. 
Then  
\[\rH:\bfZ\lra\bfY\] is a global symplectomorphism.
In particular,  $h_0$ defines a mapping class group invariant pseudo-K\"ahler metric on $\bfY$ which has signature $\left((\dim\G-3)(g-1),3g-3\right)$ and is compatible with the Atiyah-Bott-Goldman form.
\end{corollary}

\begin{proof}
Let $x\in \bfM(\G)$ be such that $\rH(x)$ is a maximal representation into $\SO_{2,n}$ or a Hitchin representation in rank $2$ (hence $x\in \bfM_0$).  In that case, the condition $x\in \bfW(\sG)$ is equivalent to being cyclic of order $3,$ $4$ or $6$ when $\gfrak$ is $ \slfrak_3\C$, $\sofrak_{n+2}\C$ or $\gfrak_2$, respectively. 
Such cyclic Higgs bundles are described in items $(1)$ and $(2)$ above, and so we can apply Theorem \ref{thm:FamilyOfCyclicHB}.

Since the dimensions of $\bfW_0$ and $\bfX(\sG)$ coincide, $\rH$ restricts
    to a local symplectomorphism from $\bfZ$ to $\bfY$. The proof that
    $\rH$ is a global symplectomorphism follows from applying \cite[Theorem 8.1.1]{cyclicSurfacesRank2}.
\end{proof}

\appendix \section{K\"ahler identities and inner products}
\label{sec:app-kahler}
 
Here, we summarize the conventions for hermitian geometry on Riemann
surfaces that we use throughout the paper.
Let $X=(\Sigma,j)$ be a Riemann surface with its hyperbolic conformal metric,
expressed as 
$ds^2=\rho(dx^2+dy^2)$ in local conformal coordinates $z=x+iy$.
We shall refer to the metric simply as $\rho$.
The area (K\"ahler) form is
$$
\nu_{\rho}:=\tfrac{i}{2}\rho\, dz\wedge d\bar z
$$
By definition, the pointwise norms are given by:
$$
|\partial_x|_\rho^2= |\partial_y|_\rho^2= \rho\quad ,\quad
|dx|_\rho^2= |dy|_\rho^2=\rho^{-1}
$$
Hence, $\star dx=dy$, $\star dy=-dx$, so that
$
dx\wedge \star dx=|dx|_\rho^2 \rho\, dx\wedge dy
$,
etc.
Now  we have  $dz=dx+idy$, and $\star dz=dy-idx=-idz$;
$\star d\bar z=+id\bar z$.
This gives us:
$$
\bar\star dz=id\bar z\quad ,\quad \bar\star d\bar z=-idz\ .
$$
Notice that $\bar\star^2=-1$. We also have, of course,
$\bar\star\nu_\rho=1$, $\bar\star
1=\nu_\rho$.
Under the real isomorphism of the tangent space with Riemannian metric and
$(1,0)$ vectors with the induced hermitian metric,
$\partial_x\mapsto \partial_z$, $\partial_y\mapsto i\partial_z$,
we have $|\partial_z|_\rho^2=\rho/2$.
Then
$$
dz\wedge \bar\star dz= i dz\wedge d\bar z =
|dz|_\rho^2\, \tfrac{i}{2}\rho\, dz\wedge d\bar z=|dz|_\rho^2 \, \nu_\rho
$$

We define the $L^2$-hermitian inner product on the space  $\Omega^{p,q}(X)$
of complex valued $(p,q)$-forms by:
\begin{equation} \label{eqn:form-inner-product}
\langle \alpha_1,\alpha_2\rangle  :=\int_X \alpha_1\wedge\bar\star\alpha_2
\end{equation}
More explicitly, we have
\begin{align}
\langle \psi_1,\psi_2\rangle  =i\int_X \psi_1\wedge\overline\psi_2
    \quad ,\quad \psi_i\in \Omega^{1,0}(X)\label{eqn:10-inner-product}\\
\langle \beta_1,\beta_2\rangle  =-i\int_X \beta_1\wedge\overline\beta_2
    \quad ,\quad \beta_i\in \Omega^{0,1}(X)\label{eqn:01-inner-product}
\end{align}
It follows immediately from \eqref{eqn:form-inner-product} 
that the formal adjoint of $\dbar$ with respect to the pairing
\eqref{eqn:01-inner-product} is
$\dbar^\ast=-\bar\star\dbar \bar\star$.
Define $\Lambda :\Omega^2(X)\to \Omega^0(X)$ by extending
$\Lambda(\nu_\rho)=1$ complex linearly. Then we have
 the K\"ahler identities:
\begin{equation} \label{eqn:kahler-identities}
\dbar^\ast=-i[\Lambda, \partial] \quad ,\quad
\partial^\ast=-i[\Lambda, \partial]
\end{equation}

The hermitian structure on sections of the holomorphic tangent bundle
$\Omega^0(TX)$ is given by:
\begin{equation} \label{eqn:vector-inner-product}
    \langle v_1,v_2\rangle =\int_X\langle v_1, v_2\rangle_\rho\,
    \nu_\rho
    \quad ,\quad v_i\in \Omega^{0}(TX)
\end{equation}
This is dual to the inner product on $(1,0)$-forms. Indeed,
writing \eqref{eqn:10-inner-product} locally, we see that 
$$
\langle \psi_1,\psi_2\rangle  =\int_X
\langle\psi_1,\psi_2\rangle_\rho
\,\nu_\rho
    \quad ,\quad \psi_i\in \Omega^{1,0}(X)
$$

Let $q_k$ be a $k$-differential on $X$, i.e.\ a section of $K_X^{\otimes k}$. 
In a  local conformal coordinate $z$ on $X$, we write $q_k=q_k(z)dz^k$ and
 define the pointwise norm:
$$
|q_k|^2_\rho(z):= |q_k(z)|^2 (\rho(z)/2)^{-k}
$$
This is independent of the choice of conformal coordinate.

Finally, given a Cartan involution $\tau$, we define a pointwise norm on
Higgs fields as follows. In local conformal coordinates, write
$\Phi=\Phi(z)dz$. Then
$$
|\Phi|_{\rho,\tau}^2(z) := \kappa_\gfrak(\Phi(z),\Phi^\ast(z))(\rho(z)/2)^{-1}
\ .
$$

\section{Link with algebraic geometry}  \label{sec:app-alg-geom}
For the purposes of this paper, we have not needed the fact that
$\Mbold(\G)$ is locally isomorphic to
the analytification of the moduli space  of stable relative $\G$-Higgs bundles
constructed, for example, in \cite{Simpson:94b} and 
\cite{Faltings:93}. 
This assertion  can be shown by producing local universal families on
$\Mbold(\G)$ and then using the modular properties of the algebraic space.
Such an argument is by now standard, but here we omit the details.  
It will be useful, however, to show that the deformation theory agrees, and
that is the purpose of this section.

Fix a smooth complex  projective curve $X$ and a complex reductive group
$\G$. 
Let $\pi:\Pcal\to X$ be an algebraic principal $\G$-bundle.
Recall that the \emph{Atiyah algebra} $\At(\Pcal)$  of $\Pcal$ 
is the  sheaf (coherent analytic on $X$) of invariant holomorphic vector fields on $\Pcal$. 
We have a tautological sequence:
\begin{equation} \label{eqn:atiyah}
    0\lra \ad(\Pcal)\lra \At(\Pcal)\xrightarrow{\hspace*{.3cm}\sigma\quad}T_X\lra 0
\end{equation}
where $\sigma$ is projection to $X$.
We have a homomorphism 
$$\At(\Pcal)\lra \At(\ad(\Pcal)) ~:~ V\longmapsto D_V$$ where $\At(\ad(\Pcal))$ 
is identified with first order differential operators on $\ad(\Pcal)$ with
scalar symbol:
 $\sigma(D_V)=\sigma(V)$.  Explicitly, given $V\in\At(\Pcal)$, and regarding a
 local section $s$ of $\ad(\Pcal)$ over $U\subset X$ as an equivariant map
 $s:\pi^{-1}(U)\to \gfrak$, we define $D_V(s):= V(s)$.  
We also define a  homomorphism 
$$
\At(\Pcal)\lra \At(\ad(\Pcal)\otimes K_X)
$$
($K_X=T^\ast_X$) 
as follows.  For local sections $s$ of $\ad(\Pcal)$  and $\alpha$ of
$K_X$, let
$$
D_V(s\otimes \alpha):= D_V(s)\otimes\alpha+s\otimes d\alpha(\sigma(V))
$$
One checks that this is a well-defined differential operator with
$\sigma(D_V)=\sigma(V)$. 
With this understood, we have the following
\begin{lemma}[cf.\ Welters \cite{Welters:83}] \label{lem:welters}
    Let $(\Pcal,\Phi)$ be a $\G$-Higgs bundle on $X$.
Consider the complex
$$
  \Bscr^\bullet:  \At(\Pcal)\xrightarrow{\hspace*{.3cm} \ad_\Phi\quad} \ad(\Pcal)\otimes
    K_X
$$
    where $\ad_\Phi(V):=[\Phi,D_V]$.
    Then the hypercohomology $\HBbb^1(\Bscr^\bullet)$ parametrizes first
    order joint deformations of $(X, \Pcal,\Phi)$.
\end{lemma}
Next, let $P$ denote the underlying smooth $\G$-bundle, and $J\in\Jcal(P)$
such that $\Pcal\simeq (P,J)$. Let $\dbar_{\ad(P)}$ denote the
$\dbar$-operator on $\ad(P)$ defining the holomorphic bundle $\ad(\Pcal)$. 
We have the following $C^\infty$ resolution of $\Bcal^\bullet$:
\begin{equation}  \label{eqn:resolution}
    \begin{tikzcd}
        \At(\Pcal) \arrow[d, "\ad_\Phi\quad"] \arrow[r]  &\aut(P)\arrow[r,
        "L"]
        \arrow[d,"\ad_\Phi\quad"]
        & T_J\Jcal(P)\arrow[r]\arrow[d,"\ad_\Phi\quad"] &0 \\
        \ad(\Pcal)\otimes K_X \arrow[r] & A^{1,0}(X,\ad(P))
       \arrow[r, "\quad \dbar_{\ad(P)}\quad "] & 
        A^{1,1}(X,\ad(P))\arrow[r] &  0
    \end{tikzcd}
\end{equation}
Here, the map $L$ is given by $V\mapsto L_VJ$, for a smooth invariant
vector field $V$ on $P$. At a solution to the Hitchin equations, 
the diagram commutes if we define $\ad_\Phi$ on $T_J\Jcal(P)$ by:
$$
(\mu,\beta)\mapsto -\tfrac{1}{2i}(\partial_A(\Phi\mu)+[\Phi,\beta]) 
$$
Hence, the complex $B^\bullet$ from \S \ref{sec:deformation} gives a smooth resolution of
$\Bscr^\bullet$.
The next result is an immediate consequence.
\begin{proposition}\label{prop:hypercohomology}
We have an isomorphism:
$\HBbb^i\left( \At(\Bscr^\bullet)\right) \simeq H^i(B^\bullet)$,
    $i=0,1,2$.
\end{proposition}

\section{The algebraic definition of $\Theta$} \label{sec:theta}
In this section we give the algebro-geometric definition of $\Theta$ (see
also \cite{Chen:12}).
Continuing with the notation of the previous section, let
 $\Ascr^\bullet$ denote the complex 
$$\Ascr^\bullet : \ad(\Pcal)\xrightarrow{\hspace*{.3cm} \ad_\Phi\quad}
\ad(\Pcal)\otimes K_X
    $$
    Then $\HBbb^1(\Ascr^\bullet)$ parameterizes the tangent space to
    $\Mbold_X(\G)$ at a stable Higgs bundle $(\Ecal,\Phi)$ on $X$
    \cite{NitsureModuli}.
    The  exact sequence  of complexes 
$$
    \begin{tikzcd}
        0\arrow[r]&
        \ad(\Pcal) \arrow[d, "\ad_\Phi\quad"] \arrow[r]&
        \At(\Pcal) \arrow[d, "\ad_\Phi\quad"] \arrow[r]  &
        T_X \arrow[d]\arrow[r]
        &0 \\
        0\arrow[r]& \ad(\Pcal)\otimes K_X \arrow[r]
        & \ad(\Pcal)\otimes K_X \arrow[r]
        &0\arrow[r]&0
    \end{tikzcd}
$$
gives the tangent sequence of $\Mbold(\G)\to \Tbold(\Sigma)$ at a stable
Higgs bundle:
$$
0\lra \HBbb^1(\Ascr^\bullet)\lra \HBbb^1(\Bscr^\bullet)\lra H^1(T_X)\lra 0
$$
On the other hand, notice that  $\Phi$ defines a map of complexes 
$$
\begin{tikzcd}
    T_X \arrow[r, "\frac{1}{2i}\Phi"]\arrow[d] & \ad(\Pcal)  \arrow[d, "\ad_\Phi\quad"]  \\
    0\arrow[r] & \ad(\Pcal)\otimes K_X 
\end{tikzcd}
$$
since $\Phi$ commutes with itself.  Let $\Theta^{alg}$ be  the induced map
$H^1(T_X)\to \HBbb^1(\Ascr^\bullet)$. The following is then
straightforward.

\begin{proposition}
    In terms of the isomorphism from Proposition \ref{prop:hypercohomology},
    the holomorphic section $\Theta^{alg}$ defined above
    satisfies 
    $\Theta^{alg}([\mu])=[(0,\frac{1}{2i}\Phi\mu,0)]$ for all $[\mu]\in
    H^1(T_X)$.
\end{proposition}

\section{Comparison with the formula of To\v{s}i\'c} \label{sec:tosic-comparison}
We explain here why our formula \eqref{eqn:hessian} for the complex Hessian of the energy agrees
with the one found by To\v{s}i\'c in \cite[Theorem 1.10]{Tosicpluri}.
Fix a Beltrami differential $\mu$ and a harmonic bundle $(J,\Phi)$. For
simplicity, set $(\mu,\beta,\psi)$ to be the minimal norm (horizontal)
tangent vector, and let $(\mu,\beta_1,\psi_1)$, $(i\mu, \beta_2,\psi_2)$ be
the isomonodromic tangent vectors. Recall that  from Lemma \ref{lem key relation
of iso},  
$$
\beta=\tfrac{1}{2}(\beta_1-i\beta_2)\quad ,\quad
\psi=\tfrac{1}{2}(\psi_1-i\psi_2)
$$
and from Lemma \ref{lem D isotr},  $\Vert (\beta_i,\psi_i)\Vert=\Vert\Phi\mu/2i\Vert$ for $i=1,2$. 
Finally, we have also proven that (see Lemma \ref{lem semiharmonic iso
gauge hermitian})
\begin{equation}\label{eqn:zeta}
    \dbar_A\zeta_i=\tfrac{1}{2i}(-\psi_i^\ast+\tfrac{1}{2i}\Phi\mu)\quad ,\quad
[(\tfrac{1}{2i}\Phi\mu)^\ast,\zeta_i]=\tfrac{1}{2i}\beta_i
\end{equation}
for $i=1,2$. 

Now there are two differences in the  normalizations used here compared
with those in \cite{Tosicpluri}. First, our
derivatives are computed with respect to $\mu$, whereas in
\cite{Tosicpluri} they are computed
with respect to $\tfrac{1}{2i}\mu$. This means that the (complex) variation of the harmonic
metric, which To\v{s}i\'c calls $w$, is $2i$ times the variation that we obtain.
So,
\begin{equation} \label{eqn:w}
w= 2i\times \tfrac{1}{2}(\zeta_1-i\zeta_2)= i\zeta_1+\zeta_2
\end{equation}
The second difference is that the Higgs field, which To\v{s}i\'c calls $\partial f_0$,
corresponds to our $\tfrac{1}{2i}\Phi$. Hence, in the notation of
\cite[Theorem 1.10]{Tosicpluri},
\begin{equation} \label{eqn:phi}
    \tfrac{1}{2i} \Phi\mu=\dbar_A\varphi+\theta
\end{equation}
where $\theta\perp \imag\dbar_A$.

With this understood, we proceed to calculate To\v{s}i\'c's expression for
the Hessian $\Delta_\mu E$:
\begin{equation} \label{eqn:tosic}
\Delta_\mu E = 8\Vert\theta\Vert^2+ 8\Vert \dbar_A(w-\varphi)\Vert^2
-8i\int_X \langle R(w,\partial f_0\wedge\dbar f_0,w\rangle
\end{equation}
where $R$ is the  curvature of the (appropriately normalized) invariant
metric on $\G/\K$.
Using \eqref{eqn:w} and \eqref{eqn:zeta}, we have
$$
\dbar_A w=-\tfrac{1}{2i}(\psi_1+i\psi_2)^\ast +\tfrac{1}{2i}\Phi\mu
$$
Now  from \eqref{eqn:phi} and the fact that $\theta\perp \imag\dbar_A$,
\begin{align*}
    \dbar_A(w-\varphi)&=\dbar_Aw+\theta-\tfrac{1}{2i} \Phi\mu \\
    \Vert \dbar_A(w-\varphi)\Vert^2 &=\Vert\dbar_A
    w\Vert^2+\Vert\theta\Vert^2+\Vert\tfrac{1}{2i}\Phi\mu\Vert^2-2\Real\langle\dbar_Aw+\theta,
    \tfrac{1}{2i}\Phi\mu\rangle
\end{align*}
But $\tfrac{1}{2i}\langle\theta,\Phi\mu\rangle=\Vert\theta\Vert^2$, so
$$
 \Vert \dbar_A(w-\varphi)\Vert^2  =\Vert\dbar_A
    w\Vert^2-\Vert\theta\Vert^2+\Vert\tfrac{1}{2i}\Phi\mu\Vert^2-2\Real\langle\dbar_Aw,\tfrac{1}{2i}
    \Phi\mu\rangle
$$
which we can write as
$$\Vert\theta\Vert^2 +
 \Vert \dbar_A(w-\varphi)\Vert^2  =\Vert\dbar_A w-\tfrac{1}{2i}\Phi\mu\Vert^2
 =\tfrac{1}{4}(\Vert\psi_1\Vert^2+\Vert\psi_2\Vert^2+2\Real\langle\psi_1,i\psi_2\rangle)
 $$
For the last term in \eqref{eqn:tosic}, we have
\begin{align*}
-i\int_X \langle R(w,\partial f_0\wedge\dbar f_0,w\rangle
    &=\Vert [\tfrac{1}{2i}\Phi\mu, i\zeta_1+\zeta_2]\Vert^2
    =\Vert [(\tfrac{1}{2i}\Phi\mu)^\ast, -i\zeta_1+\zeta_2]\Vert^2\\
    &=\Vert \tfrac{1}{2}(\beta_1+i\beta_2)\Vert^2
    =
    \tfrac{1}{4}(\Vert\beta_1\Vert^2+\Vert\beta_2\Vert^2+2\Real\langle\beta_1,i\beta_2\rangle)
\end{align*}
It follows that
\begin{equation} \label{eqn:tosic2}
\Delta_\mu
E=4(\Vert\Phi\mu/2i\Vert^2+\Real\langle(\beta_1,\psi_1),(i\beta_2,i\psi_2)\rangle)
\end{equation}
On the other hand,
\begin{align*}
    \Vert(\beta,\psi)\Vert^2&=
    \tfrac{1}{4}(\Vert(\beta_1,\psi_1)\Vert^2+\Vert(\beta_1,\psi_1)\Vert^2
    -2\Real\langle\beta_1,i\beta_2\rangle) \\
    &=\tfrac{1}{2}(\Vert\tfrac{1}{2i}\Phi\mu\Vert^2
    -\Real\langle(\beta_1,\psi_1),(i\beta_2,i\psi_2)\rangle)
    \\
\Real\langle(\beta_1,\psi_1),(i\beta_2,i\psi_2)\rangle)
    &=
    \Vert\tfrac{1}{2i}\Phi\mu\Vert^2 -2\Vert(\beta,\psi)\Vert^2
\end{align*}
Finally, plugging this into   \eqref{eqn:tosic2}, we obtain 
$$
\Delta_\mu E =8(
\Vert\tfrac{1}{2i}\Phi\mu\Vert^2 -\Vert(\beta,\psi)\Vert^2)=-8\Vert
w_\mu\Vert^2_{h_0} \ ,
$$
which agrees with \eqref{eqn:hessian} (see also Lemma \ref{lem norm of min norm
vect}).

\bibliography{paper1bib}{}
\bibliographystyle{plain}
\end{document}

%% file: macros.tex
\newcommand{\Bbold}{{\bf B}}
\newcommand{\Cbold}{{\bf C}}

\newcommand{\Ebold}{{\bf E}}

\newcommand{\Mbold}{{\bf M}}

\newcommand{\Tbold}{{\bf T}}

\newcommand{\Wbold}{{\bf W}}
\newcommand{\Xbold}{{\bf X}}

\newcommand{\CBbb}{\mathbb C}

\newcommand{\HBbb}{\mathbb H}

\newcommand{\Bcal}{\mathcal B}
\newcommand{\Ccal}{\mathcal C}

\newcommand{\Ecal}{\mathcal E}

\newcommand{\Gcal}{\mathcal G}

\newcommand{\Jcal}{\mathcal J}

\newcommand{\Lcal}{\mathcal L}

\newcommand{\Pcal}{\mathcal P}

\newcommand{\Scal}{\mathcal S}

\newcommand{\Ucal}{\mathcal U}

\newcommand{\gfrak}{\mathfrak g}

\newcommand{\kfrak}{\mathfrak k}

\newcommand{\Ascr}{\mathscr A}
\newcommand{\Bscr}{\mathscr B}

\newcommand{\G}{\mathsf{G}}
\newcommand{\SL}{\mathsf{SL}}

\newcommand{\GL}{\mathsf{GL}}

\newcommand{\K}{\mathsf{K}}

\newcommand{\SO}{\mathsf{SO}}

\newcommand{\slfrak}{\mathfrak{sl}}

\newcommand{\sofrak}{\mathfrak{so}}

\DeclareMathOperator{\End}{End}

\DeclareMathOperator{\Hom}{Hom}

\DeclareMathOperator{\Aut}{Aut}
\DeclareMathOperator{\aut}{aut}

\DeclareMathOperator{\imag}{im}

\DeclareMathOperator{\Real}{Re}
\DeclareMathOperator{\rank}{rank}

\DeclareMathOperator{\Sym}{Sym}

\DeclareMathOperator{\coker}{coker}

\DeclareMathOperator{\ad}{ad}
\DeclareMathOperator{\Ad}{Ad}
\DeclareMathOperator{\tr}{tr}
\DeclareMathOperator{\Tr}{Tr}

\newcommand{\dbar}{\bar\partial}

\newcommand{\lra}{\longrightarrow}

\newcommand{\Mod}{{\rm Mod}}

\newcommand{\Diff}{{\rm Diff}}

\newcommand{\At}{{\rm At}}

\newcommand{\pr}{{\rm pr}}

\makeatletter
\newcommand*\bigcdot{\mathpalette\bigcdot@{.5}}
\newcommand*\bigcdot@[2]{\mathbin{\vcenter{\hbox{\scalebox{#2}{$\m@th#1\bullet$}}}}}
\makeatother

\makeatletter
\newcommand{\thickcolon}{\mathpalette\thick@colon\relax}
\newcommand{\thick@colon}[2]{%
  \mspace{1mu}%
  \vbox{%
    \hbox{$\m@th#1\bigcdot$}
    \nointerlineskip
    \kern.15ex
    \hbox{$\m@th#1\bigcdot$}
    \kern-.55ex
  }%
  \mspace{1mu}%
}
\makeatother

\makeatletter
\newcommand{\thickcirc}{\mathpalette\thick@circ\relax}
\newcommand{\thick@circ}[2]{%
  \mspace{1mu}%
  \vbox{%
    \hbox{$\m@th#1\circ$}
    \nointerlineskip
    \kern.15ex
    \hbox{$\m@th#1\circ$}
    \kern-.55ex
  }%
  \mspace{1mu}%
}
\makeatother

\makeatletter
\newcommand*\bigldot{\mathpalette\bigldot@{.5}}
\newcommand*\bigldot@[2]{\mathbin{\hbox{\scalebox{#2}{$\m@th#1\bullet$}}}}
\makeatother

\newcommand*{\dt}[1]{\overset{\bigldot}{#1}}
\newcommand*{\ddt}[1]{\overset{\bigldot\bigldot}{#1}}